\documentclass[a4paper]{amsart}
\usepackage[all]{xy}

\usepackage{amsmath,amssymb}
\usepackage{amscd,eucal,amsthm}
\usepackage{epsfig}

\newtheorem{theorem}{“еорема}[section]
\newtheorem{lemma}[theorem]{Lemma}
\newtheorem{proposition}[theorem]{Proposition}
\newtheorem{corollary}[theorem]{Corollary}
\theoremstyle{definition}
\newtheorem{definition}[theorem]{Definition}

\newtheorem{remark}[theorem]{Remark}

\newcounter{itemnumber}

\DeclareMathOperator{\supp}{supp}

 \DeclareMathOperator{\rk}{rk
\, } \DeclareMathOperator{\Pic}{Pic} 

\DeclareMathOperator{\proj}{Proj}
\DeclareMathOperator{\Ad}{Ad}

\def\pp{{\mathfrak p}}
\def\gg{{\mathfrak g}}
\def\hh{{\mathfrak h}}

\def\uu{{\mathfrak u}}
\begin{document}

\title[]{On the embeddings of universal torsors over del Pezzo surfaces in the cones over flag varieties.}

\author[V.~Zhgoon ]{Vladimir S. Zhgoon }
{\vspace {1ex}}

\begin{abstract}
Following Skorobogatov and Serganova we construct the
 embeddings of universal torsors over del Pezzo surfaces in the cones over flag varieties,
 considered as the closed orbits in the projectivization of quasiminiscule representations.
We give the approach that allows us to construct the
embeddings of universal torsors over del Pezzo surfaces of the degree one.
This approach uses the Mumford stability and slightly differs from the approach of the authors named above.
\end{abstract}

\maketitle

Let $G$  be a simply connected algebraic group over algebraically closed field $\Bbb K$
of characteristics zero,
$Z(G)$ is the center of $G$, $T$
is a maximal torus in $G$, and $B$ is the Borel subgroup containing this torus,
 $U$ is the unipotent radical of $B$. This define the root system $\Delta$ and its
  subsets of positive and negative roots $\Delta^+$ and
$\Delta^-$.

Let $\Delta$  be a root system of one of the following types: $A_4,D_5,E_6,E_7,E_8$.
Denote by $\alpha$  the root corresponding to the endpoint of the Dynkin diagram: for the diagram
$D_5$ we set $\alpha=\varepsilon_4-\varepsilon_5$,  for the systems of type  $E_*$ and a system
 $A_4$ возьмем $\alpha=\varepsilon_1-\varepsilon_2$ (here $\varepsilon_i$
is the standard basis in the vector space  generated by the weights).
The root system corresponding to the Dynkin diagram with deleted end vertex we denote by
$\Delta^{'}$. Let us denote by  $\beta$ the root corresponding to the adjacent vertex to $\alpha$  (cf. fig.1). %, а  остальные корни ---через$\gamma_i$.
 The fundamental weight dual to $\alpha$ we denote by $\pi_\alpha$. By $\pi^{'}_\beta$
 we denote the fundamental weight, from the set of fundamental weights of the root system $\Delta^{'}$, dual to $\beta$ (we have to note that $(\pi^{'}_\beta;\alpha)\neq 0$).
 The systems of simple roots for the root systems $\Delta$ and $\Delta^{'}$ we denote by $\Pi$ and  $\Pi^{'}$
correspondingly.

\begin{center}
  \epsfig{file=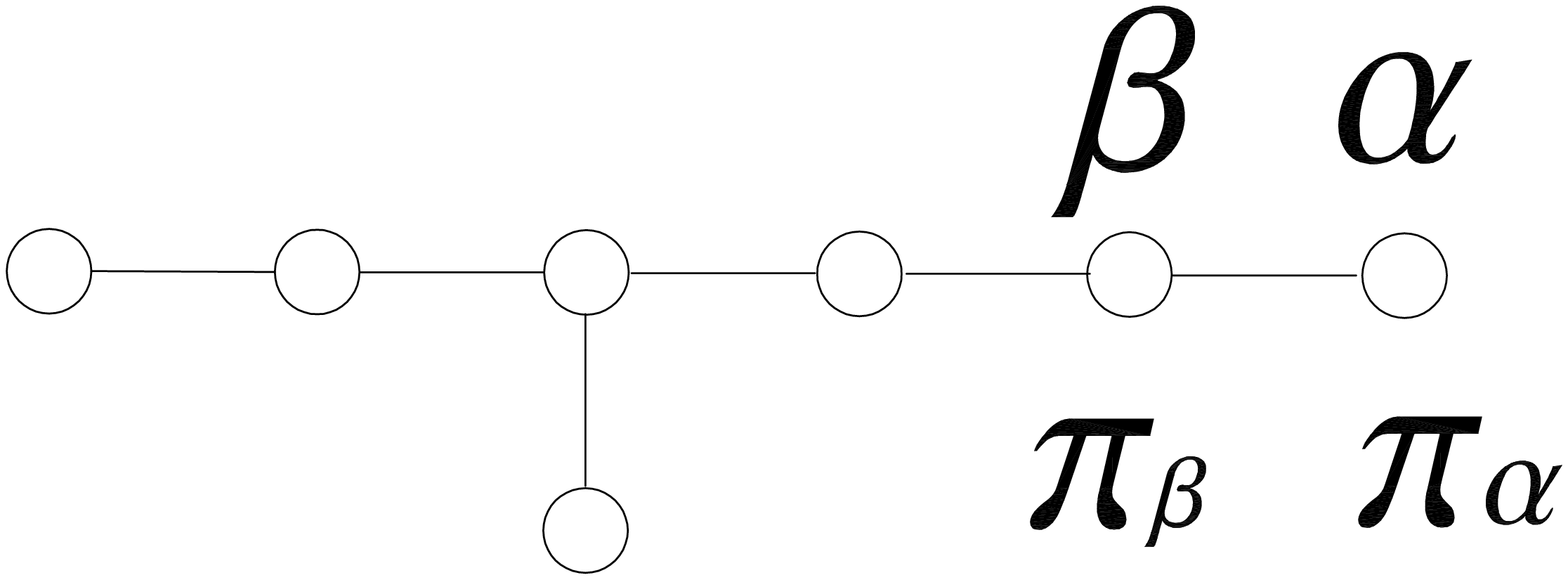,height=3cm,angle=90,clip=}

fig.1
\end{center}
{\vspace {1ex}}

Let us notice that the representation $V(\pi_\alpha)$  with the highest weight
$\pi_\alpha$
is the miniscule representation for all systems but $E_8$,  i.e. the Weyl group is
acting transitively on the weights of representation. In the case of  $E_8$
we the representation in consideration is adjoint and  the Weyl group is
acting transitively on the weights of representation distinct from zero, and a zero weight has the multiplicity $8$.

Let $P$ be a parabolic subgroup that stabilizes the point
$\langle v_{\pi_\alpha}\rangle\in \Bbb P(V(\pi_\alpha))$, where $v_{\pi_\alpha}$
 is the highest weight vector. Let
$L$ be a Levi subgroup of $P$. The semisimple part of $L$ we denote by $G^{'}$,
it has a root system
$\Delta^{'}\subset \Delta$.  The irreducible representation of
$G^{'}$ with the highest weight $\omega$ we denote by $V^{'}(\omega)$.
 The stabilizer in  $G^{'}$ of the point  $\langle v_{\pi_\beta^{'}}
 \rangle\in \Bbb P(V^{'}(\pi_\beta^{'}))$ is the parabolic subgroup $P^{'}$ with the
  Levi subgroup $L^{'}$, that has the system of simple roots  $\Pi^{''}:=\Pi\setminus \{\alpha,\beta\}$.
The semisimple part of $L^{'}$  is denoted by $G^{''}$, and its
 root system is denoted by   $\Delta^{''}$.

As it is well known from \cite{Manin},  to a del Pezzo surface one can put in the correspondence
the root system  $\Delta$  of one of the types mentioned above. Thus by   $X_{\Delta}$ we denote an arbitrary
del Pezzo surface those root system has type $\Delta$.

Let us remind the reader the definition of the universal torsor. We assume that the
 action of some torus $T_0$ on the normal variety
$\mathcal T$ is scheme theoretically free and there exists a geometric quotient $\mathcal
T\longrightarrow X$. Let $\mathcal O(\mathcal T)^{*}$  be the set of regular invertible functions on $\mathcal T$.
Then we have the following exact sequence from the work of J-L.Colliot-Th\'el\`ene et J-J.Sansuc
that we denote by {(CTS)}: $$1 \longrightarrow \mathcal O( X)^*/\Bbb
K^{*}\longrightarrow (\mathcal O(\mathcal T)^*/\Bbb K^{*})
\longrightarrow \Xi (T_0) \longrightarrow \Pic (X) \longrightarrow
\Pic(\mathcal T)\longrightarrow0,$$ where $\Xi (T_0) $ --- is the character lattice of  $T_0$.
We note that this sequence is functorial on $\mathcal T$.

\begin{definition} The torsor  $\mathcal T$ is called  {\it universal},
if the map  $\Xi (T_0) \longrightarrow \Pic (X)$ is an isomorphism.
\end{definition}

For a convenience of the reader we recall the definition of the set of semistable and stable points, introduced
by Mumford.
 \begin{definition} Let $X$ be an algebraic variety with the action of a reductive group  $H$, and $L$
 is an invertible ample
$H$-linearized sheaf on $X$.
\begin{itemize}
 \item[(i)]
 The set of semistable points is equal to  $$X_L^{ss} =\{ x\in
X :\exists n >0, \exists  \sigma \in \Gamma(X,L^{\otimes n})^H,
\  \sigma(x)\not = 0\}$$
 \item[(ii)]  The set of stable points is equal to   $$ X_L^s=\{ x \in
X_L^{ss} : \  \text{orbit} \  Hx\ \text {closed in} \  X_L^{ss}  \text{and the stabilizer} \ H_x \ \text{is finite}\}$$
\end{itemize}
\end{definition}

In our case $H=T$,   $X=G/P$,  and for the sheaf  $L$
we take $i^*\mathcal O(1)$, (where $i:G/P\subset \Bbb
P(V(\pi_\alpha))$), by $(G/P)^{sf}$ we denote the set of points  $x\in (G/P)^s$,
for which the stabilizer is equal to $T_x=Z(G)$.

 Consider the torus $T\times \Bbb K^{\times}$. Let us define the action of
 this torus on the space  $V(\pi_\alpha)$ in such way that the component
 $\Bbb K^{\times}$ acts on $V(\pi_\alpha)$
  by homotety, and the second component as the subgroup in  $G$.
 Since the representation  $V(\pi_\alpha)$ is irreducible,
  the kernel of the action of $T\times \Bbb K^{\times}$ on   $V(\pi_\alpha)$
  is isomorphic to $Z(G)$.   By $\widehat{T}$
we denote the quotient of $T\times \Bbb K^{\times}$ by this kernel.

Our aim is to prove the theorem in general case that was proved by Skorobogatov and Serganova \cite{skor}
in the case of root system $\Delta$ distinct from $E_8$.

\begin{theorem} There exists a locally  closed    $\widehat{ T}$-equivariant embedding
of the universal torsor  $\widehat{\mathcal T}$ over
del Pezzo surface $X_{\Delta}$ (in the case  $E_8$
we consider a sufficiently general surface  $X_{\Delta}$) in the affine cone in  $V(\pi_\alpha)$
over the open subset  of points  $(G/P)^{sf}$ of the flag variety   $ G/P\subset \Bbb
P(V(\pi_\alpha))$, that are stable with respect of the action of maximal torus
 $T$ and having a stabilizer $Z(G)$. Consider the intersection of
 $ \widehat{\mathcal T}$ with the $\widehat{ T}$-invariant hyperplane
 $v_{\omega}=0$ (where $\omega$
is the nonzero weight of the representation $V(\pi_\alpha)$).  Its image under the quotient by the action
 of torus $\widehat{ T}$ is the $(-1)$-curve lying on the surface
$X_{\Delta}$. All $(-1)$-curves on  $X_{\Delta}$  can be obtained in such way.
\end{theorem}
\begin{remark} The connection of universal torsors over del Pezzo surfaces  and flag varieties
embedded  in the projectivization of miniscule representation was first observed by V.V.Batyrev (see also
\cite{bat}).  The total coordinate rings of del Pezzo surfaces were calculated by V.V.Batyrev
and O.N.Popov in \cite{bat}.
For the detailed history of this question we refer the reader to  \cite{skor}.
\end{remark}

Let us give a sketch of the proof of the main theorem.
We note that the scheme of the proof was taken from
\cite{skor},  but most steps will be modified, and we shall give some proofs different from \cite{skor}.

First we prove that the Mumford quotient of the flag variety  $G/P$ by a one parameter subgroup
$\lambda:\Bbb K^{\times}\longrightarrow T$ is isomorphic to the blow up of $\Bbb
P(V(\pi_\beta^{'}))$ in the flag variety $G^{'}/P^{'}$, that we denote by
$Bl(\Bbb P(V(\pi_\beta^{'})),G^{'}/P^{'})$.  On the induction step we assume that for the
$T^{'}$-torsor $\mathcal T^{\ '}$ over
$X_{\Delta^{'}}$ we costructed its embedding into $G^{'}/P^{'}$ (it is obtained be the embedding of the universal torsor
 $\widehat{\mathcal T}^{\
'}$ in the cone over $G^{'}/P^{'}$ and a further projection on
$G^{'}/P^{'}$). Let us fix a weight basis in
$V^{'}(\pi_\beta^{'})$. Then the point $s\in \Bbb
P(V^{'}(\pi_\beta^{'}))$ whose all coordinates are nonzero defined the automorphism of  $\Bbb
P(V^{'}(\pi_\beta^{'}))$,  multiplying the coordinates of the point from $\Bbb
P(V^{'}(\pi_\beta^{'}))$ by the coordinates of the point  $s$.

Let    $e_{\Delta}$  be a point on the surface $X_{\Delta^{'}}$,
 and    $\widetilde{\sigma}:X_{\Delta} \longrightarrow X_{\Delta^{'}}$ is a blow up of this point.
Consider the point $s$ of the torsor $\mathcal T^{\ '}$, lying in the fiber over the point
 $e_{\Delta}$. Let $s^{'}\in G^{'}/P^{'}$ be a sufficiently general point.
 By applying the automorphism defined by $s^{'}s^{-1}$ that acts by multiplication of the weight
 vectors of $V^{'}(\pi_\beta^{'})$  we get that the image of the torsor $s^{'}s^{-1}\mathcal T^{\ '}$ intersect $G^{'}/P^{'}$ by
one $T^{'}$-orbit $Ts^{'}$, lying over the point  $e_{\Delta}$. Then we take a proper transform $\widetilde{\mathcal T}$ of the torsor $s^{'}s^{-1}\mathcal T^{\ '}$
under the blow up $\Bbb P(V(\pi_\beta^{'}))$ in $G^{'}/P^{'}$, and then we take its preimage under the quotient
by $\lambda$. Let us denote the obtained torsor by $\mathcal T$. We will show that the affine cone over
 $\mathcal T$ will be desired torsor.

In the case of group  $G$ of type $E_8$  the quotient  $\lambda \backslash\!\! \backslash (G/P)^{ss}$ is not
isomorphic to a blow up.
But here we can consider the quotient $\lambda \backslash\!\! \backslash (G/P)^{ss}\cap D_\alpha$ (for some $T$-invariant divisor $D_\alpha$
defined below), with the projection  $p_0$ on $\Bbb P(V_1)$. Then there exists  a $T^{'}$-invariant neighbourhood of the
point $s^{'}\in G^{'}/P^{'}\subset \Bbb P(V_1)$, such that its preimage is isomorphic to a weighted blow up in the subvariety
 $G^{'}/P^{'}$.
The local calculation of the proposition  \ref{blowupE_8} allows us to finish the proof similar to the case
 of the del Pezzo surfaces of the degree  $>1$.

\vspace{1ex} For the convenience of the reader we supply the commutative diagram illustrating the scheme of the proof.

$$
\xymatrix{
            &\mathcal T\ar@{^{(}->}[rr]\ar[d]_{\lambda \backslash\!\! \backslash}          &      &\ar[d]_{\lambda \backslash\!\! \backslash} (G/P)^{sf}  \\
            &\ar[ld]_{T^{ '} \backslash\!\! \backslash} \ar@{^{(}->}[rr]\ar[d]_{\sigma} \widetilde{\mathcal T}  &                                 &\ar[d]_{\sigma} Bl(\Bbb P(V(\pi_\beta^{'}));G^{'}/P^{'})\\
X_{\Delta}\ar[d]_{\widetilde{\sigma}}           &\ar[ld]_{T^{ '} \backslash\!\! \backslash} \ar@{^{(}->}[r] s^{'}s^{-1}\mathcal T^{\ '}   & s^{'}s^{-1}G^{'}/P^{'}\ar@{^{(}->}[r]   &  \Bbb P(V^{'}(\pi_\beta^{'})) \\
X_{\Delta^{'}}                 & \mathcal T^{\ '}   \ar@{^{(}->}[r]           & G^{'}/P^{'}\ar@{^{(}->}[ru]            & \\
}
$$

\vspace{3ex}

{\bf Let us introduce an additional notation:}

Let $W=N_G(T)/T$ be the Weyl group of $G$, $C$ is the dominant Weyl chamber for the root system $\Delta$.
 We denote by $s_\gamma\in W$ the reflection corresponding to the root $\gamma$,
Let $w\in W$, by $n_w$ we denote its representative in $N_G(T)$.
Let $H$ be a semisimple subgroup in $G$, normalized by torus $T$. We denote by $W_H$ the Weyl group of $H$,
we can assume that $W_H\subset W$.
For the Lie algebra  $\gg$ let us fix a standard basis that consists of  $e_\gamma$, where $\gamma \in \Delta$,
and $h_\vartheta$, where $\vartheta \in \Pi$.

Let $V$ be a $T$-module. By $\supp(V)$ we denote the set of weights  of the  $T$-module $V$.

For the miniscule representation  $V(\pi_\alpha)$ let us fix a basis that consists of the weight vectors $v(\omega)$,
 where $\omega \in \supp(V(\pi_\alpha))$. By  $V^*(-\pi_\alpha)$ we denote the dual module to $V(\pi_\alpha)$
  with the lowest weight $-\pi_\alpha$,  by  $\langle \cdot, \cdot\rangle$  we denote the canonical pairing,
 and by   $\{v^*(-\omega)\}$, its basis dual to $\{v(\omega)\}$.

For the birational morphism  $p:X\rightarrow Y$ and a subvariety $Z\subset Y$ we denote by $p_*^{-1}(Z)$
the proper transform of $Z$ in $X$.

\vspace{1ex}

The author is grateful to his scientific advisor I.V.Arzhantsev  for the constant attention to his work.
He wants to thank A.N.Skorobogatov and V.V.Serganova for paying his attention to the mistake in theorem \ref{E_8factor}.
I am grateful to Y.G.Prokhorov for useful discussions on the extremal contractions.
I am grateful to E.B.Vinberg, V.L.Popov and T.E.Panov for the attention to this work.

\section{The geometry of embeddings of the flag varieties in the projectivizations of miniscule
representations}

For describing  geometry of the flag varieties we shall need the following lemma
from the work of
 I.N.Berstein, I.M.Gelfand, S.I.Gelfand \cite{BGG}, describing the structure of Schubert varieties in
 this projective embedding.

{\vspace {2ex}}

\begin{lemma}\label{BGG}\cite[2.12]{BGG}  Let $w \in W$ be an element of the Weyl group,
and $B^-wP/P$ is the corresponding Schubert cell.
 Consider the closed embedding $G/P \hookrightarrow \Bbb
P(V(\chi))$ (where $\chi$ --- is the dominant weight). Let $f\in
V({\chi})$ is the vector from the orbit of the highest weight vector. Then $\langle
f\rangle \in B^-wP/P$ iff  $w\chi \in
supp(f)$ and $f \in \mathfrak U(\mathfrak b^-)v_{w\chi}$, where
$\mathfrak U(\mathfrak b^-)$ is the universal enveloping
of the Lie algebra $\mathfrak b^-$, and
$v_{w\chi}=w v_{\chi}$ is the vector of weight  ${w\chi}$ (that has multiplicity one in the representation
 $V(\chi)$). \end{lemma}

Let us study the $G^{'}$-submodules in the representation
$V(\pi_\alpha)$.

Consider the restriction of the representation
$V(\pi_\alpha)$ on the subgroup
 $G^{'}\subset P$. Let us notice that the pairing $(\pi_\alpha; \cdot )$
  defines the grading on $V(\pi_\alpha)$. Indeed, the vector $v=\sum v_{\chi}$
  belongs to the component  $V_i$
iff $( \pi_\alpha; \chi )=( \pi_\alpha; \pi_\alpha
)-i$ for any $\chi$, such that $v_\chi\neq 0$.

There is a decomposition
$$V=\langle v(\pi_\alpha) \rangle \oplus V_1 \oplus V_2 \oplus V_{\geqslant 3}.$$

The grading in consideration is $G^{'}$-invariant. To prove this, it is sufficient to check
that the action of the element $e_\tau \in
\gg^{'}$ of the standard basis  $\gg^{'}$, where $\tau \in
\Delta^{'}$, maps a weight vector $v_{\chi_0}$ to the weight vector $ v_\chi$, where $( \pi_\alpha; {\chi_0})=( \pi_\alpha; \chi
)$. Let us notice that $e_\tau v_{\chi_0}= v_{\chi_0+\tau}$, and also $(
\pi_\alpha; \tau )=0$, since $\gg^{'}$  stabilizes the line
$\langle v_{\pi_\alpha} \rangle$. This gives the required assertion.

Since $\alpha$  is the only simple root with nontrivial pairing with
 $\pi_\alpha$ it is easy to see that

$$ \supp(V_1) = ({\pi_\alpha}-\alpha-\sum_{\gamma \in \Pi^{'}} \Bbb Q_+\gamma)\cap \supp(V(\pi_\alpha)),$$
$$ \supp(V_2) = ({\pi_\alpha}-2\alpha+\sum_{\gamma \in \Delta^{'}} \Bbb Q \gamma)\cap \supp(V(\pi_\alpha)).$$

 The weight  ${\pi_\alpha}-\alpha$ is dominant for the representation
$G^{'}:V_1$ (indeed $({\pi_\alpha}-\alpha;\beta)=1$ and
$({\pi_\alpha}-\alpha;\gamma)=0$ for $\gamma \in \Pi^{''}$).
Thus $\supp V_1\subset \supp V\subset
s_\alpha(\pi_\alpha-\sum_{\gamma \in \Pi}\Bbb Q_+\gamma)$,  that
means that for  $G^{'}$-module $V_1$ we have $\supp V_1\subset
{\pi_\alpha}-\alpha-\sum_{\gamma \in \Pi^{'}} \Bbb Q_+\gamma$.

\begin{remark}\label{V_>3} The component $V_{\geqslant 3}$ is nonzero
 only for the groups of type $E_7$ and $E_8$. In the case of $E_7$ there is an
equality $V_{\geqslant 3}=V_{3}= \langle v_{-\pi_\alpha}
\rangle$. The equality $V_{\geqslant 3}= 0$ for the remaining groups follows from the fact that
 $\langle \pi_\alpha; \pi_\alpha-w_0\pi_\alpha
\rangle<3$, where $w_0\pi_\alpha$  is the lowest weight of representation
$V(\pi_\alpha)$.
\end{remark}

\begin{lemma}\label{V_1=p^u} The representation $G^{'}:V_1$ is irreducible miniscule representation.
 Let $G^{'}:\pp^-_u$ be an adjoint representation of the semisimple part of the Levi subgroup
 $L$ on the unipotent radical  $\pp^-_u$  of the parabolic subalgebra $\pp^-$. Then we have
\begin{itemize}
 \item[$\bullet$] if $\Delta$ is distinct from type $E_8$,  then we have isomophism
 of the  $G^{'}$-modules $V_1\cong V(\pi_\beta^{'}) \cong \pp^-_u$.

 \item[$\bullet$] in the case when $\Delta$ is the root system of type
 $E_8$, we have isomorphisms of the
 $G^{'}$-modules $V_1\cong V(\pi_\beta^{'})$ and $V_1\oplus \langle v_0 \rangle\cong \pp^-_u$,
  where $v_0=e_{-\pi_\alpha}v_{\pi_\alpha}$ is the vector of weight zero (In the case of $E_8$
 the weight $\pi_\alpha=\varepsilon_1-\varepsilon_9$ is a root).
\end{itemize}
\end{lemma}
\begin{proof} First let us prove that the representation $V_1$ is irreducible.
Let us recall that all weights of  $V_1$ have multiplicity one. (The weights $V_1$ are those
weights $\tau$ of the representation $V(\pi_{\alpha})$ that belong to the hyperplane
 $\langle \pi_\alpha; \pi_\alpha - \tau \rangle=1$.
But the nonzero weights of $V(\pi_{\alpha})$ have multiplicity one.) If $V_1$ is reducible, then there are
 two subrepresentations with the different highest weights $\omega_1$ and
$\omega_2$. Since the representation  $V(\pi_\alpha)$ is $G$-irreducible,
the difference of these weights is the integral linear combination of simple roots from $\Delta$.
In other words there exists a weight $\omega$ such that

$$\omega=\omega_1-n_\alpha\alpha-\sum \limits_{\gamma\in I_1\subset \Pi^{'}} n_\gamma\gamma=
\omega_2-m_\alpha\alpha-\sum \limits_{\gamma\in I_2\subset\Pi^{'}}
m_\gamma\gamma$$ where all the coefficients
$n_\alpha,n_\gamma,m_\alpha,m_\gamma$ are nonnegative and the sets
$I_1$ and $I_2$ do not intersect.

Let us notice that  $( \pi_\alpha; \omega_1 )=( \pi_\alpha; \omega_2 )=(
\pi_\alpha; \pi_\alpha)-1$, taking into account that $( \pi_\alpha;
\gamma )=0$ for $\gamma\in \Pi^{'}$ we get:

$$( \omega;\pi_\alpha )=(\omega_1;\pi_\alpha )-n_\alpha(\alpha;\pi_\alpha )=
(\omega_2;\pi_\alpha )-m_\alpha(\alpha;\pi_\alpha ).$$

Thus $n_\alpha=m_\alpha$ and we can assume that
$n_\alpha=m_\alpha=0$.

Let us show that the weight $\omega$ is dominant with respect to $G^{'}$. Let
$\gamma_0\notin I_2\subset\Pi^{'}$, then

$$( \omega;\gamma)=(\omega_2;\gamma_0 )-\sum \limits_{\gamma\in I_2\subset\Pi^{'}}
m_\gamma(\gamma;\gamma_0 )\geqslant 0,$$ since $\omega_2$
is dominant and $\langle\gamma;\gamma_0 \rangle\leqslant0$ for
$\gamma\neq\gamma_0$. The similar equality holds for
$\gamma_0\notin I_1\subset\Pi^{'}$. That is required. Since
$\omega$ is dominant and can be obtained from $\omega_1$ and $\omega_2$
by subtracting simple roots, this weight belong to the weights of both irreducible modules,
that implies that its multiplicity is at least 2. We come to a contradiction since all weights of the representation $V_1$
have multiplicity one.

To show that $V_1\cong V(\pi_\beta^{'})$ let us calculate the pairing of
 $\pi_\alpha-\alpha$ with the simple roots in $\Delta^{'}$.
As we have already seen $({\pi_\alpha}-\alpha;\beta)=1$ and
$({\pi_\alpha}-\alpha;\gamma)=0$ for $\gamma \in \Pi^{''}$, that is needed.

We notice that $\pp^-_uv(\pi_\alpha)\subset V_{\geqslant 1}$
 and that $\pp^-_uv(\pi_\alpha)$ is not contained in $
V_{\geqslant 2}$. Indeed  $\pp^-_u$ contains the root $-\alpha$, hence
 $\pp^-_uv(\pi_\alpha)$ contains
$e_{-\alpha}v(\pi_\alpha)\in V_1$.

Comparing the dimensions we get $\dim V(\pi_\beta^{'})=\dim
V_1=\dim(\pp^-_u)$. Taking into account that there exists a nontrivial projection
of $G^{'}$-module $\pp^-_u$ to $V_1$, we obtain the isomorphism of these modules in the
first case.

In the second case when  $\Delta$  is of the type $E_8$, the proof is similar.
As before all weights of  $V_1$
have multiplicity  $1$ (since the only weight with multiplicity in  $V(\pi_\alpha)$
is the zero weight but it lie in  $V_2$). We also notice that
 $v_0=e_{-\pi_\alpha}v(\pi_\alpha)=h_{\pi_\alpha}$ is
 $G^{'}$-fixed vector. The irreducibility of
$V_1$ and the following equality on dimensions
 $\dim \pp^-_u=\dim V_1 +1$ terminates the proof.
\end{proof}

\begin{corollary}\label{wtp_u} The intervals connecting the weight $\pi_\alpha$ with the weights
$\pi_\alpha-\delta$, where $\delta \in \Delta_{\pp^-_u}$ (corr.
$\delta \in \Delta_{\pp^-_u}\setminus \pi_\alpha$ in the case of type $E_8$),
are the edges of weight polytop of the representation
$V(\pi_\alpha)$.

\end{corollary} \begin{proof}
 Let us notice that the cone with the vertex   $\pi_\alpha$, spanned by  the roots $\delta \in
 \Delta_{\pp^-_u}$ coincides with  the cone spanned by the edges of polytop with the end
 in the vertex $\pi_\alpha$.
 The proposition follows from the fact that the edges in consideration
 are the edges of the cone with the vertex $\pi_\alpha$ over the convex polytop
formed by the weights $\pi_\alpha-\delta$. One has to use the fact that
they the vertices of the weight polytop of the representation $V_1$,
 that are also the vertices of the  weight polytop of the representation $V(\pi_\alpha)$.

 We omit the proof for the case of  $E_8$ since it is similar.
 \end{proof}

In the proposition  \ref{V_1=p^u}  we obtained that the miniscule representation $V_1$
of the group $G^{'}$ includes in the  quasiminiscule representation $V(\pi_\alpha)$. The next lemma
 relates the flag varieties  $G/P$ and $G^{'}/P^{'}$, that lie in the projectivizations of these
 quasiminiscule representations.

\begin{proposition} \label{intersection} Let $V(\pi_\alpha)$  be a quasiminiscule representation,
and $G/P$ be a flag variety that embeds in $\Bbb P(V(\pi_\alpha))$ as the projectivization of the orbit of
the highest weight vector. Consider the decomposition of $V(\pi_\alpha)$ in the sum of irreducible $G^{'}$-modules:
$$V(\pi_\alpha)|_{G^{'}}=\langle v(\pi_\alpha)\rangle\oplus V_1\oplus \ldots.$$

Then the intersection $G/P\cap\Bbb P(V_1)$  is isomorphic to the flag variety
 $G^{'}/P^{'}$, embedded in  $\Bbb P(V_1)$ as a
$G^{'}$-orbit of the point $\langle e_{-\alpha}v(\pi_\alpha)\rangle$.
\end{proposition}
\begin{proof} Let $\langle v \rangle\in G/P\cap\Bbb P(V_1)$. Since
  $V_1$  is a miniscule representation, without a loss of generality
  we may assume that
  $v_{\pi_\alpha-\alpha}\neq 0$. Indeed we can find a component
$v_\omega\neq 0$ for some weight $\omega$. Since the Weyl group of the root system
 $\Delta^{'}$ is acting transitively on the weights of
$V_1$, there exists an element  $w$ of this Weyl group, translating the weight
$\omega$ in  $\pi_\alpha-\alpha$. Let us note that any representative
$n_w\in G^{'}$ of the element $w$ maps $G/P$ and $\Bbb P(V_1)$ into themselves.
Besides the component of the weight $\pi_\alpha-\alpha$ of the vector $n_wv$
is nonzero (since the component  $(n_wv)_{\pi_\alpha-\alpha}$
is proportional to  $v_\omega$ with nonzero coefficient.)
Instead of the point $\langle v\rangle$ we may consider the
 line $n_w\langle v\rangle\in G/P\cap\Bbb P(V_1)$.

Let us notice that if  $v_{\pi_\alpha-\alpha}\neq 0$, then the component
$(s_{\alpha}v)_{\pi_{\alpha}}=s_\alpha v_{\pi_\alpha-\alpha}$
of the vector $s_\alpha v$ is nonzero (here we used that
 $s_\alpha \pi_\alpha=\pi_\alpha-\alpha$). By the Lemma
\ref{BGG} we get that $\langle s_\alpha v\rangle$ belongs to the open cell
 $P^-_uP/P$. The latter is equivalent to the equality $\langle
s_\alpha v\rangle=\langle u(s_{\alpha}v)_{\pi_{\alpha}}\rangle$ for $u\in P^-_u$.
Applying the exponential presentation of the element
 $u\in P^-_u$ (i.e. $u=\exp(\sum_{\gamma \in
\Delta_{\pp_u}}c_\gamma e_{-\gamma})$ where $c_\gamma \in \Bbb K$
 and $e_{-\gamma}$ is an element of the standard basis of the Lie algebra $\gg$),
we can represent $\langle v\rangle$ in the form

$$\langle v\rangle=\langle s_\alpha (\exp(\sum \limits_{\gamma \in \Delta_{\pp_u}}c_\gamma e_{-\gamma})(s_{\alpha}v)_{\pi_{\alpha}})\rangle=
\langle \exp(\sum \limits_{\gamma \in \Delta_{\pp_u}}c_\gamma
s_\alpha e_{-\gamma})v_{\pi_\alpha-\alpha}\rangle=$$
$$=\langle v_{\pi_\alpha-\alpha} +
\frac{\sum c_\gamma s_\alpha
e_{-\gamma}}{1!}v_{\pi_\alpha-\alpha}+\frac{(\sum c_\gamma
s_\alpha e_{-\gamma})^2}{2!}v_{\pi_\alpha-\alpha}+\ldots\rangle.$$

Let us show that $\langle v\rangle$ belongs to a $G^{'}$-orbit of the vector
$v_{\pi_\alpha-\alpha}$. Assume that it is not so. Then in the exponential representation
we can find  $\gamma$, such that
$s_\alpha\gamma \notin \Delta^{'}$. The weight $\pi_\alpha-\gamma$
is extremal by the Corollary  \ref{wtp_u}, in other words  $\gamma$ cannot be represented as a sum of at least
two roots from $s_\alpha\Delta_{\pp_u}$. Let us show that the vector
$v$ has nonzero component of weight
$s_\alpha(\pi_\alpha-\gamma)=\pi_\alpha-\alpha-s_\alpha\gamma$ that is equal to
 $c_\gamma e_{-s_\alpha\gamma}v_{\pi_\alpha-\alpha}$.
 From the exponential representation and extremality of the weight
 $\pi-\gamma$ we obtain that the component of weight
$s_\alpha(\pi-\gamma)$ is equal to $c_\gamma
e_{-s_\alpha\gamma}v_{\pi_\alpha-\alpha}$. Assume that this component is zero.
This will imply that the vector
$v_{\pi_\alpha-\alpha}$ is the lowest weight vector of the representation
 of the three dimensional  algebra $\mathfrak{sl}_2$, generated by the triple
$\{e_{-s_\alpha\gamma},h_{s_\alpha\gamma},e_{s_\alpha\gamma}\}$.
That is impossible by the following chain of inequalities
  $(\pi_\alpha-\alpha;s_\alpha\gamma
  )=(s_\alpha\pi_\alpha;s_\alpha\gamma)=(\pi_\alpha;\gamma)>0$.
  Where the last inequality is due to the fact that $\gamma \in \pp^-_u$ and
$P$ is the stabilizer of the line spanned by the highest weight vector with the weight $\pi_\alpha$.

Let us notice that the component of the vector $v$  of weight
$\pi_\alpha-\alpha-s_\alpha\gamma$ does not belong to $V_1$.  We come to the contradiction
since in the exponential representation we have   $\gamma$, such that $s_\alpha\gamma \notin \Delta^{'}$.
Thus we obtain that $\sum_{\gamma \in \Delta_{\pp_u}}c_\gamma
e_{-s_\alpha\gamma}\in \gg^{'}$. That means that $\exp(\sum_{\gamma \in
\Delta_{\pp_u}}c_\gamma e_{-s_\alpha\gamma})\in U\cap G^{'}$.

From the above we have the inclusion $G/P\cap\Bbb P(V_1) \subseteq
G^{'}\langle v_{\pi_\alpha-\alpha} \rangle$. And from the
$G^{'}$-invariance of the intersection we obtain the equality
 $G/P\cap\Bbb P(V_1) = G^{'}/P^{'}$.\end{proof}

\begin{proposition} Let $G$ be the group of one of the types considered above.
 And let $i:G/P\subset \Bbb
P(V(\pi_\alpha))$ be the embedding of the flag variety in the projectivization of miniscule
representation. For the action of torus $T$  consider the set of stable points
  $(G/P)^{s}_T$ with respect to the sheaf
$i^*\mathcal O(1)$ .  Then the complement to the set  $(G/P)^{s}_T$
in $G/P$ has the codimension $>1$.

The codimension of the points from $G/P$ with the stabilizer strictly containing  $Z(G)$ is strictly greater than $1$.
\end{proposition}
\begin{proof} Let us denote by $W^{st}_{\pi_\alpha}$ the set of elements
 $w\in W$, such that $\langle
w\pi_\alpha;\lambda\rangle\leqslant 0$ for every  $\lambda \in C$.
By the Theorem  1.6 \cite{zh}  the set of stable points is described by the following formula:
$$(G/P)^{s}_T=\bigcap \limits_{\tilde{w}\in W}\bigcup \limits_{w \in W^{st}_{\pi_\alpha}}\widetilde{w}BwP/P.$$
Let us note that in the Schubert decomposition for $G/P$ there exists only one cell of codimension 1,  i.e.
 $Bw_0s_\alpha P/P$. Using the previous formula we see that for the proof of proposition
it is sufficient to show that
$w_0s_\alpha\in W^{st}_{\pi_\alpha}$.

Let us set $w_0\lambda=-\sum a_\gamma \pi_\gamma \in -C$, where $a_\gamma\geqslant0$.
 The following calculation shows that $w_0s_\alpha\in W^{st}_{\pi_\alpha}$:
 $$-( w_0s_{\alpha} \pi_\alpha;\lambda)=(
\pi_{\alpha}-\alpha;-w_0\lambda)=a_\alpha((\pi_{\alpha};\pi_{\alpha})-\frac{(\alpha;\alpha)}{2})+\sum
\limits_{\gamma \neq \alpha ;\gamma \in
\Pi}a_{\gamma}(\pi_{\alpha};\pi_{\gamma})> 0,$$ where we used the fact  that in considered cases
$(\pi_{\alpha};\pi_{\alpha})-\frac{(\alpha;\alpha)}{2}>0$,  and that
 $(\pi_{\alpha};\pi_{\gamma})> 0$ for a simple root system $\Delta$ (see also \cite[Prop. 2.4]{skor}).

For the last assertion see \cite[Prop. 2.4]{skor}.
\end{proof}

\section{The quotient of $G/P$ by the one-parameter subgroup $\lambda$}

 Let us define the one-parameter subgroup $\lambda$ corresponding to the weight $\pi_\alpha$,
  by means of the pairing with the weights  $\chi \in \Xi(T)$ by the following formula
 $$\chi(\lambda(t))=t^{\langle \pi_\alpha; \chi \rangle}.$$

Consider the action  $\lambda$ on  $G/P$ by the left translations.
Let us fix a linearization of the action $\lambda$, defined by a  $G$-linearized sheaf $i^*\mathcal O(1)$,
where $i:G/P \hookrightarrow \Bbb P(V(\pi_\alpha))$ is the
$G$-equivariant embedding. We can consider the set of stable points
 $(G/P)^{s}$ with respect to the action  $\lambda$
 (from the proof of the next  theorem it will be seen that the set of stable points
 $(G/P)^{s}$
 in the case of root system distinct from $E_8$, coincide with the set $(G/P)^{ss}$ of semistable points).
 By  $p_{\pi_\alpha}: (G/P)^{s}\longrightarrow \lambda \backslash \!\!\backslash (G/P)^s$ we denote a quotient by
 the action of one-parameter subgroup  $\lambda$.

\begin{theorem}\label{lambda factor} Let  $G$  be a simple group of one of the considered types
except $E_8$. The geometric quotient
$\widetilde{X}=\lambda  \backslash \!\!\backslash (G/P)^{s}$ is equal to the blow up of
$\Bbb P(V_1)$ in the flag variety  $G^{'}/P^{'}$.
\end{theorem}
\begin{proof}
We shall need a following lemma.

\begin{lemma}\label{p_oexist} There exist a projection  $p_0$ from the considered
quotient  $\widetilde{X}$ to $\Bbb
P(V_1)$.
\end{lemma}
\begin{proof}
 Let $\langle v\rangle \in (G/P)^{ss}$. To the point $\langle v\rangle $
 we assign the point $\langle v_1\rangle \in
\Bbb P(V_1)$, where $v_1$ is the projection of $v$ to $V_1$ along the subspace $\Bbb Kv(\pi_\alpha)\oplus V_{\geqslant2}$.

 Let us notice that this map factors through $\widetilde{X}$.
This is a corollary of universal property of quotient $\lambda$
and the fact that action on $\Bbb
P(V_1)$ is trivial (the action of $\lambda(t)$  multiplies the vectors from
$V_1$ by $t^{\langle \pi_\alpha; \pi_\alpha-\alpha
\rangle}=t^{\langle \pi_\alpha; \pi_\alpha\rangle-1}$).

 It is left to check that $p_0$ is well defined on $(G/P)^{ss}$.
 To prove this let us notice that for the groups of considered type $G$ (
 except $E_8$) the linear subspace $\langle v_{\pi_\alpha} \rangle \oplus V_1$
has positive weights with respect to the action of $\lambda$, and the subspaces  $ V_2 \oplus
V_{\geqslant 3}$ has negative weights ( in the case of  $E_8$ the subspace $V_2$
has zero weight and  $V_{\geqslant 3}$ has negative).
Thus for the stability of $\langle v\rangle$ it is necessary that the component
of  $v$ in the subspace  $\langle v_{\pi_\alpha} \rangle
\oplus V_1$ is not equal to zero. Let the map $p_0$ is not defined in $\langle v\rangle$,
i.e. $v_1=0$. Then since
$\langle v\rangle\in(G/P)^{ss}$ we have $v_{\pi_\alpha}\neq0$.
Thus we obtain that the vector  $v$  belongs to the open cell
$P^-_uP/P$. Let us represent $v$ as the exponential map  from the element of the Lie algebra $\pp^{-}_u$,
applied to the vector $v_{\pi_\alpha}$:
$$v=\exp(\sum \limits_{\gamma \in \Delta_{P_u}}c_\gamma e_{-\gamma})v_{\pi_\alpha}=v_{\pi_\alpha} +
\underbrace{\frac{\sum c_\gamma
e_{-\gamma}}{1!}v_{\pi_\alpha}}_{\in V_1}+\underbrace{\frac{(\sum
c_\gamma e_{-\gamma})^2}{2!}v_{\pi_\alpha}}_{\in V_2}+\ldots.$$

Since $v_1=0$
% компонента $v$, лежаща€ в $V_1$, нулева€,
all the coefficients $c_\gamma$  are equal to zero (
the vectors $e_{-\gamma}v_{\pi_\alpha}$ for $\gamma\in \Delta_{\pp^-_u}$ have different weights and form
 the basis  $\pp^-_u v_{\pi_\alpha}\cong V_1$). Thus
$\langle v\rangle=\langle v_{\pi_\alpha}\rangle$,  but the latter point is unstable.
\end{proof}

We split the proof of the theorem into a few steps. First we construct the quotient of stable
orbits lying in the open cell,
then for the orbits lying in the complement of the open cell
(Steps 1,2). Then using the projection $p_0$ onto the quotient
$\widetilde{X}$ and the Moishezon contraction theorem \cite{Moishezon},
we obtain that these sets glue together in a blow up (Step 3).

{\bf Step 1.}  Consider the open cell $P^-_uP/P$. Let us show that the quotient
 $\lambda \backslash \!\!\backslash  (P^-_uP/P)$ is isomorphic to the projective
space $\Bbb P(V_1)$ with the deleted flag variety $G^{'}/P^{'}$.

As before let us use an exponential representation $v$:
$$v=\exp(\sum \limits_{\gamma \in \Delta_{P_u}}c_\gamma e_{-\gamma})v_{\pi_\alpha}=v_{\pi_\alpha} +
\underbrace{\frac{\sum c_\gamma
e_{-\gamma}}{1!}v_{\pi_\alpha}}_{\in V_1}+\underbrace{\frac{(\sum
c_\gamma e_{-\gamma})^2}{2!}v_{\pi_\alpha}}_{\in V_2}+\ldots.$$

Let us notice that the  set of coefficients $\{c_\gamma\}_{\gamma \in \Delta_{P_u}}$ considered up to a multiplication by a constant с
defines the orbit of
$\lambda$, thus the quotient $\lambda
\backslash \!\!\backslash  ((P^-_uP/P)\setminus \langle v_{\pi_\alpha} \rangle) $,
is identified with $\Bbb P(V_1)$. To obtain that $\lambda
\backslash  \!\!\backslash (P^-_uP/P)^{ss}$
we have to delete from $\Bbb P(V_1)$ those points
 $(c_{\gamma_1}:\ldots:c_{\gamma_{\dim V_1}})$ for which the corresponding orbits
  $\lambda(t) \langle v \rangle $ are unstable.

The orbit of the point $\langle v\rangle$ from the open cell is unstable iff
 $v$ belongs to $\Bbb K
v(\pi_\alpha) \oplus V_1$. Let us prove the following lemma:

\begin{lemma}\label{unst} Let $v\in\Bbb K
v(\pi_\alpha) \oplus V_1$. Then we have an inequality: $$\lim \limits_{t\rightarrow 0}\lambda(t)(G/P\cap \Bbb P(\Bbb K
v(\pi_\alpha) \oplus V_1))=(G/P)\cap \Bbb P(V_1)=G^{'}/P^{'}.$$

\end{lemma}
\begin{proof} Applying Proposition \ref{intersection} we see that we need to prove the following
limit:

$$\lim \limits_{t\rightarrow 0}\lambda(t)\langle
v\rangle=\lim \limits_{t\rightarrow 0} \langle t^{\langle
\pi_\alpha; \pi_\alpha \rangle}v_{\pi_\alpha} + \frac{\sum
c_\gamma t^{\langle \pi_\alpha; \pi_\alpha-\gamma
\rangle}e_{-\gamma}}{1!}v_{\pi_\alpha}\rangle=$$
$$=\lim
\limits_{t\rightarrow 0} \langle tv_{\pi_\alpha} + \frac{\sum
c_\gamma e_{-\gamma}}{1!}v_{\pi_\alpha}\rangle=\langle\frac{\sum
c_\gamma e_{-\gamma}}{1!}v_{\pi_\alpha}\rangle \in \Bbb P(V_1).
$$
Here we used that $\langle \pi_\alpha; \gamma
\rangle=1$ for all $\gamma \in \Delta_{P^u}$. We also notice that $\langle
v_1\rangle=\lim \limits_{t\rightarrow 0}\lambda(t)\langle
v\rangle$.
\end{proof}

By Lemma \ref{unst} we get that for $\langle v\rangle \notin (P^{-}_uP/P)^{ss}$
we have the inclusion $\langle v_1\rangle\in G^{'}/P^{'}$.
Let us show that this condition is also sufficient.
Consider the vector $v=\exp(e_{-\alpha})v_{\pi_\alpha}=v_{\pi_\alpha}+e_{-\alpha}v_{\pi_\alpha}$, it is easy to see that
  $\langle v\rangle \notin (G/P)^{ss}$. Since  $G^{'}$ and $\lambda$ commute
  the orbit $G^{'}\langle v\rangle$ is unstable. Besides the projection is equal to
    $p_0(G^{'}\langle v\rangle)=G^{'}\langle v_1\rangle = G^{'}/P^{'}\subset\Bbb P(V_1)$.

From what was said above and from isomorphism
 $\lambda\backslash \!\!\backslash((P^{-}_uP/P)\setminus \langle v_{\pi_\alpha}\rangle)\cong \Bbb P(V_1)$ we get
that unstable points project by means of $p_0$  exactly onto $G^{'}/P^{'}$.
 Thus we proved that  $\lambda\backslash \!\!\backslash(P^{-}_uP/P)^{ss}\cong\Bbb P(V_1)\setminus (G^{'}/P^{'})$.

{\bf Step 2.} The space of stable orbits from the complement to $P^-_uP/P$
is isomorphic to the projectivization of the conormal bundle in  $\Bbb P(V_1)$
to the flag variety $G^{'}/P^{'}$.

From Lemma \ref{BGG} and \ref{p_oexist},  we get that the stable orbits from the complement of
to $P^-_uP/P$ satisfy the following condition:

\begin{itemize}
 \item[$\bullet$] The component of vector $v$ of weight $\pi_\alpha$ is equal to zero (since $\langle v\rangle \notin P^-_uP/P$).
 \item[$\bullet$] The component of vector $v$ that belong to
 $V_1$ is nonzero (by the stability of $\langle v\rangle $).
\end{itemize}

In particular these orbits belong to the closure of the Schubert divisor $\overline{B^{-}s_\alpha P/P}$. Indeed % первое условие
%эквивалентно тому, что рассматриваемые орбиты лежат в дополнении к
%открытой клетке. ”тверждение следует из того, что
in the в Bruhat decomposition of $G/P$ there is a unique cell of codimension
one $\overline{B^{-}s_\alpha P/P}$; and it contains in its closure all cells of smaller dimension.

\begin{lemma}\label{component_V_1} Let $\langle v\rangle \in \Bbb
P(V(\pi_\alpha))$ be a stable point such that the component of vector
 $v$ of the weight $\pi_\alpha$  is nonzero.
 Then the component  $v_1$  of vector $v$ that belong to
 $V_1$ is distinct from zero and $\langle v_1\rangle \in \Bbb
P(V_1)\cap G/P$.
\end{lemma}
\begin{proof}
Consider the decomposition $v=\sum_{i\geqslant 1} v_i$ for $v_i \in V_i$,
besides we  are given $v_1 \neq 0$. Then we have an equality:
$$\lim_{t \rightarrow \infty} \langle \lambda(t)v\rangle=
\lim_{t \rightarrow \infty} \langle v_1+ \sum_{i> 1}
t^{-i+1}v_i\rangle=\langle v_1\rangle.$$ Since $\lim_{t
\rightarrow \infty} \langle \lambda(t)v\rangle \in G/P$ this proves the lemma.
\end{proof}

Let us notice that the action of $G^{'}$ commutes with $\lambda(t)$, and
the intersection $\Bbb P(V_1)\cap G/P=G^{'}/P^{'}$ is a single orbit
by Lemma \ref{intersection}. By Lemma \ref{component_V_1}
for the component $v_1\in V_1$ of vector $v$ we have the inclusion
$\langle v_1\rangle \in G^{'}/P^{'}$. Acting by an element from the group  $G^{'}$
we may assume that  $v_1=v_{\pi_\alpha-\alpha}$. Let us represent an element  $v$ in terms of exponential map в
 (from the element of $\uu$) applied to the vector
$v_{\pi_\alpha-\alpha}$.
$$v=v_{\pi_\alpha-\alpha} +
\frac{\sum c_\gamma
e_{-\gamma}}{1!}v_{\pi_\alpha-\alpha}+\frac{(\sum c_\gamma
 e_{-\gamma})^2}{2!}v_{\pi_\alpha-\alpha}+\ldots.$$

By the assumptions we have $\pi_\alpha-\alpha-\gamma \notin \supp(V_1)$ that implies that the coefficients $c_\gamma$
 can be nonzero only for the roots of type $\gamma=m\alpha+k\beta+\sum_{\theta \in
\Pi^{''}} a_\theta\theta$ where $a_\theta\geqslant 0$ only when $m,k>0$.

Since $P^{'}$ stabilizes $\langle v_{\pi_\alpha-\alpha}\rangle$ the vectors with the weights
 $\pi_\alpha-\alpha+\gamma$ form a $P^{'}$-module
$\mathcal N$. (The vector subspace
 $\mathcal N$ is isomorphic to the module
$\pp^-_uv_{\pi_\alpha-\alpha}$. The structure of the $P^{'}$-module  $\pp^-_uv_{\pi_\alpha-\alpha}$
is defined as follows. For $p^{'}\in
P^{'}$ taking into account the equality $p^{'}v_{\pi_\alpha-\alpha}=v_{\pi_\alpha-\alpha}$ we get
$p^{'}\pp^-_uv_{\pi_\alpha-\alpha}=(\Ad(p^{'})\pp^-_u)(p^{'}v_{\pi_\alpha-\alpha})\subset \pp^-_uv_{\pi_\alpha-\alpha},$
 where $\Ad(P^{'}):\pp^{-}_u$ is an adjoint representation.)

\vspace{1ex}

Let us prove the following lemma on the structure of the $P^{'}$-module $\mathcal N$.

\begin{lemma}\label{cotang} Let $\gamma\in \Delta_{\pp_u}$ be a positive root such that $\pi_\alpha-\alpha+\gamma \in \supp(\mathcal
N)$. Consider a decomposition $\gamma=m\alpha+k\beta+\sum_{\theta \in
\Pi^{''}} a_\theta\theta$ of the root $\gamma$ in the sum of simple roots.
  Then in this decomposition we have  $k=2$, $m=1$. We have the inclusion
   $\mathcal N \subset V_2$. And the equality
   $(\alpha;\gamma)=0$ take place.

    The module $\mathcal N$ is a simple $P^{'}$-module with the trivial $P^{'}_u$-action.
      As $G^{''}\ltimes P^{'}_u$-module it is isomorphic to the fiber of the conormal bundle
      to the flag variety $G^{'}/P^{'}\subset \Bbb P(V_1)$  in the point $\langle
v(\pi_\alpha-\alpha) \rangle$.
\end{lemma}
\begin{proof} Let us denote by  $V^{'}_1$ and $V^{'}_2$ the irreducible $L^{'}$-submodules in $V_1$ that are the graded components
of the weights
$\langle \pi_\beta^{'}; \pi_\beta^{'}\rangle -1$ and $\langle \pi_\beta^{'}; \pi_\beta^{'}\rangle -2$ correspondingly with respect to
the pairing with $\pi_\beta^{'}$.

First let us notice that as  $L^{'}$-module the fiber of cotangent bundle to $G^{'}/P^{'}$
in the point $\langle v(\pi_\alpha-\alpha)\rangle$ can be identified with the module
 $\pp^{'-}_u v(\pi_\alpha-\alpha)$, that is isomorphic to $V^{'}_1$ by Lemma \ref{V_1=p^u}.
 Thus the fiber of conormal bundle in the point $\langle
v(\pi_\alpha-\alpha)\rangle$ is identified with factor module $V_1/(\Bbb K v(\pi_\alpha-\alpha)\oplus
V^{'}_1)$. By the Remark \ref{V_>3}, we have isomorphism of $L^{'}$-modules $V^{'}_2\cong V_1/(\Bbb K v(\pi_\alpha-\alpha)\oplus
V^{'}_1)$. Since the fiber of conormal bundle as
$L^{'}$-module is isomorphic to a simple module $V^{'}_2$, $P_u^{'}$ is acting on it trivially and $Z(L^{'})$ is acting
 by the multiplication with the scalar.
 Thus for the proof of the last part of the lemma it is sufficient to check the isomorphism of  $G^{''}$-modules
 $\mathcal N$ and $V_2^{'}$.

Let the coefficient  $c_\gamma$ is not equal to zero for the root
$\gamma=m\alpha+k\beta+\sum \limits_{ \theta\in
\Pi^{''}}a_\theta\theta$ (where $ a_\theta>0$). Since the weight
 $s_\alpha\pi_\alpha-\gamma$ belongs to the weight polytop,
 the weight  $\pi_\alpha-s_\alpha\gamma$ also lie in this polytop.
Thus we obtain that
$s_\alpha\gamma \in \Delta_{\pp_u}$. Using the equality
$s_\alpha\beta=\beta+\alpha$ and the fact that  $s_\alpha$ fixes \ $\theta\in \Pi^{''}$ we obtain

$$s_\alpha\gamma =-m\alpha+k(\beta-\alpha)+\sum \limits_{
\theta\in \Pi^{''}}a_\theta\theta=k\beta+(k-m)\alpha+\sum
\limits_{ \theta\in \Pi^{''}}a_\theta\theta$$

Since $s_\alpha\gamma \in \Delta_{\pp_u}$ from lemma
\ref{V_1=p^u} it follows that the coefficient of  $\alpha$  in the decomposition of
 $s_\alpha\gamma$ in the sum of simple roots is equal to $1$ that implies
 $k-m=1$. Since $V_1$  is miniscule representation,
for the root systems of types  $A_4,D_5$ or $E_6$ then the coefficient by
$\beta$ in the decomposition of $s_\alpha\gamma$ is not greater than $2$ (this follows from the fact that
the vector of weight  $s_\alpha\gamma-\alpha$ belong to the graded component with the weight not less than
$(\pi_\beta^{'};\pi_\beta^{'})-2$ with respect to the pairing
$(\pi_\beta^{'};\cdot)$, see Remark \ref{V_>3}). Thus we get $m=1$ and
$k=2$. Since the coefficient of $\beta$ in this decomposition is equal to
$2$ the root $s_\alpha\gamma$ belongs to $V_2^{'}$. In other words we get that the element
 $s_\alpha$ maps the module  $\mathcal N$
in $V_2^{'}$. Since the element  $s_\alpha$ is acting trivially on $G^{''}$ by the conjugations
it gives the map of the $G^{''}$-module $\mathcal
N$ into the irreducible $G^{''}$-module $V_2^{'}$. Since
$s_\alpha^2=1$ this map gives the isomorphism of these modules. That is required.

We are finished by providing the following calculation
$$(\alpha;\gamma)=(\alpha;\alpha+2\beta+\sum \limits_{ \theta\in
\Pi^{''}}a_\theta\theta)=(\alpha;\alpha)+2(\alpha;\beta)=0,$$ where
we used the equalities  $(\alpha;\beta)=-1$,
$(\alpha;\theta)=0$ for $\theta\in \Pi^{''}$.

Since $L^{'}$ differs from $G^{''}$ by the central torus from the isomorphism of
 $G^{''}$-modules $V_2^{'}$ and $\mathcal N$ we get the isomorphism of $L^{'}$-varieties $\Bbb P(V_2^{'})$ and $ \Bbb
P(\mathcal N)$.
\end{proof}

\begin{proposition} The orbits $\langle\lambda(t)v\rangle$ of the points $\langle v \rangle\in (G/P)^{ss}$ with the condition
$v_{\pi_\alpha}=0$ are parameterized by the variety $G^{'}*_{P^{'}}\Bbb P(\mathcal N)$. That
is isomorphic to the projectivization of conormal bundle to $G^{'}/P^{'}$.
\end{proposition}
\begin{proof}
It is easy to see that the orbits
$\langle\lambda(t)v\rangle$ with the conditions  $v_{\pi_\alpha}=0$ and $v_1=v_{\pi_\alpha-\alpha}$
are parameterized by the projective space $\Bbb P(\mathcal N)$.
Since $\mathcal N \subset V_2$ the one-parameter subgroup
 $\lambda(t)$ is acting on  $\mathcal
N$ by the multiplication with  $t^{-2}$, and on the vector  $v_{\pi_\alpha-\alpha}$ by the multiplication with $t^{-1}$.
Thus the map
$$\ \ \  \langle v\rangle \rightarrow
\left( \frac{\langle v,v^*(-\pi_\alpha+\alpha+\gamma_1)\rangle}{\langle v,v^*(-\pi_\alpha+\alpha)\rangle}:\ldots:\frac{\langle v,v^*(-\pi_\alpha+\alpha+\gamma_{\dim \mathcal N})\rangle}{\langle v,v^*(-\pi_\alpha+\alpha)\rangle}\right)
$$
identifies the considered space of $\lambda(t)$-orbits  and $\Bbb P(\mathcal N)$.
(Let us notice that $\frac{\langle v,v^*(-\pi_\alpha+\alpha+\gamma)\rangle}{\langle v,v^*(-\pi_\alpha+\alpha)\rangle}$
is the coefficient  $c_\gamma$ by the vector $e_{-\gamma}v_{\pi_\alpha-\alpha}$ in the exponential representation for $\langle v\rangle$.)
% коэффициенты $c_\gamma$,
%задающие орбиту, определены с точностью до пропорционалности.
By the Lemma \ref{cotang} $\Bbb P(\mathcal N)$  is isomorphic to the projectivisation of the fiber
of the conormal bundle in the point $\langle v_{\pi_\alpha-\alpha}\rangle$.

By Lemma \ref{component_V_1} the variety of orbits  $\langle\lambda(t)v\rangle$ such that
$v_{\pi_\alpha}=0$ projects surjectively by means of  $p_0$ onto the flag variety $G^{'}/P^{'}$.
 Since the projection $p_0$ is $G^{'}$-equivariant
we obtain that $G^{'}*_{P^{'}}p_0^{-1}(\langle v_{\pi_\alpha-\alpha}\rangle)\cong G^{'}*_{P^{'}}\Bbb P(\mathcal N)$
(where we used the isomorphism of  $P^{'}$-varieties
$p_0^{-1}(\langle v_{\pi_\alpha-\alpha}\rangle)\cong \Bbb P(\mathcal N)$ from the paragraph).

The homogeneous bundle $G^{'}*_{P^{'}}\Bbb P(\mathcal N)$ is isomorphic to the projectivization of  the conormal bundle to
 $G^{'}/P^{'}$.
Indeed by Lemma  \ref{cotang} the fibers over $eP^{'}/P^{'}$ of the considered bundles
  are isomorphic as  $P^{'}$-variaties. By the
$G^{'}$-equvariance we get the isomorphism of the projective bundles.
\end{proof}

{\bf Step  3.} Denote by $D_{\mathcal N}\subset \widetilde{X}$
the divisor corresponding to the subvariety  $G^{'}*_{P^{'}}\Bbb P(\mathcal
N)$ in $\widetilde{X}$ and by  $\mathcal L_{\mathcal N}$ the
line bundle corresponding to this divisor.

To finish the proof of the theorem  \ref{lambda factor} we shall
use Moishezon contraction theorem (\cite {Moishezon}).

To apply it we need to proof the following proposition.

\begin{proposition}\label{D|D} For every point $x\in D_{\mathcal N}$ the restriction of the line bundle
 $\mathcal L_{\mathcal N}$ on the fiber  $p^{-1}_0(p_0(x))$
(that is the fiber of the projectivization of the normal bundle
${\mathcal N}$ to the flag variety $G^{'}/P^{'}$) is isomorphic to the line bundle
$\mathcal O(-1)$.
\end{proposition}

\begin{proof}
First let us notice that we can assume that the line bundle $\mathcal L_{\mathcal N}$ is
$G^{'}$-linearized. The divisor $D_{\mathcal N}$ is invariant with respect to the action of
  $G^{'}$. The projection $p_0$ is  $G^{'}$-equivariant and the latter group is acting transitively on
$p_0(D_{\mathcal N})\simeq G^{'}/P^{'}$. From the above
it is sufficient to check the condition of the Moishezon contraction theorem
only for one point $x\in D_{\mathcal N}$.

Let us describe the line bundle $\mathcal L_{\mathcal N}$. Let us recall that the preimage
of the divisor  $D_{\mathcal N}$  by the quotient morphism is the divisor that is equal to the intersection
of the Schubert divisor $\overline{B^{-}s_\alpha P/P}$ and $(G/P)^{ss}$. The line bundle
corresponding to this divisor can be described as $\mathcal L=G*_P
k_{\pi_\alpha}$ (where  $k_{\pi_\alpha}$ --- is one dimensional module where $P$
is acting by multiplication with the character $\pi_\alpha$).
Let us notice that the section of the line bundle $\mathcal L_{\mathcal N}$ can
be considered as the function on $V_{\pi_\alpha}$ (since the linear system corresponding
to  $\mathcal L_{\mathcal N}$ defines an embedding $G/P\subset \Bbb P(V_{\pi_\alpha})$). The section $\mathcal
L_{\mathcal N}$ those zero set is equal to $\overline{B^{-}s_\alpha
P/P}$ is described by the equation $\langle v^*_{-\pi_\alpha};.\rangle=0$,
it is semiinvariant with respect to the action of $B^-$  with the weight $-\pi_\alpha$ (cf. for example \cite{zh}).

We want to take the decent of the sheaf  $\mathcal L_{\mathcal N}$ to the sheaf
$\widetilde{\mathcal L}_{\mathcal N}$  on the quotient of the set
$(G/P)^{ss}$ by the one-parameter subgroup  $\lambda$  in such way that
the Weyl divisor defined by the section $\langle
v^*_{-\pi_\alpha};.\rangle=0$ after the quotient morphism
$p_{\pi_\alpha}$ maps in the section of the line bundle $\widetilde{\mathcal
L}_{\mathcal N}$ (cf. \cite{kn}).

We can obtain this by making the section $\langle
v^*_{-\pi_\alpha};.\rangle=0$ invariant with respect to $\lambda$.
This can be achieved  by taking  instead of the line bundle $\mathcal L_{\mathcal N}$ the
line bundle $\mathcal L_{\mathcal N}\otimes
k_{-\pi_\alpha}$ isomorphic to the previous one with the action of $\lambda$ twisted by the
character  $-\pi_\alpha$. The image of the divisor $\overline{B^{-}s_\alpha
P/P}$ on the quotient  $\widetilde{X}$ is defined by the section $\langle
v^*_{-\pi_\alpha};.\rangle=0$ of the line bundle
$\lambda \backslash \!\!\backslash(\mathcal L_{\mathcal N}\otimes
k_{-\pi_\alpha})$.

 For the fiber of  $p_0$ on which we want to restrict the line bundle
 $-\mathcal L_{\mathcal N}$ we chose the fiber over the point
$v_{\pi_\alpha-\alpha}$. Then the fiber over this point is defined by
 $\exp(\sum \limits_{\pi_\alpha-\alpha-\gamma \in \mathcal
N}c_{\gamma}e_{-\gamma})s_\alpha P/P$. The restriction of the line bundle
$-\mathcal L_{\mathcal N}\otimes k_{\pi_\alpha}$  to this fiber is defined by the formula
$$(-\mathcal L_{\mathcal N}\otimes k_{\pi_\alpha})|_{D_\mathcal N}=\exp(\sum \limits_{\pi_\alpha-\alpha-\gamma \in \mathcal
N}c_{\gamma}e_{-\gamma})s_\alpha P*_P (k_{-\pi_\alpha}\otimes
k_{\pi_\alpha}).$$

This is the line bundle over the affine space $\mathcal
N$ and its quotient by  $\lambda$ is the line bundle over the projective space
 $\Bbb P(\mathcal N)$. To prove that the latter line bundle is isomorphic to
 $\mathcal O(1)$ one can check that when we multiply the coordinates $\{c_\gamma\}$ by $t$ the fiber
of the line bundle is multiplied by  $t^{-1}$. This assertion follows from the line of equalities:
$$\lambda(t)^{-1}(-\mathcal L_{\mathcal N}\otimes k_{\pi_\alpha})|_{D_\mathcal N}=$$

$=(\lambda(t)^{-1}\exp(\sum \limits_{\pi_\alpha-\alpha-\gamma \in
\mathcal N}c_{\gamma}e_{-\gamma})\lambda(t)\lambda(t)^{-1}s_\alpha
P*_P k_{-\pi_\alpha})\otimes
\lambda(t)^{-1}k_{\pi_\alpha})=(\exp(\sum
\limits_{\pi_\alpha-\alpha-\gamma \in \mathcal
N}c_{\gamma}\lambda(t)^{-1}e_{-\gamma}\lambda(t))s_\alpha
(s_\alpha\lambda)(t)^{-1}P*_P k_{-\pi_\alpha})\otimes
\lambda(t)^{-1}k_{\pi_\alpha})=(\exp(\sum
\limits_{\pi_\alpha-\alpha-\gamma \in \mathcal
N}c_{\gamma}t^{\langle \pi,\gamma\rangle}e_{-\gamma})s_\alpha P*_P
t^{-\langle
\pi_\alpha-\alpha,-\pi_\alpha\rangle}k_{-\pi_\alpha})\otimes
t^{-\langle
\pi_\alpha,\pi_\alpha\rangle}k_{\pi_\alpha})=(\exp(\sum
\limits_{\pi_\alpha-\alpha-\gamma \in \mathcal
N}c_{\gamma}te_{-\gamma})s_\alpha P*_P t^{\langle
\pi_\alpha-\alpha,\pi_\alpha\rangle}k_{-\pi_\alpha})\otimes
t^{-\langle
\pi_\alpha,\pi_\alpha\rangle}k_{\pi_\alpha})=(\exp(\sum
\limits_{\pi_\alpha-\alpha-\gamma \in \mathcal
N}c_{\gamma}te_{-\gamma})s_\alpha P*_P
t^{-1}k_{-\pi_\alpha})\otimes k_{\pi_\alpha},$

\noindent where we used the equalities $\langle
\alpha,\pi_\alpha\rangle=1$,
 $\lambda(t)^{-1}e_{-\gamma}\lambda(t)=t^{\langle
\pi,\gamma\rangle}e_{-\gamma}$ and also
$(s_\alpha\lambda)(t)=s_\alpha\lambda(t)s_\alpha=t^{\langle
s_\alpha \pi_\alpha,. \rangle}=t^{\langle \pi_\alpha-\alpha,.\rangle}$.
\end{proof}

We apply the Moishezon contraction theorem to the morphism
$p_0:\widetilde{X}\longrightarrow \Bbb P(V_1)$ and exceptional divisor
 $D_{\mathcal N}$ that contracts onto the flag variety
$G^{'}/P^{'}$. We obtain that  $\widetilde{X}$ is the blow up of
$\Bbb P(V_1)$ in the flag variety $G^{'}/P^{'}$. That finishes the proof of the theorem.
\end{proof}

\begin{remark}{\label{equation}} From the proof of the previous theorem we can get the ь
description of $G^{'}/P^{'}\subset \Bbb P(V_1)$ as the set of zeros of the homogeneous system of equations
 (generating the ideal $\mathcal
J_{G^{'}/P^{'}}$). First let us define the open cell  $P^-_uP/P$ from
$G/P$ as the exponential map of the Lie algebra $\pp^{-}_u$ applied to
$v_{\pi_\alpha}$:
$$v=\exp(\sum \limits_{\gamma \in \Delta_{P_u}}c_\gamma e_{-\gamma})v_{\pi_\alpha}=v_{\pi_\alpha} +
\underbrace{\frac{\sum c_\gamma
e_{-\gamma}}{1!}v_{\pi_\alpha}}_{\in V_1}+\underbrace{\frac{(\sum
c_\gamma e_{-\gamma})^2}{2!}v_{\pi_\alpha}}_{\in V_2}+\ldots.$$

As we observed the set of nonstable points with respect to
$\lambda(t)$ is defined by the condition $V_2=0$ that implies that the component in
the subspace $V_{\geqslant3}$ also should be zero. These conditions can be rewritten as the system of equations on $\{c_\gamma\}$:

$$p_\mu(c)=\sum \limits_{\pi-\gamma-\delta=\mu}c_{\gamma}c_{\delta}(e_{-\gamma}e_{-\delta}v_{\pi_\alpha})=0,$$
for $\mu \in Supp(V_2)$. Let us notice that the weights of all quadratic polynomials are different.
 I.e. the component  $\mathcal
J_{G^{'}/P^{'}}(\mu)$  of the weight $\mu$ in the ideal $\mathcal
J_{G^{'}/P^{'}}$ has the dimension $1$.

In the Lemma \ref{unst} it was shown that
$$\lim \limits_{t\rightarrow 0}\lambda(t)(P^-_uP/P\cap \Bbb P(
\Bbb Kv(\pi_\alpha)
 \oplus V_1))=G/P\cap \Bbb P(V_1)=G^{'}/P^{'}.$$ Since the action of $\lambda(t)$
 do not change the set of homogeneous coordinates
$\{c_\gamma\}$ the flag variety  $G^{'}/P^{'}$  in the projective space
 $\Bbb P(V_1)$ is defined by the system of homogeneous equations
$p_{\mu}(c)=0$ (где $\mu \in Supp(V_2)$).

 Later we shall need the equations corresponding to the vanishing of the components lying in $V_{3}$:
$$q_\nu(c)=\sum \limits_{\pi_\alpha-\mu_1-\mu_2-\mu_3=\nu}c_{\mu_1}c_{\mu_2}c_{\mu_3}(e_{-\mu_1}e_{-\mu_2}e_{-\mu_3}v_{\pi_\alpha})=0,$$
 that vanish on the flag variety  $G^{'}/P^{'}$. In the case of $E_7$ it is a single equation for which
  $\nu=-\pi_\alpha$.
\end{remark}

\vspace{2ex}

\section{The quotient by the subgroup $\lambda$ of the flag variety $G/P$ from $ \Bbb P(\gg)$ in tha case of type $E_8$}

Let us study the case when the root system  $\Delta$ has the type
$E_8$.  Instead of miniscule representations $E_8$ we should consider the case of adjoint representation
 $Ad:\gg$. Our first task is to study a $P^{'}$-module $\mathcal N=\langle
\pp^-_uv(\pi_\alpha-\alpha)\rangle$ (cf. fig.2а).

\begin{center}
  \epsfig{file=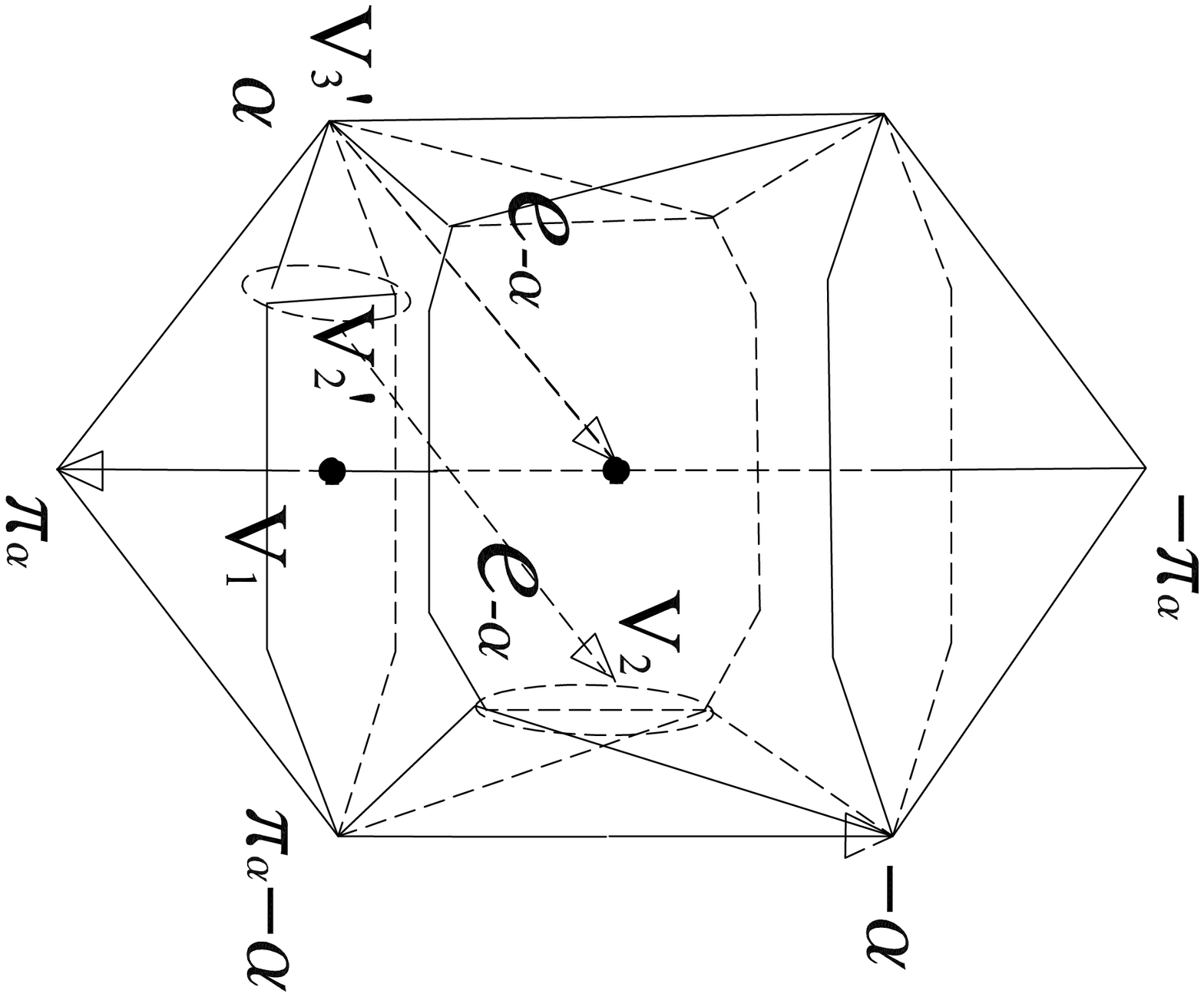,height=5cm,angle=90,clip=} \epsfig{file=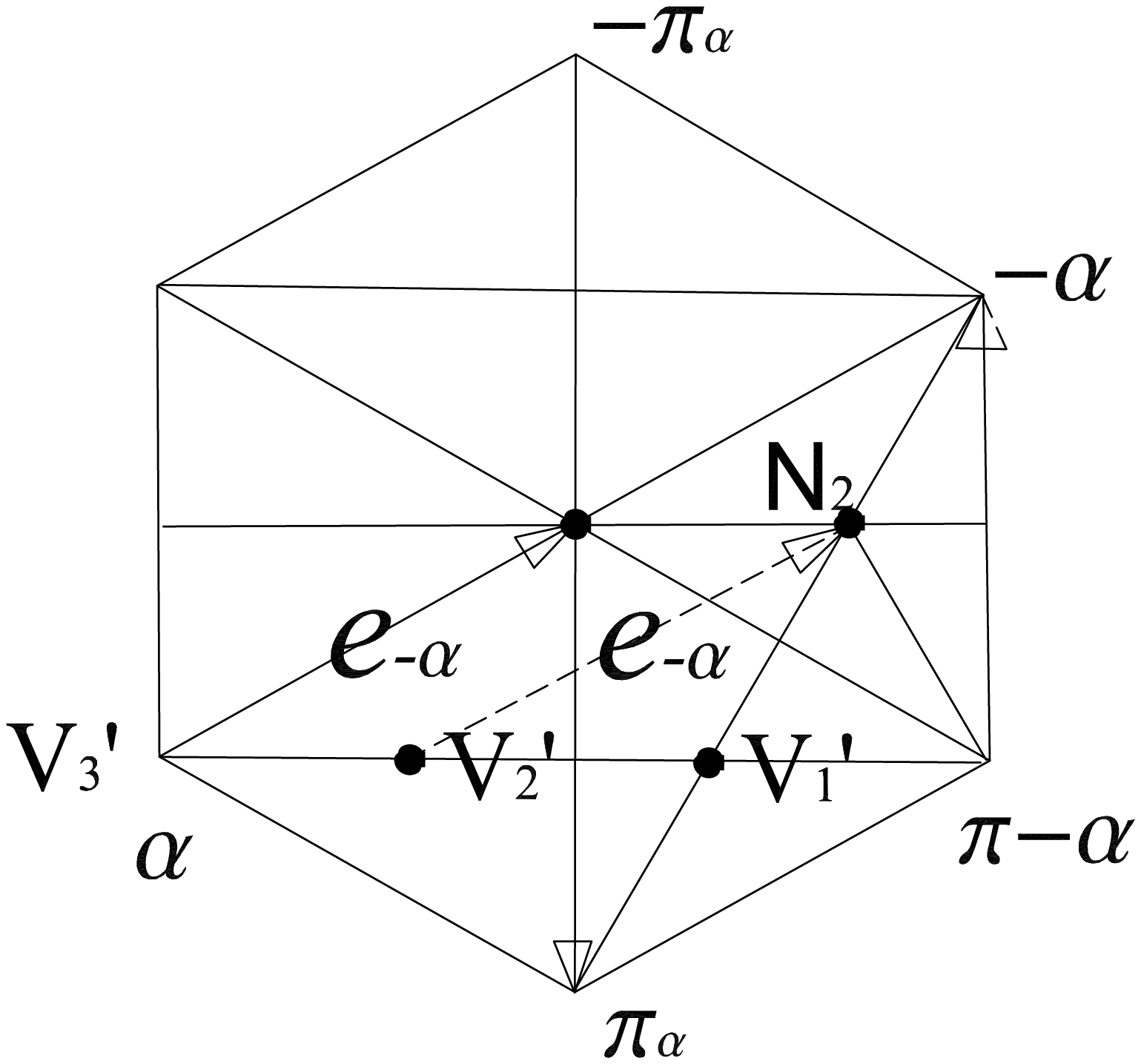,height=5cm,angle=90,clip=}

fig.2 a) \ b)
\end{center}
{\vspace {1ex}} The module $\mathcal N$ consists of the weights of type
$s_\alpha\pi_\alpha-\gamma=s_\alpha\pi_\alpha-m\alpha-k\beta-\sum \limits_{
\theta\in \Pi^{''}} a_\theta\theta$ where $m,k>0$ and
$a_\theta\geqslant 0$. Our aim to define the values of $m,k$ and to find decomposition of
 $\mathcal N$ in the irreducible $G^{''}$-submodules.
 We shall need the decomposition  of $V_1$ in the sum of simple $G^{''}$-modules:
$$V_1=\Bbb Kv(\pi_\alpha-\alpha)\oplus V^{'}_1\oplus V^{'}_2\oplus V^{'}_3.$$

On the figure 2b) we have the orthogonal projection of the weight polytop on the plane
generated by $\alpha$ and $\pi_\alpha$
(since $\pi_\beta^{'}=\pi_\alpha/2-\alpha$, $\supp(V^{'}_i)$ projects in the point).

\begin{proposition} Let  $\gamma\in \Delta_{\pp_u}$ be a positive root
such that $\pi_\alpha-\alpha+\gamma \in \supp(\mathcal
N)$. Consider the decomposition  $\gamma=m\alpha+k\beta+\sum_{\theta \in
\Pi^{''}} a_\theta\theta$ of the root $\gamma$ in the sum of simple roots,
 where $a_\theta\geqslant 0, \ m,k>0$.
  Then we have the following possibilities for the coefficients
  $m,k$:
\begin{itemize}
 \item[$\bullet$] The vector with the weights $\pi_\alpha-\alpha+\gamma$ with $m=1,k=2$,
 belong to $V_2$ and
 form a simple $G^{''}$-module isomorphic to $V^{'}_2$. In this case we have an equality
  $(\alpha;\gamma)=0$.
 \item[$\bullet$] Let $m=1,k=3$.
 This condition is satisfied by the unique vector $h_{\pi_\alpha-\alpha}$ with zero weight.
 There is an isomorphism of  $G^{''}$-modules $\mathcal N \cap
V_2\cong \Bbb Kh_{\pi_\alpha-\alpha}\oplus V^{'}_2$.
 \item[$\bullet$]  Let $m=2,k=3$.
 This condition is satisfied by a unique vector $v(-\alpha)\in V_3$.
\end{itemize}

The projective space  $\Bbb P(\mathcal N \cap V_2)$
 is isomorphic as $P^{'}$-variety to the projectivization of the fiber of normal bundle
 to the flag variety $G^{'}/P^{'}\subset \Bbb P(V_1)$  in the point $\langle
v(\pi_\alpha-\alpha) \rangle$.
\end{proposition}
\begin{proof} Let us represent the weights of the module $\mathcal N$ in the form
$s_\alpha\pi_\alpha-\gamma=s_\alpha\pi-m\alpha-k\beta-\sum
\limits_{ \theta\in \Pi^{''}} a_\theta\theta$ for $m,k>0$.
We apply to these weights the reflection $s_\alpha$. The weight
$s_\alpha\pi_\alpha$ maps to the $\pi_\alpha$. Let us notice that the weights
 $\pi-s_\alpha\gamma$ belong to the weight polytop. Since
 $s_\alpha\gamma$ is a root by Lemma \ref{V_1=p^u} we get $s_\alpha\gamma\in \pp_u$, that is equivalent
to $s_\alpha\pi_\alpha-\gamma \in V_1\oplus \Bbb
Ke_{-\pi_\alpha}v(\pi_\alpha)$. In the adjoint representation we have:
$e_{-\pi_\alpha}v(\pi_\alpha)=[e_{-\pi_\alpha};e_{\pi_\alpha}]=h_{\pi_\alpha}\in
\hh$.

We have the equality
$$s_\alpha\gamma=-m\alpha+k\alpha+k\beta+\sum \limits_{ \theta\in
\Pi^{''}} a_\theta\theta.$$

Let us notice that $k\leqslant 3$ since the miniscule representation
$G^{'}:V_1$ has the grading not greater than $3$ with respect to the pairing with
 $\pi_\beta^{'}$ (the pairing with $\pi_\beta^{'}$ is not less than $\langle
\pi_\beta^{'};\pi_\beta^{'}\rangle-3$). Let us describe the vectors $\gamma$ and the weights of
the module  $\mathcal N$. Consider the following cases:

 Let $\pi_\alpha-s_\alpha\gamma=0$ ($k-m=2$). From this we get that the weight
 $(\pi_\alpha-\alpha)-\gamma$ is equal to zero.  The corresponding vector is equal to
 $$e_{\alpha-\pi_\alpha}(e_{-\alpha}v(\pi_\alpha))=e_{\alpha-\pi_\alpha}v(\pi_\alpha-\alpha)
 =[e_{\alpha-\pi_\alpha};e_{\pi_\alpha-\alpha}]=-h_{\pi_\alpha-\alpha}.$$

  Let $\pi_\alpha-s_\alpha\gamma\in V_1$ (in this case
 $m-k=1$). Then we can decompose the representation of the group $G^{'}$ on
 $V_1$ in the graded components with the grading  defined by the pairing with
  $\pi_\beta^{'}$:
 $$V_1=\Bbb Kv(\pi_\alpha-\alpha)\oplus V^{'}_1\oplus V^{'}_2\oplus V^{'}_3.$$

 Let $k=3, \ m=2$. Then the weight space in consideration belong to
$V^{'}_3$, that is one dimensional and spanned by the lowest weight vector
$v_\tau$ of the representation $G^{'}:V_1$. Since $G^{'}:V_1$
is self dual (the root system $E_7$ does not have  automorphisms that are not inner)
$\pi_\alpha-\alpha\equiv -\tau \mod \Bbb Q\pi$ thus we have $\tau=\alpha$. The corresponding weight vector is
$e_{-\pi_\alpha}v(\pi_\alpha-\alpha)=v(-\alpha)$.

If $k=2, \ m=1$ then $\pi_\alpha-s_\alpha \gamma \in V^{'}_2$. In this case the
element $s_\alpha$ gives an isomorphism of the simple $G^{''}$-module
$V_2^{'}$ and a submodule in $\mathcal N$. In this case we have an equality $(\gamma;\alpha)=0$.
That follows from
$(\gamma;\alpha)=(\alpha+2\beta+\sum \limits_{ \theta\in \Pi^{''}}
a_\theta\theta;\alpha)=(\alpha;\alpha)+2(\beta;\alpha)=0$.

Let us prove the following proposition.

\begin{proposition}\label{cotangE_8} The projectivizations of factor module $V_1/(V^{'}_1\oplus \Bbb Kv(\pi_\alpha-\alpha))$
and module $\mathcal N\cap V_2$ are isomorphic as  $P^{'}$-varieties.
The first module is identified with the fiber of normal bundle to the flag variety $G^{'}/P^{'}\subset \Bbb P(V_1)$ in the point
$\langle v(\pi_\alpha-\alpha)\rangle$.
\end{proposition}
\begin{remark} The modules $\mathcal N \cap V_2$ and $V_1/(V^{'}_1\oplus \Bbb Kv(\pi_\alpha-\alpha))$
are isomorphic as $G^{''}\ltimes P^{'}_u$ modules but not isomorphic as
$P^{'}$-modules.
\end{remark}
\begin{proof} First let us notice that the module $V_1/(V^{'}_1\oplus \Bbb Kv(\pi_\alpha-\alpha))$
 can be identified with the following  $L^{'}$ module $\widetilde{V}=V_3^{'}\oplus V_2^{'}$
with the following action of $e_\beta$ ($P^{'}$ is generated by $\exp(e_\beta)$ and
 $L^{'}$):

\begin{itemize}
 \item[$\bullet$] On the component  $V_2^{'}$ the action of $e_\beta$ is trivial.
 \item[$\bullet$] On $V_3^{'}$ the action of $e_\beta$ comes from the action on $V_1$.
\end{itemize}
\begin{remark}To prove that $P^{'}$ is generated by $\exp(e_\beta)$ and
 $L^{'}$ it is sufficient to notice that the linear span of the elements
 $L^{'}e_\beta$ generates the Lie algebra $\pp^{'}_u$. The last assertion
 is due to the fact that on the unipotent radical  $\pp^{'}_u$ the action of  $L^{'}$
 is irreducible.
\end{remark}

From the above we have a decomposition $\mathcal N\cap
V_2:=\Bbb Kh_{\pi_\alpha-\alpha}\oplus \mathcal N_2$  as the sum of irreducible
 $G^{''}$-modules; where the module  $\mathcal N_2$
is generated by the vectors $e_{\gamma}v(\pi_\alpha-\alpha)$ for the roots
$\gamma \in \pp_u$ such that $e_{s_\alpha\gamma} v(\pi_\alpha)
\in V_2^{'}$.

Our aim is to show that the element $e_{-\alpha}$ maps
$\widetilde{V}$ in the module $\Bbb K(h_{\pi_\alpha-\alpha}+v_0)\oplus
\mathcal N_2$, where $v_0$  is some vector of weight zero. We shall show that
$e_\beta v_0=0$ that will imply that
$G^{''}\ltimes P^{'}_u$-module $\Bbb
K(h_{\pi_\alpha-\alpha}+v_0)\oplus \mathcal N_2$ is isomorphic to
$\mathcal N\cap V_2$.

Let us prove that $e_{-\alpha}$ maps the $G^{''}$-module $V^{'}_2$ in
 $\mathcal N_2$. Indeed the weights of  $V^{'}_2$ are of the form
$\pi_\alpha-(\alpha+2\beta+\sum \limits_{ \theta\in \Pi^{''}}
a_\theta\theta)$ and the map into the set of weights of $\mathcal N_2$
by the action of $s_\alpha$. The weights from the latter module can be written in the form
$\pi_\alpha-(2\beta+\sum \limits_{ \theta\in \Pi^{''}}
a_\theta\theta)$. This implies that they are obtained from the weights of
$V^{'}_2$ by subtracting $\alpha$. Let us notice that the action $e_{-\alpha}$
commutes with $G^{''}$ (the root  $\alpha$ is orthogonal to the roots from
$\Delta^{''}$ corresponding to the group $G^{''}$). Since $G^{''}$-module
 $V^{'}_2$ is simple,  for the proof of the isomorphism it is sufficient to check that
  $e_{-\alpha}V^{'}_2\neq 0$.
Let us choose a weight vector $v_\chi \in V^{'}_2$. The weights $\chi$ and
$\chi-\alpha$ are the ends of the edge of a weight polytop
(see Corrolary  \ref{wtp_u}), that implies that the vector $e_{-\alpha}v_\chi$ is nonzero.

Next let us notice that $e_{-\alpha}$ maps the weight vector from
$V^{'}_3$ to the vector $e_{-\alpha}v_\alpha=h_\alpha$. We have
$[e_\beta;h_\alpha]=(\beta;\alpha)e_\beta=-(\beta;\pi_\alpha-\alpha)e_\beta=
-[e_\beta;h_{\pi_\alpha-\alpha}]$ since $(\beta;\pi_\alpha)=0$. Thus we obtain that
$e_{-\alpha}v_\alpha=h_\alpha=-h_{\pi_\alpha-\alpha}+v_0$
for some vector $v_0$, annihilated by $e_{\beta}$.

At last we are left to show that the morphism
$e_{\alpha}:\widetilde{V}\longrightarrow \Bbb
Ch_{\pi_\alpha-\alpha}\oplus \mathcal N_2$ is equivariant with respect to the action of
 $e_{\beta}$. Since $[e_{\beta};e_{-\alpha}]\in
\gg_{\beta-\alpha}=0$ ( this holds since the difference of the simple roots is not a root)
we get a chain of equalities:
$$e_{\beta} e_{-\alpha}V^{'}_2=e_{-\alpha} e_{\beta} V^{'}_2+[e_{\beta};e_{-\alpha}]V^{'}_2=e_{-\alpha} e_{\beta} V^{'}_2,$$
 that is required.

Let us notice that the element  $h_{\pi^{'}_\beta}$ and the algebra $\gg^{''}+
\pp^{'}_u$ generate $\pp^{'}$. Thus for the proof it is sufficient to check that
 $h_{\pi^{'}_\beta}$
act in the same manner on $\Bbb P(\mathcal N \cap V_2)$ and $\Bbb
P(V_1/(V^{'}_1\oplus \Bbb Kv(\pi_\alpha-\alpha)))$. But this is a corollary of the equality
 $$\supp(\mathcal N \cap
V_2)=\supp(V_1/(V^{'}_1\oplus \Bbb Kv(\pi_\alpha-\alpha)))~+~\alpha.$$

 \end{proof}

This proof of the lemma finishes the proof of the proposition.\end{proof}

\vspace{1ex}

{\bf Let us introduce some additional notation.}

\vspace{1ex}

Consider the set of stable points $(G/P)^s_{\varepsilon}$
for the action of $\lambda$ (corr. $T$)
With respect to the very ample bundle $\mathcal M=\mathcal
O(2)|_{G/P}\otimes k_{-\pi_{\alpha}}$. As a sheaf it is isomorphic to
$\mathcal O(2)|_{G/P}$ but the action of $\lambda$  (corr. $T$)
is twisted by a character $-\pi_{\alpha}$.
After such linearization we cam assume that $\lambda$ is acting on the component
 $V_i$ with the weight $t^{3-2i}$; and the weights of the subspace $\langle v_{\pi_\alpha}\rangle \oplus V_1$
are positive and the weights of $V_{\geqslant2}$ are negative.

%„тобы сформулировать теорему \ref{E_8factor} нам понадобитс€  следующа€ естественна€ конструкци€.

\vspace{3ex}

{\bf Weighted blow up and weighted projective space.}

\vspace{1ex}

Here we recall the definition of the weighted blow up adapted to our situation.
Let $\Bbb P(V)\subset \Bbb P(\Bbb Kw\oplus V)$ be a natural embedding of a projective space as a subspace $w=0$.
Let us define on $\Bbb P(\Bbb Kw\oplus V)$ the involution $\imath$ acting by the formula
 $\imath(w)=-w$ and $\imath(v)=v$ for $v\in V$.
Consider the quotient $\Bbb P(\Bbb Kw\oplus V)/\imath$ of the subspace
$\Bbb P(\Bbb Kw\oplus V)$  by the action of involution $\imath$.
This quotient is the weighted projective space with the weights $(2,1,\ldots,1)$ (cf. \cite{dolgWeight}).
The quotient map  $$p_\imath:\Bbb P(Kw\oplus V)\rightarrow \Bbb P_{wt}(Kw\oplus V)$$ can be written down in homogenous
coordinates as $$p_\imath: (z_0:z_1:...:z_n)\longrightarrow (z_0^2:z_1:...:z_n).$$
 On the variety  $\Bbb P(Kw\oplus V)$ the ramification locus of $p$ that is equal to the set of fixed points
 with respect to the involution $\imath$ consists of the divisor
   $\Bbb P(V)=(0:z_1:\ldots:z_n)$ and of the isolated  ramification point $(1:0:\ldots:0)$. The points of the divisor
  $p((0:z_1:\ldots:z_n))\subset \Bbb P_{wt}(Kw\oplus V)$ are smooth points of $\Bbb P_{wt}(Kw\oplus V)$,
  and a  point $p((1:0:\ldots:0))$ is an isolated singular point.  Let us note that the
  weighted projective space
   $\Bbb P(\Bbb Kw\oplus V)/\imath$ is the cone over quadratic Veronese map of the projective space $\Bbb P(V)$
    (cf. \cite{dolgWeight}). This map can be written in the weighted homogeneous coordinates as
  $$\Bbb P(Kw\oplus V)\rightarrow\Bbb P_{wt}(Kw\oplus V)\rightarrow \Bbb P(Kw\oplus V^{\otimes 2})$$
  $$(z_0:z_1:...:z_n)\rightarrow (z^2_0:z_1:...:z_n)\rightarrow (z^2_0:z_1^2:\ldots:z_jz_i:\ldots)_{0<j\leq i}$$

Consider $V=V_0\oplus V_1$ where $(z_1,\ldots,z_k)$ are the
coordinates in $V_0$ and  $(z_{k+1},\ldots,z_n)$   $V_1$. Then we
can consider a weighted blow up of $\Bbb Kw\oplus V$ in the subspace
 $V_1$ with the weight $2$ on the subspace $\Bbb Kw$ and weight
 $1$ on the subspace $ V_0$. It can be described as a subvariety of
  $V\times \Bbb P_{wt}(\Bbb Kv\oplus V_1)$ (where
$P_{wt}(\Bbb Kv\oplus V_0)$ is the weighted projective space with the weights
$(2,1,\ldots,1)$ and the coordinates
$(\xi_0:\xi_1:\ldots:\xi_k)$) by the following system of equations:
  $$ \left \{ \begin
{array}{cl}
 z_i\xi_j=z_j\xi_i    \\
 z_0\xi_j^2=z_j^2\xi_0     \\
\end{array}
\right.$$

On the invariant language it is a weighted $\proj_{wt}(\bigoplus \mathcal J^n)$ where $\mathcal J=(z_0,z_1,\ldots,z_k)$ is the
 ideal defining $V_1$. Here we assume
 that the coordinate  $z_0$ has weight $2$, and the latter coordinates have weight $1$.

\vspace{2ex}

Our main task is to study the quotient $\lambda \backslash \!\!\backslash(G/P)^{ss}_\varepsilon$
for the case of the simple group of type $E_8$.
The latter set probably  has more complicated  description than the corresponding set from the Proposition
\ref{lambda factor}, but for construction of the embedding
of $\mathcal T$ it will be sufficient to apply the results about this set
obtained below.

%\begin{theorem}
%ѕусть $Bl(\Bbb P(\Bbb Kw\oplus V(\pi_\alpha)), G^{'}/P^{'})$ раздутие
%$\Bbb P(\Bbb Kw\oplus V(\pi_\alpha))$ в многообразии флагов $G/P$, а $\widetilde{\imath}$
%определенна€ выше инволюци€.
%‘актор $\widetilde{X}$ подмножества $(G/P)^s_{\varepsilon} \subset \Bbb
%P(V(\pi_\alpha))$ по однопараметрической подгруппе $\lambda$
%отождествл€етс€ с фактором раздути€  $Bl(\Bbb P(\Bbb Kw\oplus V(\pi_\alpha)), G^{'}/P^{'})$
%по инволюции $\widetilde{\imath}$.
%\end{theorem}

 The arguments repeat in general the proof of Theorem \ref{lambda factor}.
  We shall give their modifications.

\begin{proposition}\label{E_8factor}
Let $\Bbb P_{wt}:=\Bbb P_{wt}(\langle
e_{-\pi_\alpha}v_{\pi_\alpha}\rangle \oplus V_1)$ be a weighted projective
space with the weights
$(2,\underbrace{1,\ldots,1}_{\dim V_1})$ obtained as a  quotient of the vector space
 $\langle e_{-\pi_\alpha}v_{\pi_\alpha}\rangle \oplus
V_1$ by the action of one-parameter subgoup $\lambda$. The flag variety
$G^{'}/P^{'}$ embedes in $\Bbb P_{wt}$ asa composition of natural incusions $G^{'}/P^{'}\subset \Bbb P(V_1)\subset \Bbb P_{wt}$.

The quotient  $\lambda\backslash \!\! \backslash(P^-_uP/P)^s_{\varepsilon}$ is isomorphic to
$\Bbb P_{wt}\setminus (G^{'}/P^{'})$ the subset of $\Bbb P_{wt}$ with deleted flag variety $G^{'}/P^{'}$.
\end{proposition}
\begin{proof}
 Consider the orbit $\langle\lambda(t)v\rangle \subset P^-_uP/P$. %, принадлежащую открытой клетке $P^-_uP/P$.
 Let us represent the vector  $v$ as the exponential map of the Lie algebra
$\pp^{-}_u$ applied to the vector $v_{\pi_\alpha}$:

$$v=\exp(\sum \limits_{\gamma \in \Delta_{P_u}}c_\gamma e_{-\gamma})v_{\pi_\alpha}=v_{\pi_\alpha} +
\underbrace{\frac{\sum_{\gamma \in \Delta_{P_u}\setminus
\{\pi_\alpha\}}c_\gamma e_{-\gamma}}{1!}v_{\pi_\alpha}}_{\in
V_1}+$$
$$+\underbrace{c_{\pi_\alpha}e_{-\pi_\alpha}v_{\pi_\alpha}+\frac{(\sum_{\gamma
\in \Delta_{P_u}\setminus \{\pi_\alpha\}} c_\gamma
e_{-\gamma})^2}{2!}v_{\pi_\alpha}}_{\in V_2}+\ldots.$$

(In other words we consider the embedding of the open cell
$P^-_uP/P$ in the space $\Bbb P (V(\pi_\alpha))$ as an orbit of the point $\langle v_{\pi_\alpha}\rangle$.)

 Consider the expression for the component  $v_0$ of the weight zero for the vector
$v$.

$$v_0=c_{\pi_\alpha}e_{-\pi_\alpha}v_{\pi_\alpha}+ \frac{1}{2}\sum
\limits_{\gamma\in \Delta_{P_u}\setminus \{\pi_\alpha\}}
c_{\gamma}c_{\pi_{\alpha}-\gamma}
(e_{-\gamma}e_{-\pi_{\alpha}+\gamma}+e_{-\pi_{\alpha}+\gamma}e_{-\gamma})v_{\pi_\alpha}$$

Using the inclusion $\mathfrak {sl}_3 \subset \gg$ defined by the positive roots
$\{\gamma,\pi_{\alpha}-\gamma,\pi_{\alpha}\}$ we can show by the explicit calculations for $\mathfrak {sl}_3$ that
the expession $(e_{-\gamma}e_{-\pi_{\alpha}+\gamma}+e_{-\pi_{\alpha}+\gamma}e_{-\gamma})v_{\pi_\alpha}$
is proportional to
$(h_{-\gamma}-h_{-\pi_{\alpha}+\gamma})=h_{\pi_{\alpha}-2\gamma}.$
Since $(\pi_{\alpha}-2\gamma;\pi_\alpha)=0$ we get that when  $c_{\pi_\alpha}=0$
the vector  $v_0$ satisfies the equality
$(v_0;h_{\pi_\alpha})=0$. That implies that orthogonal  (with respect to Cartan-Killing form)
projection of  $v$ to the component
$h_{\pi_\alpha}\in V_2$ is equal
$c_{\pi_\alpha}e_{-\pi_\alpha}v_{\pi_\alpha}$.

Let us notice that
$$\langle \lambda(t)v\rangle=\langle v_{\pi_\alpha} +
{\sum \limits_{\gamma \in \Delta_{P_u}\setminus
\{\pi_\alpha\}}tc_\gamma e_{-\gamma}}v_{\pi_\alpha}
+t^2c_{\pi_\alpha}e_{-\pi_\alpha}v_{\pi_\alpha}+\frac{1}{2!}{(\sum
\limits_{\gamma \in \Delta_{P_u}\setminus \{\pi_\alpha\}}
tc_\gamma e_{-\gamma})^2}v_{\pi_\alpha}+\ldots\rangle.$$

It is easy to see that the set of coefficients
$(t^2c_{\pi_\alpha},tc_{\gamma_1},\ldots,tc_{\gamma_{\dim V_1}})$,
where ${\gamma_i\in \Delta_{P_u}\setminus \{\pi_\alpha\}}$ defines
the orbit of  $\lambda(t)$ from the open cell.
Considering $\lambda$-equivariant  map $(P^-_uP/P)\setminus \langle v(\pi_\alpha) \rangle\rightarrow \Bbb P_{wt}$ :
$$\ \ \ \langle  v\rangle \rightarrow \left (\frac{(v,h_{\pi_\alpha})}{(v,v(-\pi_\alpha))}:
\frac{(v,v(-\pi_\alpha+\gamma_1))}{(v,v(-\pi_\alpha))}:\ldots:\frac{(v,v(-\pi_\alpha+\gamma_{\dim V_1}))}{(v,v(-\pi_\alpha))}\right)
$$
we get that the isomorphism of the considered space of orbits to the weighted
projective space
$\Bbb P_{wt}(\langle h_{\pi_\alpha} \rangle \oplus V_1)$.

\vspace{1ex}

It remains to settle which points of $\Bbb P_{wt}$ correspond to the unstable orbits and should be excluded.
The orbit $\lambda\langle v\rangle$ is unstable only in the case when the component of vector $v$
belonging to $V_{\geqslant 2}$ is equal to zero. In particular the components
$c_{\pi_\alpha}e_{-\pi_\alpha}v_{\pi_\alpha}$ и ${(\sum
\limits_{\gamma \in \Delta_{P_u}\setminus \{\pi_\alpha\}} c_\gamma
e_{-\gamma})^2}v_{\pi_\alpha}$ that are orthogonal with
respect to the Cartan-Killing form should be zero as well.
  Since $v\in \Bbb K v_{\pi_\alpha}\oplus V_1$,
 By Lemma \ref{unst} we get
 $\lim_{t\rightarrow 0}\lambda(t)\langle v\rangle=\langle v_1\rangle\in \Bbb P(V_1)\cap G/P=G^{'}/P^{'}$,
(where $v_1$ --- is the projection of $v$ on $V_1$).

Thus  for the inclusion $\langle v\rangle \notin (G/P)^{ss} $ it is nessesary that
 $c_{\pi_\alpha}=0$ and $\langle v_1\rangle\in G^{'}/P^{'}$.
This condition is also sufficient.
Let $v=\exp(e_{-\alpha})v_{\pi_\alpha}=v_{\pi_\alpha}+e_{-\alpha}v_{\pi_\alpha}$ it is easy to see that
 $\langle v\rangle \notin (G/P)^{ss}$. In the same time since $G^{'}$ and $\lambda$ commutes,
 the orbit $G^{'}\langle v\rangle$  is unstable. In this case the projection of  $G^{'}\langle v\rangle$ on $\Bbb P(V_1)$
is equal to $G^{'}\langle v_1\rangle\cong G^{'}/P^{'}$. Since the map $\lambda \backslash \!\!\backslash ((P^-_uP/P)\setminus \langle v_{\pi_\alpha} \rangle)\rightarrow \Bbb P_{wt}$ is an isomorphism we obtain that the unstable orbits
  project exactly to $G^{'}/P^{'}\subset \Bbb P_{wt}$ i.e.
   $\lambda \backslash \!\!\backslash (P^-_uP/P)^{ss}=\Bbb P_{wt}\setminus (G^{'}/P^{'})$.
\end{proof}

 Let us describe the set of  $\lambda$-orbits that do not belong to the open cell.
 Applying Lemma \ref{component_V_1} and \ref{intersection} we obtain that for the component
  $v_1\in V_1$ and vector $v$ we have the inclusion
 $\langle v_1 \rangle \in G^{'}/P^{'}$.   Acting by $G^{'}$
we can assume that $v_1=v_{\pi_\alpha}$ (since the action of $G^{'}$ commute with  $\lambda$).
 We also obtain that the image of orthogonal projection on  $\Bbb P(V_1)$ %отображении $p_0$
 of the set of $\lambda$-orbits that do not belong to the open cell is equal to
 $G^{'}/P^{'}$.

\begin{lemma} Consider the space of $\lambda$-orbits of the points  $\langle v \rangle$such that the projection of $v$
on the subspace $\langle v_{\pi_\alpha}\rangle \oplus V_1$
is proportional to $v_{\pi_\alpha-\alpha}$. This orbit space as a  $G^{''}$-variety is identified with the quotient of the module
 $\langle
e_{-\pi_\alpha}v(\pi_\alpha-\alpha)\rangle\oplus\mathcal N\cap
V_2$ by the one-parameter subgroup $\lambda$ where $\lambda$ is acting on the first component
with the weight $2$ and on $\mathcal N\cap V_2$ with the weight
$1$.
\end{lemma}

\begin{proof}
Let us represent  $v$ as the exponential map from the element of the Lie algebra
 Ћи $\uu$), applied to $v_{\pi_\alpha-\alpha}$. We get that

$$v=v_{\pi_\alpha-\alpha} +
\sum c_\gamma
e_{-\gamma}v_{\pi_\alpha-\alpha}+c_{\pi_\alpha}e_{-\pi_\alpha}v_{\pi_\alpha-\alpha}+
\frac{1}{2!}(\sum c_\gamma
 e_{-\gamma})^2v_{\pi_\alpha-\alpha}+\ldots,$$
where both sums are taken over the roots $\gamma$ such that
$\pi_\alpha-\alpha+\gamma\in\mathcal N\cap V_2$.

Let us show that the component  $v_{-\alpha}$ of the vector $v$ is equal to $c_{\pi_\alpha}e_{-\pi_\alpha}v_{\pi_\alpha-\alpha}$.
Let us write down its value:

$$c_{\pi_\alpha}e_{-\pi_\alpha}v_{\pi_\alpha-\alpha}+
\frac{1}{2!}\sum_{\gamma +\delta=\pi_\alpha} c_\gamma c_\delta
 e_{-\gamma}e_{-\delta}v_{\pi_\alpha-\alpha},$$
where $\gamma, \delta \in \Delta_{\pp_u}$ are such that $\pi_\alpha-\gamma,\pi_\alpha-\delta\in \supp(V_2^{'})$.
But from the last inclusion we obtain that the equality $\gamma +\delta=\pi_\alpha$ is not possible
(Indeed the last equality contradicts the equalities $(\gamma-\pi_\alpha;\pi_\beta^{'})=(\delta-\pi_\alpha;\pi_\beta^{'})\neq 0$ and $(\pi_\alpha;\pi_\beta^{'})=0$).
Frow which we get  $v_{\alpha}=c_{\pi_\alpha}e_{-\pi_\alpha}v_{\pi_\alpha-\alpha}.$

 The set of coefficients
$(t^2c_{\pi_\alpha},tc_{\gamma_1},\ldots,tc_{\gamma_k})$, where
${\gamma_i\in \Delta_{P_u}\setminus \{\pi_\alpha\}}$ define the orbit uniquely.
Thus the set of orbits can be identified with the weighted projective space
 $\Bbb P_{wt}(\Bbb K e_{-\pi_\alpha}v(\pi_\alpha-\alpha)\oplus \mathcal N\cap V_2)$.
\end{proof}

\begin{remark}\label{pr}
For the points $\langle v\rangle \in (G/P)^{ss}_{\varepsilon}$ different from $\langle v_{\pi_\alpha}+\Bbb K h_{\pi_\alpha}\rangle$,
the orthogonal projection on the subspace $V_1$ is not equal to zero. This defines a  $\lambda$-equivariant
projection on $\Bbb P(V_1)$ as well as the map
$$
\sigma_0:\lambda\backslash \!\! \backslash((G/P)^{ss}_{\varepsilon}\setminus
 \langle v_{\pi_\alpha}+\Bbb K h_{\pi_\alpha}\rangle) \rightarrow \Bbb P(V_1)
$$
\end{remark}

By $D_{h_{\pi_\alpha}}$ we denote a  $G^{'}$-invariant hyperplane section in  $\Bbb P(V)$ defined
by the vanishing of the coordinate by the vector $h_{\pi_\alpha}$.
 From the Remark \ref{pr}it follows that we have well defined $\lambda$-equivariant  projection on $\Bbb P(V_1)$
  of the set $(G/P)^{ss}_{\varepsilon}\cap D_{h_{\pi_\alpha}}$ that defines a map
 $\sigma_0:\lambda\backslash \!\! \backslash((G/P)^{ss}_{\varepsilon}\cap D_{h_{\pi_\alpha}})\longrightarrow \Bbb P(V_1)$.
 Our next aim is to study the quotient  $\lambda\backslash \!\! \backslash((G/P)^{ss}_{\varepsilon}\cap D_{h_{\pi_\alpha}})$
in the neighbourhood  $\sigma_0^{-1}(z)$ where $z\in G^{'}/P^{'}\subset \Bbb P(V_1)$ is the sufficiently general point.

Let us choose in  $\Bbb P(V)$ the affine chart $\mathcal U_{\alpha}=\{\langle v\rangle \in \Bbb P(V) | v_{\pi_\alpha-\alpha}\neq 0\}$.
We recall that conormal bundle to $G^{'}/P^{'}\subset \Bbb P(V_1)$ is identified with
$G^{'}*_{P^{'}}\mathcal N$ where $\mathcal N\cong V_1/(\langle v_{\pi_\alpha-\alpha} \rangle\oplus V_1^{'})$
is a fiber over the point $v_{\pi_\alpha-\alpha}$. (The fiber  $\mathcal N$ has a structure of  $P^{'}$-module.)

Consider the restriction of the considered conormal bundle
on the open cell $P^{'-}_uP^{'}/P^{'}=G^{'}/P^{'}\cap \mathcal U_\alpha$.
On $P^{'-}_uP^{'}/P^{'}$ we have a transitive action of the group $P^{'-},$
and the stabilizer of the point $eP^{'}/P^{'}$ is equal to $L^{'}$ .
 This implies that tha restriction of the conormal bundle is isomorphic to $P^{'-}*_{L^{'}}\mathcal N|_{L^{'}}$,
where $\mathcal N|_{L^{'}}\cong V_2^{'}\oplus \langle v_\alpha \rangle$ is a  $P^{'}$-module $\mathcal N$ considered as an
$L^{'}$-module. In particular the last isomorphism claims the restriction of the conormal bundle on the open cell
$P^{'-}_uP^{'}/P^{'}$ is isomorphic to the direct sum of subbundles $P^{'-}*_{L^{'}}V_2^{'}$ and
   $P^{'-}*_{L^{'}}\langle v_\alpha \rangle$.

\begin{proposition}\label{bl}   The quotient  $\lambda\backslash \!\! \backslash((G/P)^{ss}_{\varepsilon}\cap D_{h_{\pi_\alpha}}\cap\mathcal U_{\alpha})$ is isomorphic to the weighted blow up of the variety
  $\Bbb P(V_1)\cap \mathcal U_{\alpha}$ in the subvariety $P^{'-}_uP^{'}/P^{'}$. The preimage
   of $P^{'-}_uP^{'}/P^{'}$ is isomorphic to the projectivization of the conormal bundle to  $P^{'-}_uP^{'}/P^{'}$ in $\Bbb P(V_1)\cap \mathcal U_{\alpha}$,
and the fibers of subbundle $P^{'-}*_{L^{'}}\langle v_\alpha \rangle$ have the weight $2$, and the fibers $P^{'-}*_{L^{'}}V_2^{'}$ have the weight $1$.
\end{proposition}
\begin{remark} During the proof it will be stated in a precise way how we define the weights for the weighted blow up
 and we shall give the explicit  equations defining it.
\end{remark}
\begin{proof} Let $\langle v \rangle \in (G/P)^{ss}_{\varepsilon}\cap\mathcal U_{\alpha}$. Applying Lemma
 \ref{BGG} let us represent the vector  $v$ as the exponential map from the element of the Lie algebra
  $\widetilde{u}\in\Ad (s_\alpha) \pp_u^-$ applied to the vector $v_{\pi_\alpha-\alpha}$:
$$
v=\exp(\widetilde{u})v_{\pi_\alpha-\alpha}
$$
The map  $Ad (s_\alpha) \pp_u^-\longrightarrow Ad (s_\alpha) \pp_u^-v_{\pi_\alpha-\alpha}\subset V$ defines an
$L^{'}$-equivariant embedding of the Lie algebra in the  $G$-module $V$.
We have the following isomorphism of the $L^{'}$-modules:
$$
Ad (s_\alpha) \pp_u^-v_{\pi_\alpha-\alpha}\cong \langle v_{\pi_\alpha}\rangle
\oplus V_1^{'}\oplus s_\alpha V_2^{'}\oplus \langle v_{-\alpha} \rangle
$$

The last decomposition allows us to represent the vector $v$ in the form
$$
v=\exp(b_{\alpha}e_\alpha+\sum_{\pi_\alpha-\alpha-\delta\in \supp V^{'}_1} b_{\delta}e_{-\delta}
+\sum_{\pi_\alpha-\alpha-\gamma\in \supp V_2} c_{\gamma}e_{-\gamma}+c_{\pi_\alpha}e_{-\pi_\alpha})v_{\pi_\alpha-\alpha},
$$
for $\gamma,\delta\in s_\alpha\Delta_{\pp_u^-}$.
Also we should note that the pairing with the weights  $\pi_\alpha$ and $\pi'_\beta$ define $L'$-invariant
 gradings on $V(\pi_\alpha)$.
Their values define uniquely the components of the decomposition of the module $Ad (s_\alpha) \pp_u^-v_{\pi_\alpha-\alpha}$
into irreducible  $L'$-modules given above. The latter means that using this gradings we
can define in  modules the lie the corresponding monomials in the exponential  representation of vector  $v$.
This implies that the component of vector $v$ with the weight $\pi_\alpha-\alpha-\delta\in \supp V^{'}_1$
is equal to $b_{\delta}e_{-\delta}v_{\pi_\alpha-\alpha}$, the component of weight  $\pi_\alpha-\alpha-\gamma\in \supp V_2$
  is equal to $c_{\gamma}e_{-\gamma}v_{\pi_\alpha-\alpha}$, and a component of weight $\pi_\alpha-\gamma \in \supp V^{'}_2$
is equal to $$v_{\pi_\alpha-\gamma}=\frac{1}{2!}\sum_{\gamma=\delta_1+\delta_2+\alpha}b_{\delta_1}b_{\delta_2}(e_{-\delta_1}e_{-\delta_2}+
e_{-\delta_2}e_{-\delta_1})v_{\pi_\alpha-\alpha}
+\frac{1}{2!}b_\alpha c_\gamma (e_{\alpha}e_{-\gamma}+
e_{-\gamma}e_{\alpha})v_{\pi_\alpha-\alpha}=
$$

$$=\sum_{\gamma=\delta_1+\delta_2+\alpha}b_{\delta_1}b_{\delta_2}e_{-\delta_1}e_{-\delta_2}v_{\pi_\alpha-\alpha}
+b_\alpha c_\gamma
e_{-\gamma}e_{\alpha}v_{\pi_\alpha-\alpha},
$$
where $\pi_\alpha-\alpha-\delta_i\in \supp V^{'}_1$, and an equality holds since $[e_{-\delta_2};e_{-\delta_1}]=0$ и $[e_{\alpha};e_{-\gamma}]=0$.

Now let us find a coefficient of $h_{\pi_\alpha}$. Let us begin with writing down
the component of the weight zero for the vector $v$:
%%%% x_{h_{\pi_\alpha}}=
$$
\frac{1}{2!}\sum_{\delta+\gamma=\pi_\alpha-\alpha} b_{\delta} c_{\gamma} (e_{-\delta} e_{-\gamma}+e_{-\gamma} e_{-\delta})v_{\pi_\alpha-\alpha}+c_{\pi_\alpha-\alpha} e_{-\pi_\alpha+\alpha}v_{\pi_\alpha-\alpha}-\frac{1}{2!}
b_{-\alpha}c_{\pi_\alpha}h_{\pi_\alpha+\alpha}=
$$

$$
=\frac{1}{2}\sum_{\delta+\gamma=\pi_\alpha-\alpha} b_{\delta} c_{\gamma}h_{-\pi_\alpha+\alpha+2\delta}+c_{\pi_\alpha-\alpha}h_{-\pi_\alpha+\alpha}
-\frac{1}{2}b_{-\alpha}c_{\pi_\alpha}h_{\pi_\alpha+\alpha},
$$
where we used the equality $e_{-\delta} e_{-\gamma}+e_{-\gamma} e_{-\delta}=h_{\pi_\alpha-\alpha-2\delta}$ (obtained by the
calculation in $\mathfrak sl_2$-subalgebra generated by $e_{-\delta}$ and $e_{-\gamma}$), and the equality
$e_{-\pi_\alpha+\alpha}v_{\pi_\alpha-\alpha}=h_{\pi_\alpha-\alpha}$ (that is rewritten equality
$[e_{-\pi_\alpha+\alpha},e_{\pi_\alpha-\alpha}]=-h_{\pi_\alpha-\alpha}$).

Taking into account that $\langle \delta, \pi_\alpha \rangle=0$, and that the lenghts of the projections
 $h_{\pi_\alpha-\alpha}$ and $h_{\pi_\alpha+\alpha}$
 on the component $h_{\pi_\alpha}$ are equal to $(\pi_\alpha-\alpha,\pi_\alpha)/(\pi_\alpha,\pi_\alpha)=1/2$ and $(\pi_\alpha+\alpha,\pi_\alpha)/(\pi_\alpha,\pi_\alpha)=3/2$ respectively we obtain
 that the projection of $v$ on the component $h_{\pi_\alpha}$
is equal to:
$$
 x_{h_{\pi_\alpha}}:=-\frac{1}{2}(c_{\pi_\alpha-\alpha}+\frac{1}{2}\sum_{\delta+\gamma=\pi_\alpha-\alpha} b_{\delta} c_{\gamma})
 -\frac{3}{2}b_{-\alpha}c_{\pi_\alpha}.
$$

Let us write down the component of weight  $\alpha$ for the vector $v$:

$$v_\alpha=\sum_{\delta_1+\delta_2+\delta_3=\pi_\alpha-2\alpha}b_{\delta_1}b_{\delta_2}b_{\delta_3}
e_{-\delta_1}e_{-\delta_2}e_{-\delta_3}v_{\pi_\alpha-\alpha}+
$$

$$
+\frac{1}{6}c_{-\pi_\alpha}b^2_\alpha
(e^2_\alpha e_{-\pi_\alpha}+e_\alpha e_{-\pi_\alpha}e_\alpha)v_{\pi_\alpha-\alpha}+\frac{1}{2}b_\alpha c_{\pi_\alpha-\alpha}(e_\alpha e_{-\pi_\alpha+\alpha}+ e_{-\pi_\alpha+\alpha}e_\alpha)v_{\pi_\alpha-\alpha}+
$$

$$
+\frac{1}{6}b_\alpha\sum_{\delta+\gamma=\pi_\alpha-\alpha}
 c_\gamma b_\delta
(e_{-\delta}e_{\alpha}e_{-\gamma}+e_{-\delta}e_{-\gamma}e_{\alpha}+e_{\alpha}e_{-\delta}e_{-\gamma}+e_{\alpha}e_{-\gamma}e_{-\delta})v_{\pi_\alpha-\alpha}
$$

Let us notice that in the last summand we omitted the zero monomials $e_{-\gamma}e_{\alpha}e_{-\delta}v_{\pi_\alpha-\alpha}$  and $e_{-\gamma}e_{-\delta}e_{-\alpha}v_{\pi_\alpha-\alpha}$.
They are trivial since the weights of the vectors
$e_{\alpha}e_{-\delta}v_{\pi_\alpha-\alpha}=e_{-\delta}e_{-\alpha}v_{\pi_\alpha-\alpha}$  are equal to $\pi_\alpha-\delta$
and these weight do not belong to the weight polytop of the representation  (that follows from the equality $(\pi_\alpha;\delta)=0$). Also we have $e_{\alpha}e_{-\gamma}e_{-\delta}v_{\pi_\alpha-\alpha}=0$. Indeed,
 $e_{-\gamma}e_{-\delta}v_{\pi_\alpha-\alpha}=[e_{-\gamma}[e_{-\delta},e_{\pi_\alpha-\alpha}]]=N_{\delta,{\pi_\alpha-\alpha}}[e_{-\gamma},e_{\gamma}]
 =-N_{\delta,{\pi_\alpha-\alpha}}h_\gamma$. But we have $[e_\alpha;h_\gamma]=0$,
 since  $(\gamma,\alpha)=0$;  that proves the formula.

Let us simplify expression for $v_\alpha$ using commutation relations  $[e_{-\pi_\alpha+\alpha},e_\alpha]=0$,
$[e_{-\delta},e_{\alpha}]=0$ and $[e_{\alpha};e_{-\gamma}]=0$:

$$v_\alpha=\sum_{\delta_1+\delta_2+\delta_3=\pi_\alpha-2\alpha}b_{\delta_1}b_{\delta_2}b_{\delta_3}
e_{-\delta_1}e_{-\delta_2}e_{-\delta_3}v_{\pi_\alpha-\alpha}
%+\frac{1}{3}c_{-\pi_\alpha}b^2_\alpha
%(e^2_\alpha e_{-\pi_\alpha})v_{\pi_\alpha-\alpha}+
-\frac{1}{2}c_{-\pi_\alpha}b^2_\alpha e_\alpha+
$$

$$
+b_\alpha c_{\pi_\alpha-\alpha}(e_\alpha e_{-\pi_\alpha+\alpha})v_{\pi_\alpha-\alpha}
+\frac{1}{2}b_\alpha\sum_{\delta+\gamma=\pi_\alpha-\alpha}
 c_\gamma b_\delta
(e_{\alpha}e_{-\delta}e_{-\gamma})v_{\pi_\alpha-\alpha}
$$

Denoting the first summand by $q_\alpha$ we obtain that:
$$v_\alpha=q_\alpha+Ac_{-\pi_\alpha}b^2_\alpha
(e^2_\alpha e_{-\pi_\alpha})v_{\pi_\alpha-\alpha}+x_{h_{\pi_\alpha}}b_\alpha e_{\alpha}e_{-\pi_\alpha+\alpha}v_{\pi_\alpha-\alpha},
$$
where $A$ is nonzero constant whose value is not important for us.

%  (e^2_\alpha e_{-\pi_\alpha})v_{\pi_\alpha-\alpha}=3e_\alpha

To find the intersection of $\mathcal U_\alpha$ with the divisor $D_{h_{\pi_\alpha}}$ we must put the coordinate $x_{h_{\pi_\alpha}}$,
 equal to zero, that implies  $\langle v\rangle \in G/P\cap \mathcal U_\alpha\cap D_{h_{\pi_\alpha}}$:
$$v_\alpha=q_\alpha+Ac_{-\pi_\alpha}b^2_\alpha
(e^2_\alpha e_{-\pi_\alpha})v_{\pi_\alpha-\alpha}
$$

Denote by  $x_\omega$ the coordinate of the weight vector $v(\omega)$. The ideal  $\mathcal J$
of the subvariety  $P^{'-}_uP^{'}/P^{'}\subset \Bbb P(V_1)
\cap \mathcal U_\alpha$
is generated by the following elements:

$$
Q_\alpha v(\alpha):=x_\alpha v(\alpha)-\sum_{\delta_1+\delta_2+\delta_3=\pi_\alpha-2\alpha}x_{\delta_1}x_{\delta_2}x_{\delta_3}
e_{-\delta_1}e_{-\delta_2}e_{-\delta_3}v_{\pi_\alpha-\alpha}
$$
$$
P_{\pi_\alpha-\gamma}v(\pi_\alpha-\gamma):=x_{\pi_\alpha-\gamma}v(\pi_\alpha-\gamma)-\sum_{\mu=\delta_1+\delta_2+\alpha}x_{\delta_1}x_{\delta_2}e_{-\delta_1}e_{-\delta_2}v_{\pi_\alpha-\alpha},
$$
where $\pi_\alpha- \gamma\in \supp(V_2^{'})$. The last assertion can be obtained from  the expression of
 the element of $P^{'-}_uP^{'}/P^{'}$ as the exponential map  $\exp(\sum_{\pi_\alpha-\alpha-\delta\in \supp V^{'}_1} b_{\delta}e_{-\delta})v_{\pi_\alpha-\alpha}$.
The conormal bundle to  $P^{'-}_uP^{'}/P^{'}$ in $\Bbb P(V_1)$ is identified with $\mathcal J/\mathcal J^2$.
It is generated by the elements described above. Indeed the number of these elements is equal to the dimension of the fiber
 of conormal bundle and the linear parts of these equations are equal to the coordinates  $x_\alpha$ and $x_{\pi_\alpha-\gamma}$
for $\pi_\alpha- \gamma\in \supp(V_2^{'})$, that implies the linear independence of these equations modulo
 $\mathcal J^2$ (since the elements from this ideal do not have linear parts).

The relations for the coordinates of the vector $\langle v \rangle\in ((G/P)^{ss}_{\varepsilon}\cap D_{h_{\pi_\alpha}}\cap\mathcal U_{\alpha})$ can be written in the form:

\hspace{20ex}  $ \left \{ \begin
{array}{cl}
Q_\alpha=c_{-\pi_\alpha}b^2_\alpha N_\alpha  \\
\ \ P_{\pi_\alpha-\gamma}=b_\alpha c_\gamma N_{\pi_\alpha-\gamma}     \\
\end{array} \ \ \ \ \ (*)
\right.$

\noindent where $N_\alpha=A{|(e^2_\alpha e_{-\pi_\alpha})v_{\pi_\alpha-\alpha}|}/{|v(\alpha)|}, N_{\pi_\alpha-\gamma}=|{e_{-\gamma}e_{\alpha}v_{\pi_\alpha-\alpha}|}/{|v(\pi_\alpha-\gamma)|}$ are the  normalizing constants.

Let us notice that since we have a decomposition  $\mathcal J/\mathcal J^2$ in the direct sum of two vector bundles
 $P^{'-}*_L\langle v_\alpha\rangle$ and $P^{'-}*_LV^{'}_2$,
  we can consider a weighted pprojectivisation $P^{'-}*_L\Bbb P_{wt}(\Bbb K v_\alpha\oplus V^{'}_2)$
with the weight $2$ on the first subbundle and with the weigth $1$ on the second.
We also have a correctly defined weighted blow up $Bl_{wt}$ of the variety  $\mathcal U_{\alpha}\cap D_{h_{\pi_\alpha}}$
in $P^{'-}_uP^{'}/P^{'}$. In other words we can consider the weighted   $\proj_{wt} \bigoplus \mathcal J^n$,
 where the generator $Q_\alpha$ of the module $\mathcal J$
has weight $2$,  and other generators  $P_{\pi_\alpha-\gamma}$ have weight  $1$.

We shall gve the argument that provide us the explicit equations of the quotient that
allow us to identify it with the weighted blow up.
In the local coordinates the weighted blow up $Bl_{wt}$ can be described as the subvariety in $\Bbb K^{\dim V_1-1}\times \Bbb P_{wt}(\Bbb K v_\alpha\oplus V^{'}_2)$. Here we identify  $\Bbb K^{\dim V_1-1}$ with $V_1^{'}\oplus V_2^{'}\oplus \Bbb K v_{\pi_\alpha}$.
We also denote by $c_{\pi_\alpha}$, $c_\gamma$ for $\pi_\alpha-\gamma\in \supp V_2^{'}$ the weighted homogenous coordinatesof the subspace
 $\Bbb P_{wt}(\Bbb K v_\alpha\oplus V^{'}_2)$. Besides the first coordinate has weight $2$, and the latter have weigth $1$.
The variety  $Bl_{wt}$ is defined by the followin system of equations:

\hspace{20ex}  $ \left \{ \begin
{array}{cl}
P_{\pi_\alpha-\gamma_i}c_{\gamma_j}=P_{\pi_\alpha-\gamma_j}c_{\gamma_i} \\
Q_{\alpha}c^2_{\gamma_j}=P^2_{\pi_\alpha-\gamma_j}c_{\pi_\alpha}       \\
\end{array} \ \ \ \ \ (*_{BL})
\right.$
%$$
%p_{\pi_\alpha-\gamma_i}c_{\gamma_j}=p_{\pi_\alpha-\gamma_j}c_{\gamma_i}
%$$
%$$
%q_{\alpha}c^2_{\gamma_j}=p^2_{\gamma_j}c_{\pi_\alpha}
%$$

\noindent дл€ $\pi_\alpha-\gamma\in \supp V_2^{'}$.
This system of equations is invariant with respect of the action of one-parameter subgroup  $\lambda$. Using the equations
 $(*)$   on the coordinates of $\langle v\rangle \in((G/P)^{ss}_{\varepsilon}\cap D_{h_{\pi_\alpha}}\cap\mathcal U_{\alpha})$,
we obtain that the required quotient  $\lambda\backslash \!\! \backslash((G/P)^{ss}_{\varepsilon}\cap D_{h_{\pi_\alpha}}\cap\mathcal U_{\alpha})$
is defined by the system of equations $(*_{BL})$ and thus can be identified withe the weighted blow up of the variety $\Bbb P(V_1)\cap\mathcal U_{\alpha}$ in the subvariety $P^{'}_uP^{'}/P^{'}$.
\end{proof}

\begin{remark}\label{wtblrem}
For some purposes it is useful to have the following interpretation of the weighted blow up. Let $ \Bbb K^{l_0+l}$
be the vector space that we want to blow up in the subspace  $\Bbb K^{l_0}$ with the weights $(a_1,\ldots,a_l)$ on the subspace  $\Bbb K^l$.
 Let $\mu=\mu_{a_1}\times \ldots\times \mu_{a_l}$ be the product of the groups of the roots of unity of orders $(a_1,\ldots,a_l)$.
Consider the quotient morphism $\tau:\Bbb K^l\rightarrow \Bbb K^l$ by the action of the group  $\mu$, defined on the $i$-th coordinate as  $z_i\rightarrow z_i^{a_i}$.
  We also consider the quotient morphism  $\tau_{\Bbb P}$
  of the projective space  $\Bbb P(\Bbb K^l)$  by the action of $\mu$, that maps $\Bbb P(\Bbb K^l)$
  to the weighted projective space
  $\Bbb P(a_1,\ldots,a_l)$. In the homogeneous coordinates we have:
  $$
  (z_1:\ldots:z_l)\longrightarrow   (z_1^{a_1}:\ldots:z_l^{a_l})
  $$

  The action of $\mu$ can be defined on the blow up $W\oplus \Bbb K^{l}$ in $W$ ( defined as
  $Bl:=\{(v,z,\xi)\in \Bbb K^{l_0+l} \times \Bbb P(\Bbb K^l)|  z_i\xi_j=z_j\xi_i\}$). The generator $\varepsilon_{i} \in \mu_{a_i}$
  acts as  $z_i\rightarrow \varepsilon_{i}z_i$ on the coordinates of $\Bbb K^{l}$ and  as $\xi_i\rightarrow \varepsilon_{i}\xi_i$
on the homogeneous coordinates of $\Bbb P(\Bbb K^l)$. The quotient $\widetilde{\tau}$ of the blow up $Bl$ by the considered action of $\mu$ is identified with the weighted blow up $Bl_{wt}=\{(v,z,\xi)\in \Bbb K^{l_0+l}\times \Bbb P(a_1,\ldots,a_l)|  z_i^{a_j}\xi_j^{a_i}=z_j^{a_i}\xi_i^{a_j}\}$.
The quotient map is the restriction of the map $$\Bbb K^{l_0+l} \times \Bbb P(\Bbb K^l)\rightarrow
\Bbb K^{l_0+l}\times \Bbb P(a_1,\ldots,a_l)$$
$$(w\oplus (z_1,\ldots,z_l))\times (\xi_1:\ldots:\xi_l)\longrightarrow (w\oplus (z_1^{a_1},\ldots,z_l^{a_l}))\times  (\xi_1^{a_1}:\ldots:\xi_l^{a_l}) $$
Indeed making a substitute   $z_i\rightarrow z_i^{a_i}$, $\xi_i\rightarrow \xi_i^{a_i}$ in the equation $z_i^{a_j}\xi_j^{a_i}=z_j^{a_i}\xi_i^{a_j}$,
we get $(z_i\xi_j)^{a_ia_j}=(z_j\xi_i)^{a_ia_j}$, that is the system of equations defining the weighted blow up $Bl$.

Summarizing what was said above we obtain the following commutative diagram:
$$
\xymatrix{
 \Bbb K^{l_0+l} \ar[d]_{\tau}& \ar[l] Bl   \ar@{^{(}->}[r] \ar[d]_{\widetilde{\tau}}&  \ar[d]_{\tau\times\tau_{\Bbb P}}\Bbb K^{l_0+l} \times \Bbb P(\Bbb K^l) & \\
 \Bbb K^{l_0+l}&  \ar[l] Bl_{wt}   \ar@{^{(}->}[r]&\Bbb K^{l_0+l}\times \Bbb P(a_1,\ldots,a_l)& \\
}
$$

\end{remark}

%%%%%%%%%%%%%%%%%%%%%%%%%%%%%%%%%%%%%%%%%%%%%%%%%%%%%%%%%%%%%%%%%%%%%%%%%%%%%%%%%%%%%%%%%%%%%%%%%%%%%%%%%%%%%%%%%%%%%%%%%%%%%%%

%%%%%%%%%%%%%%%%%%%%%%%%%%%%%%%%%%%%%%%%%%%%%%%%%%%%%%%%%%%%%%%%%%%%%%%%%%%%%%%%%%%%%%%%%%%%%%%%%%%%%%%%%%%%%%%%%%%%%%%%%%%%%%%

\section{The embedding of the torsor over  $X_{\Delta}$ in the affine cone over $G/P$}

Let us formulate the main theorem from the paper of Skorobogatov and Serganova
\cite{skor}. We shall follow their main scheme of the proof but we shall modify some details
that will allow us to get the analogous results in the case of the root system $E_8$.

\begin{theorem}\label{main} Consider the embedding of the variety $G/P$ in the projectivization
of miniscule representation $\Bbb P(V(\pi_\alpha))$. For the left action of maximal torus
 $T$ on $G/P$ we denote by $(G/P)^s$
the set of stable points with respect of  $G$-linearized sheaf
$\mathcal O(1)|_{G/P}$. Let   $\tau:(G/P)^s\longrightarrow
T\backslash \!\! \backslash(G/P)^s$ be a quotient morphism. There exists an embedding
$\imath$ of the del Pezzo surface $X_{ \Delta}$ with the degree $9-\rk
\Delta$ into the quotient  $T\backslash \!\!\backslash(G/P)^s$ such that for the torsor
$\mathcal T=\tau^{-1}(\imath (X_{ \Delta}))$ the following conditions hold

\begin{itemize}
 \item[$\bullet$] Consider the hypersurface $H_\omega=\{\langle v\rangle \in \Bbb P(V(\pi_\alpha))\  |\
 v_\omega=0\}$. The divisor $E_\alpha=\tau(H_\alpha\cap \mathcal T)\subset X_{\Delta}$ is equal to a
 $(-1)$-curve on $X_{\Delta}$.

 \item[$\bullet$] The $(-1)$ curves $E_{\omega_1}$ and $E_{\omega_2}$ do not intersect iff the weights
  $\omega_1$ and $\omega_2$ are not adjacent in the weight polytop.
\end{itemize}

\end{theorem}

We construct the required embedding by the induction. The base of the induction follows from the isomorphism
 $X_{A^4}\cong T \backslash \!\!\backslash (G/P)$
(for $G=SL(5)$), that was proved by Skorobogatov in   \cite {SKO1}.

 Let us assume that  embedding of the torsor
$\mathcal T^{\, '}\subset G^{'}/P^{'}\subset \Bbb P(V_1)$ is already constructed.
 Let us denote by $\tau^{'}$ the quotient by the left action of the torus  $T^{'}$.
 The quotient $\tau^{'}(\mathcal
T^{'})=T^{'}\backslash \!\!\backslash\mathcal T^{\, '}$ is a del Pezzo surface
$X_{\Delta^{'}}$ of the degree $9-\rk \Delta^{'}$. This is a plane with the blown up $\rk \Delta^{'}$ points
$e_i$ in a general position. (Cf.\cite{Manin}. By the general position it is assumed that no
line contain three points and no conic contain $6$ points from this set)

 The divisors corresponding to the exceptional curves lying over $e_i$ we denote
by $E_i$. We want  to construct torsor $\mathcal T\subset G/P\subset
\Bbb P(V(\pi_\alpha))$ over del Pezzo surface $X_\Delta$,
that is obtained from $X_{\Delta^{'}}$ by the blow up of the point $e_{\rk
\Delta}\in X_{\Delta^{'}}$ that is in general position with the points  $e_i$ where
 $i\leqslant \rk \Delta^{'}$ (The image of the point
$e_{\rk \Delta}$ under the contraction $X_{\Delta^{'}}\dashrightarrow
\Bbb P^2$ is also denoted by $e_{\rk \Delta}$).

Let  $s\in G^{'}/P^{'}$ be a point for which
$p_{\pi_\alpha}(s)=e_{\rk \Delta}$, and $s^{'}\in G^{'}/P^{'}$ is another point
on the flag variety (later we assume that it is sufficiently general).
 Let $\{s_\omega\}$ be the set of homogenous coordinates of the point
 $s\in \Bbb P(V_1)$ with respect to the weight basis of the subspace
 $V_1$ (it is well defined since
 $V_1$ --- is miniscule representation). By the action of $s$
on the point of projective space $x=\langle v\rangle$ we mean the action $sx= \langle\sum s_\omega v_\omega\rangle$.

Let us prove the following theorem (cf. \cite[6.3]{skor}).

\begin{theorem} \label{P-Q} Consider the embedding $G^{'}/P^{'}\subset \Bbb
P(V_1)$. Let $s\in G^{'}/P^{'}$ be a point of a flag variety  such that
 $p_{\pi_\alpha}(s)=e_{\rk \Delta}$. Consider the equations
$p_\mu(x)=0$ and $q_\nu(x)=0$ from the Remark \ref{equation}. For the general point
 $s^{'}\in G^{'}/P^{'}$ we have the following statements

\begin{itemize}
 \item[1)] The restrictions of  $p_\mu(x)=0$ and
 $q_\nu(x)=0$ to $s^{'}s^{-1}\mathcal T^{\,
 '}$ nontrivial. The image of the set of zeros of these equations define nontrivial divisors
 on $X_{\Delta^{'}}$ those  proper transforms with respect to the blow up
  $\sigma:X_{\Delta}\longrightarrow X_{\Delta^{'}}$ are $(-1)$-curves.
 \item[2)] The varieties  $s^{'}s^{-1}\mathcal T^{\,
 '}$ and  $G^{'}/P^{'}$ intersect in a single  $T^{'}$-orbit i.e. $s^{'}s^{-1}\mathcal T^{\,
 '}\cap G^{'}/P^{'}=T^{'}s^{'}$.
\end{itemize}

\end{theorem}
\begin{remark} We give an alternative prove that differs from   \cite[6.3]{skor}
and can be written uniformly for   $p_\mu(x)$ and
 $q_\nu(x)$ that generalizes to the case of $E_8$.
\end{remark}
\begin{remark} For the points $x\in \mathcal T^{\,
 '}$  by the phrase: ``the equation $x_\omega=0$ (or $p_\mu(x)=0$,
 $q_\nu(x)=0$) defines a divisor  (curve) on the surface $X_{\Delta^{'}}$'', for the brevity we mean the following:
  the equality  $x_\omega=0$ (or $p_\mu(x)=0$,
 $q_\nu(x)=0$) defines a divisor of zeros of the section of the line bundle  $\mathcal O(1)$
 (corr. $\mathcal O(2)$, $\mathcal O(3)$)
 on $\Bbb P(V_1)$. We can intersect it with the torsor ${\mathcal T^{\,
 '}}$. In the case when this intersection is nontrivial since the sections are semiinvariant
 with respect to the action $T^{'}$ it follows that $x_\omega|_{{\mathcal T^{\, '}}}=0$ (or $p_\mu(x)|_{\mathcal T^{\,
 '}}=0$, $q_\nu(x)|_{\mathcal T^{\, '}}=0$) define also the divisors on the quotient
  $X_{\Delta^{'}}=T^{'}\backslash \!\! \backslash {\mathcal T^{\,
 '}}$.
\end{remark}
\begin{proof} The first part of the theorem is equivalent to the statement  that the restrictions on $\mathcal T^{\,
 '}$ of the equations $p_\mu(s^{'}s^{-1}x)=0$ and
 $q_\nu(s^{'}s^{-1}x)=0$ are nontrivial. Indeed let us notice that these polynomials  are zero in the point   $s$,
 since $p_\mu(x)$ and
 $q_\nu(x)$ are equal to zero on the flag variety $G^{'}/P^{'}$.
By the induction assumption the equations  $x_\omega=0$ define the
$(-1)$-curves $E_\omega$ on $X_{\Delta^{'}}$, and also on
$T^{'}\backslash \!\!\backslash{(s^{'}s^{-1}\mathcal T^{\,
 '}})$.

Consider the curves  $C$ of the degree  $2$ on $X_{\Delta^{'}}$ with respect to the pairing
with the canonical class (i.e.
$(K_{X_{\Delta^{'}}};C)=2$) passing through  $e_{\rk\Delta}$ (it will
become a  $(-1)$-curve after the blow up of the point $e_{\rk\Delta}$).
Let us describe  these curves  by the equations. From \cite{Manin} it is easy to see that after
the contraction of $X_{\Delta^{'}}$  to $\Bbb P^2$ such curve maps
to the line passing through   $e_{\rk\Delta}$ and
$e_i$ for some $i$ (for example when $i=1$ we obtain
$C=L_0-E_1-e_{\rk\Delta}$), or to the conic passing through any four points and the point $e_{\rk\Delta}$
($C=L_0-E_1-E_2-E_3-E_4-e_{\rk\Delta}$). Or in the cubic  (for example
$C=L_0-2E_1-E_2-E_3-E_4-E_5-E_6-E_7-e_{\rk\Delta}$)  that has the ordinary double point
in some point $e_i$ and passes through all points  (this case occurs only when
$\Delta^{'}$ is of type $E_7$). Let us notice that the dimensions of the linear systems
defined by these curves is equal to $0$. Let us show that
$p_\mu(s^{'}s^{-1}x)=0$ defines a curve of the considered type (in the case when $p_\mu(s^{'}s^{-1}x)$ is not equal to zero on
$\mathcal T^{\, '}$).

In the first case we have the following presentations of the class of curve $C$ as
the sum of two classes of $(-1)$-curves (cf. fig.3):
 $C=(L_0-E_1-E_i)+E_i, \ \ \ \text{where}\ i\neq 1.$
In the second  (cf. fig.3) and in the third case we have:
\noindent $C=(2L_0-E_1-E_2-E_3-E_4)=(L_0-E_1-E_2)+(L_0-E_3-E_4)=
(L_0-E_1-E_4)+(L_0-E_3-E_2)=(L_0-E_1-E_3)+(L_0-E_2-E_4)=
(2L_0-E_1-E_2-E_3-E_4-E_5)+E_5=(2L_0-E_1-E_2-E_3-E_4-E_6)+E_6$

\begin{center}
  \epsfig{file=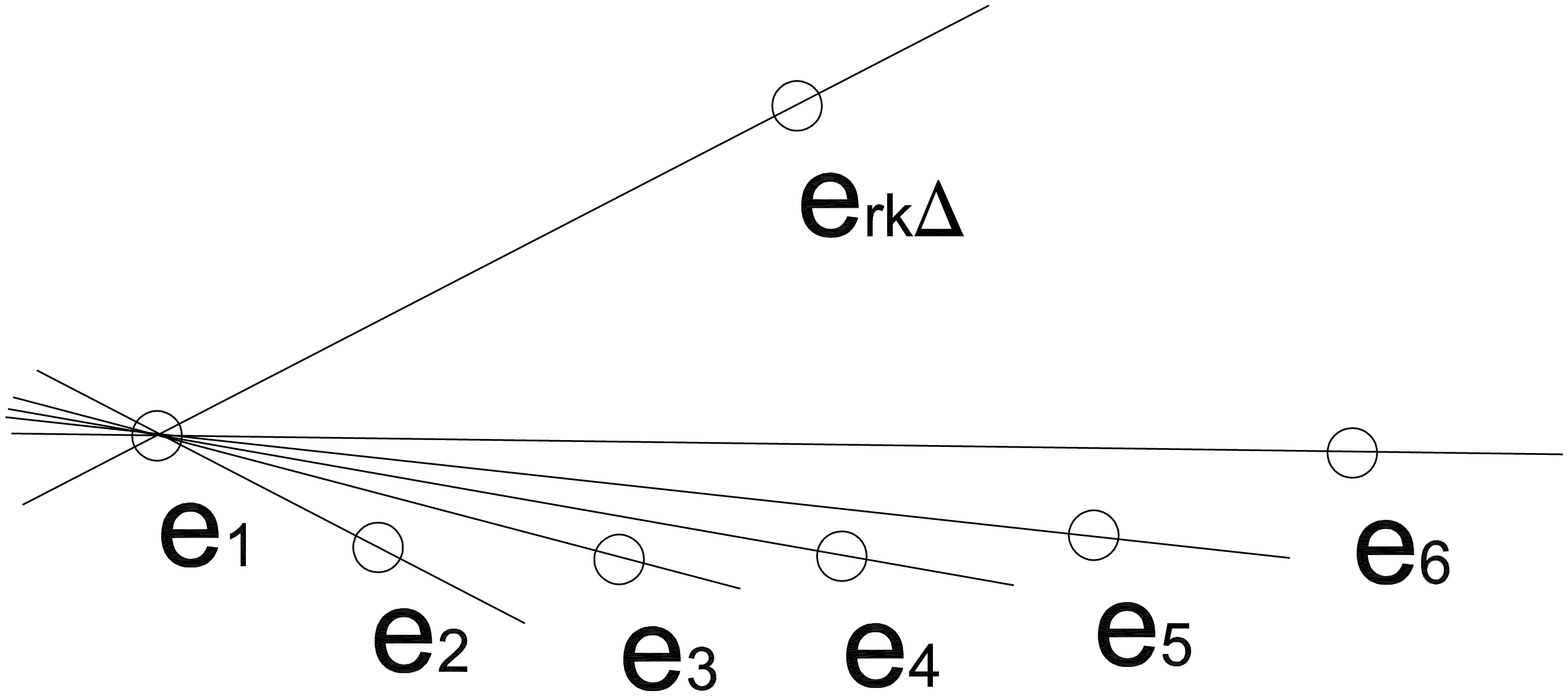,height=7cm,angle=90,clip=}\epsfig{file=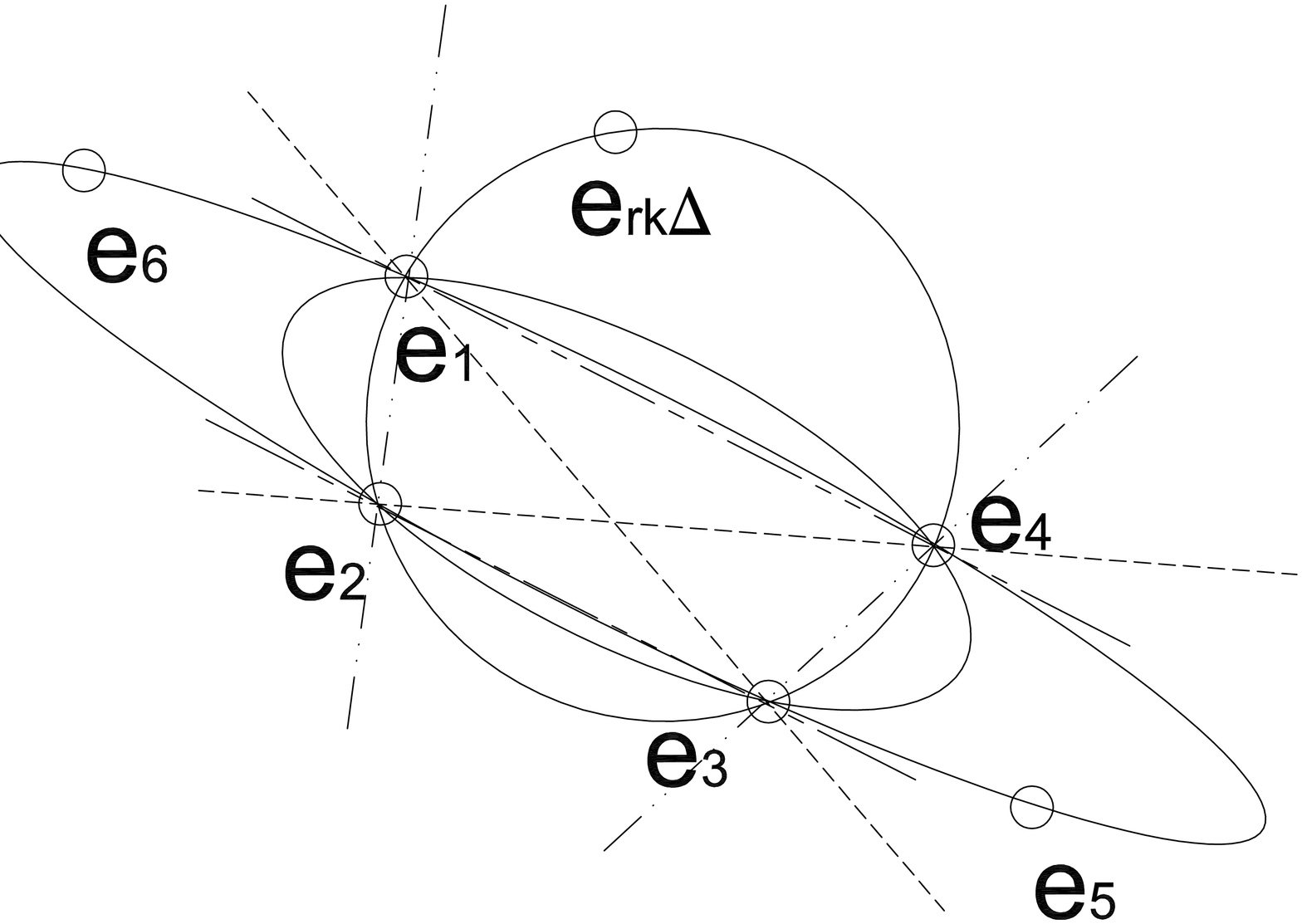,height=6cm,angle=90,clip=}

fig.3
\end{center}

$$C=(3L_0-2E_1-\sum \limits_{i\neq 1}E_i)=(L_0-E_1-E_i)+(2L_0-\sum \limits_{j\neq 1,i}E_j) \ \text{где}\ i\neq 1.$$
The classes of $(-1)$-curves are in bijective correspondence  with the weights
$\omega$ (cf.\cite{Manin}), the following decompositions correspond to presentation of some weight $\mu$ as he sum
$\mu=\gamma_i+\delta_i$ where $\gamma_i$ is the weight corresponding to
$L_0-E_1-E_i$ (exceptional curve defined by the equation
$x_{\gamma_i}=0$), and $\delta_i$
is the weight corresponding to $E_i$ (with the equation $x_{\delta_i}=0$). The linear system  $|L_0-E_1|$ has the dimension
 $1$ and the considered decompositions correspond to the reducible elements of this linear system. There linear combination
$$p_\mu(s^{'}s^{-1}x)=\sum \limits_{\pi-\gamma-\delta=\mu}x_{\gamma}x_{\delta}\frac{s^{'}_{\gamma}s^{'}_{\delta}}
{s_{\gamma}s_{\delta}}(e_{-\gamma}e_{-\delta}v_{\pi_\alpha})=0,$$
defines some element of this linear system  (or a zero element) passing through  $e_{\rk\Delta}$. Thus we get that
 $p_\mu(s^{'}s^{-1}x)=0$ defines a curve $C$.

Let us rewrite $p_\mu(s^{'}s^{-1}x)=0$ as a polynomial on  $s^{'}$
$$\widetilde{p}_\mu(s^{'})=p_\mu(s^{'}s^{-1}x)=\sum
\limits_{\pi-\gamma-\delta=\mu}p_{\gamma\delta}\frac{x_{\gamma}x_{\delta}}
{s_{\gamma}s_{\delta}}s^{'}_{\gamma}s^{'}_{\delta}=0,$$

 Now let us prove nontriviality of the restriction
$p_\mu(s^{'}s^{-1}x)=0$ on the torsor $\mathcal T^{\, '}$ for
a general point  $s^{'}\in G^{'}/P^{'}$. Assume the contrary, then the restriction
$p_\mu(s^{'}s^{-1}x)|_{\mathcal T^{\, '}}$ is trivial for all
$s^{'}\in G^{'}/P^{'}$,  in other words
$\widetilde{p}_\mu(s^{'}) \in \mathcal J_{G^{'}/P^{'}}$ for all
$x\in {\mathcal T^{\, '}}$. It is easy to see that for
$\widetilde{p}_\mu(s^{'}) \in \mathcal J_{G^{'}/P^{'}}(\mu)$ all monomials from
 $\widetilde{p}_\mu(s^{'})$ have the form
$s^{'}_{\gamma}s^{'}_{\delta}$ for $\pi-\gamma-\delta=\mu$
(where the weight $\mu$ is considered as the $T^{'}$).  Since $ \mathcal
J_{G^{'}/P^{'}}(\mu)$ is generated by the equation $p_\mu(s^{'})=0$, we get that the set of coefficients
$\{p_{\gamma\delta}\frac{x_{\gamma}x_{\delta}}
{s_{\gamma}s_{\delta}}\}$ is proportional to the set
$\{p_{\gamma\delta}\}$. Thus for any pair
$\gamma_1,\delta_1$ and $\gamma_2,\delta_2$, such that
$\pi-\gamma_1-\delta_1=\mu$ и $\pi-\gamma_2-\delta_2=\mu$ we have:
$$\frac{x_{\gamma_1}x_{\delta_1}}
{s_{\gamma_1}s_{\delta_1}}=\frac{x_{\gamma_2}x_{\delta_2}}
{s_{\gamma_2}s_{\delta_2}}.\eqno{(*)}$$

Let us recall that equations $x_\omega=0$ define $(-1)$-curves on
$X_{\Delta^{'}}$. Let us choose some point on the  $(-1)$-curve
defined by $x_{\gamma_1}=0$  and not lying on the other $(-1)$-curves.
Let us take a point in its preimage in the torsor $\mathcal T^{\, '}$ with the coordinates
 $\{x^0_\omega\}$.  Then $x^0_{\gamma_1}=0$, but also
$x^0_{\gamma_2}\neq 0$ and $x^0_{\delta_2}\neq 0$ that contradicts with the equality $(*)$.
That proves the assertion.

Let us give a second proof of this fact and its generalization in the case of the cubic curve  $q_\nu(x)$.

Assume that claim  $1)$ is not true.  Then the restrictions
 $\widetilde{p}_\mu(s^{'})=p_\mu(s^{'}s^{-1}x)$ (corr.
$\widetilde{q}_\nu(s^{'})=q_\nu(s^{'}s^{-1}x)$) for $x\in\mathcal
T^{\, '}$ are identically equal to zero on  $G^{' }/P^{'}$. Consider the restriction
  $\widetilde{p}_\mu(s^{'})$
($\widetilde{q}_\nu(s^{'})$) on the open cell
$P_u^{'-}P^{'}/P^{'}=  \exp(\sum \limits_{\gamma \in
\Delta(\pp_u^{'})}c_\gamma e_{-\gamma})v_{\pi_\alpha-\alpha}$.
Then $\widetilde{p}_\mu(s^{'})$ (corr.
$\widetilde{q}_\nu(s^{'})$) are the polynomials (on the open cell identified
with the linear space) from the coordinates
 $c_\gamma$ where $\gamma \in \Delta(\pp_u^{'})$. We shall show that for some point
 $x\in \mathcal T^{\, '}$ for the polynomial $\widetilde{p}_\mu(s^{'})$ (соот.
$\widetilde{q}_\nu(s^{'})$) the coefficient by some monomial on the coordinates
   $c_\gamma$ is not equal to zero. That implies nontriviality of
$\widetilde{p}_\mu(s^{'})$ (corr. $\widetilde{q}_\nu(s^{'})$) as a polynomial on $s^{'}$.

Let us fix the weight bases  $\{v_\omega\}$ and $\{v^{'}_\omega\}$
of the subspace $V(\pi_\alpha)$ and $V_1$. Let us notice that we do not assume that these bases are concordant.
The coefficients of the polynomials
$\widetilde{p}_\mu(s^{'})$ (corr. $\widetilde{q}_\nu(s^{'})$)
we choose with respect to the basis   $\{v_\omega\}$. In other words
$$v_{\pi_\alpha-\gamma}=e_{-\gamma}v_{\pi_\alpha}, \text{где} \ \gamma\in \Delta_{\pp_u},$$
$$e_{-\gamma}e_{-\delta}v_{\pi_\alpha}=p_{\gamma\delta}v_\mu, \text{где} \  \pi-\gamma-\delta=\mu,$$
$$e_{-\mu_1}e_{-\mu_2}e_{-\mu_3}v_{\pi_\alpha}=q_{\mu_1\mu_2\mu_3}v_\nu,\text{где} \ \pi_\alpha-\mu_1-\mu_2-\mu_3=\nu.$$
Correspodingly for the subspace $V_1$ we have:
$$v^{'}_{(\pi_\alpha-\alpha)-\gamma^{'}}=e_{-\gamma^{'}}(e_{-\alpha}v_{\pi_\alpha})=n_{\gamma^{'}}
v_{\pi_\alpha-\gamma}, \text{где} \ \gamma^{'}\in
\Delta_{\pp^{'}_u}, \gamma=\alpha+\gamma^{'},$$
$$e_{-\gamma^{'}}e_{-\delta^{'}}(e_{-\alpha}v_{\pi_\alpha})=p^{'}_{\gamma^{'}\delta^{'}}v^{'}_{\mu^{'}}, \text{где} \
 (\pi_\alpha-\alpha)-\gamma^{'}-\delta^{'}=\mu^{'}.$$
\begin{lemma}\label{[ee]}
We have the equality
$[e_{-\alpha},e_{-\gamma^{'}}]=-n_{\gamma}e_{-\gamma}$, где
$\gamma^{'}\in \Delta_{\pp^{'}_u}$, а $\gamma=\alpha+\gamma^{'}$.
\end{lemma}
\begin{proof} Since $\gamma^{'}\in
\Delta_{\pp^{'}_u}$ we have $\langle\alpha; \gamma^{'}\rangle=0$,
that implies $e_{\gamma^{'}}v_{\pi_\alpha}=0$. From the chain of equalities
$$ce_{-\gamma}v_{\pi_\alpha}=[e_{-\alpha},e_{-\gamma^{'}}]v_{\pi_\alpha}
=(e_{-\alpha}e_{-\gamma^{'}}-e_{-\gamma^{'}}e_{-\alpha})v_{\pi_\alpha}=-e_{-\gamma^{'}}e_{-\alpha}v_{\pi_\alpha}
=-n_{\gamma^{'}}e_{-\gamma}v_{\pi_\alpha}$$ we get that
$c=-n_{\gamma^{'}}$.
\end{proof}

Without the loss of generality we can assume that the  $(-1)$-curve
that correspond to $v_{\pi_\alpha-\alpha}=0$ is $E_{\rk
\Delta^{'}}$, and one of the monomials  $p_{\mu}$ is equal to $s^{'}_\alpha
s^{'}_\omega$ where $\omega=L-E_1-E_{\rk \Delta^{'}}$. Let us recall that we consider the points
 $s^{'}=\langle v\rangle$ that belong to the open cell i.e.

$$v=v_{\pi_\alpha-\alpha} +
\frac{\sum_{\gamma\in \Delta_{\pp^{'}_u}} c_\gamma
e_{-\gamma}}{1!}v_{\pi_\alpha-\alpha}+\frac{(\sum_{\gamma\in
\Delta_{\pp^{'}_u}} c_\gamma
e_{-\gamma})^2}{2!}v_{\pi_\alpha-\alpha}+\ldots.$$

 Consider the curves $s^{'}_\gamma=0$ and $s^{'}_\delta=0$, where
 $\pi_\alpha-\gamma=L-E_1-E_2$ and $\pi_\alpha-\delta=E_2$ (the weights $\pi_\alpha-\alpha-\gamma^{'}$ and $\pi_\alpha-\alpha-\delta^{'}$
 as the weights of module $V_1$ correspond to the curves  $L-E_1-E_2$ and $E_2$. But we consider them as the curves on the del Pezzo
 surface where the curve $E_{\rk \Delta^{'}}$ is contracted).
It is easy to check (by a direct calculation with weights or by considering all the
curves in the linear systems to which belong the curves
$L-E_1-E_2$ and $E_2$) that only  $s_{\gamma^{'}}s_{\delta^{'}}$ and
$s_{\alpha}^{'}s_{\omega}^{'}$ for $\omega=\gamma+\delta-\alpha$,
contain the monomial $c_\gamma c_\delta$.

We obtain that
$v_{\pi_\alpha-\omega}=p^{'}_{\gamma^{'}\delta^{'}}c_{\gamma^{'}}c_{\delta^{'}}+\ldots$,
where we omitted the monomials that not containing  $c_{\gamma^{'}}$ and $c_{\delta^{'}}$
(here we used that  $e_{\gamma_i }$ for
$\gamma_i\in \Delta_{\pp^{'}_u}$ commutes pairwise).
 Then the coefficient by $c_\gamma c_\delta$
in the polynomial $\widetilde{p}_\mu(s^{'})$ is equal to

$$p^{'}_{\gamma^{'}\delta^{'}}p_{\alpha\omega}\frac{x_\alpha x_{\omega}}{s_\alpha s_{\omega}}+
n^{'}_{\gamma}n^{'}_{\delta}p_{\gamma \delta}\frac{x_\gamma
x_\delta}{s_\gamma s_\delta}. \eqno{(*)}$$

Let us note that $x_\alpha x_{\omega}$ and $x_\gamma x_\delta$ are linearly independent in the linear
system of conics on the del Pezzo surface
$X_{\Delta^{'}}$ (By contracting some  (-1)-curves we can obtain that these monomials define a pair  of lines
that are the elements of the linear system of conics passing through $4$
fixed points $e_{i_1},e_{i_2},e_{i_3},e_{i_4}$) thus the equation $(*)$ is not identically zero on the del Pezzo surface.

Let us prove the nontriviality of $\widetilde{q}_\nu(s^{'})$. In the case of
 $E_8$ without loss of generality we can assume that the cubic is of the form
  $C=L_0-E_1-E_2-E_3-E_4-E_5-E_7-2e_{\rk{\Delta}}$ (this can be obtained by the contraction $X_{\rk
\Delta^{'}}\longrightarrow \Bbb P^2$ since the Weyl group
$W^{'}$ is acting transitively on the weights of $V_3$ corresponding to the considered cubics).
In the case of  $E_7$ we consider the surface obtained by blowing up points $e=i$ for $i=1\ldots 7$,
$i\neq 6$. In this case such cubic is unique. Let us fix the
$(-1)$-curves $H_\alpha=E_7$, $H_{\mu}=E_2$,
$H_{\nu}=L_0-E_1-E_2$, $H_{\delta}=E_1$,
$H_{\gamma}=2L_0-E_1-E_2-E_3-E_4-E_5$. The required cubic can be represented as

$C=L_0-E_1-E_2-E_3-E_4-E_5-E_7-2e_{\rk{\Delta}}=E_7+(E_2+(L_0-E_1-E_2))
+(E_1+(2L_0-E_1-E_2-E_3-E_4-E_5)).$

\begin{center}
  \epsfig{file=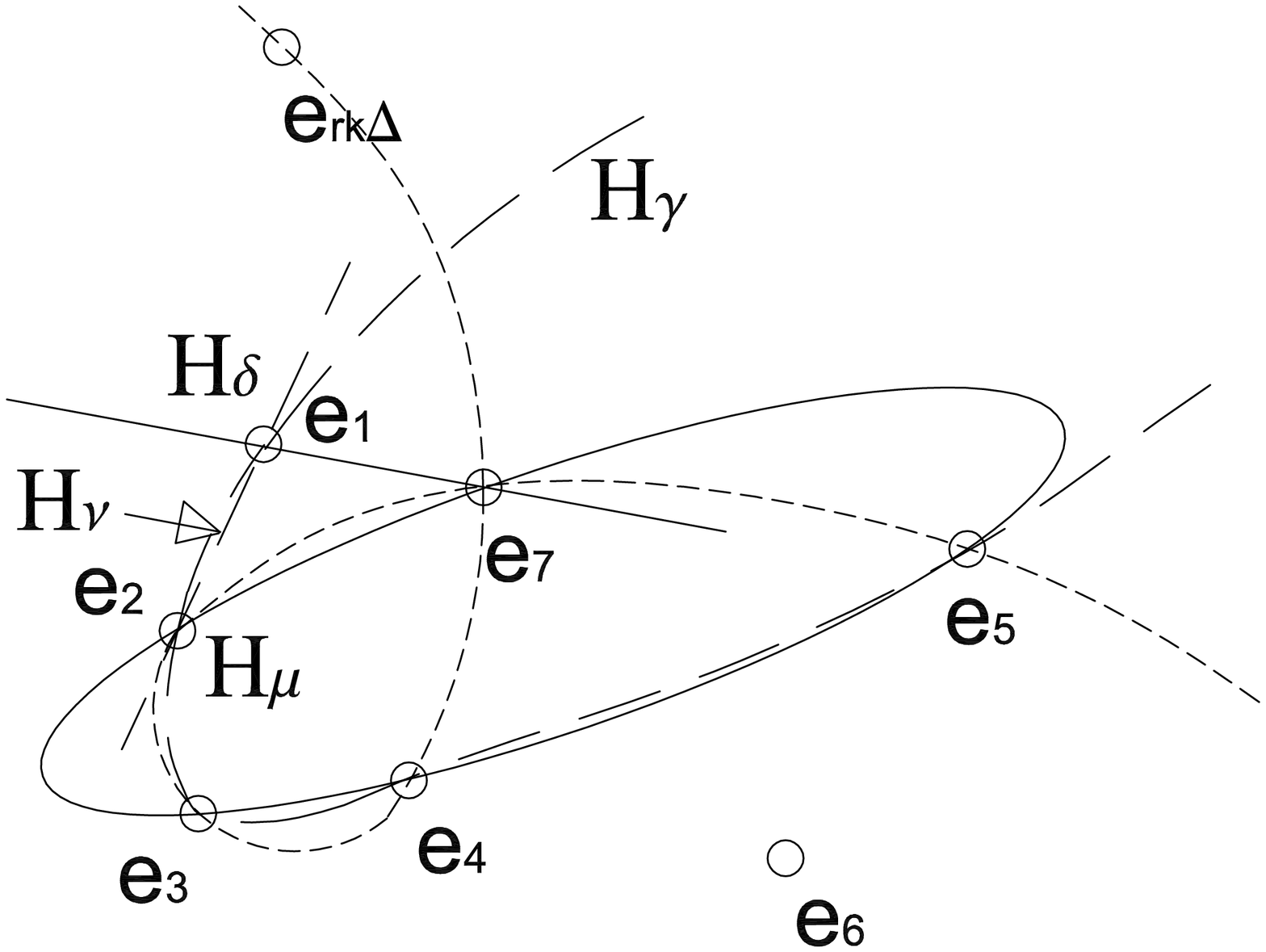,height=8cm,angle=90,clip=}

fig.4
\end{center}
Consider the expressions
 $s^{'}_{\vartheta_1}s^{'}_{\vartheta_2}s^{'}_{\vartheta_3}$,
from the variables  $\{c_\vartheta\}_{\vartheta \in \pp^{'}_u}$ and let use find those
that contain the monomial $c_\mu c_\nu c_\gamma c_\delta$.
 Let us notice that  $s^{'}_{\vartheta_i}$ define some linear system; more precisely
 when  $s^{'}_{\vartheta_i}=n^{'}_\zeta
 c_\zeta$ (где $\zeta \in \pp^{'}_u$) it is a curve from $X_{\Delta^{'}}$  corresponding to the weight
 $\pi-\alpha-\zeta$ (the same curve can be considered on $X_{\Delta^{''}}$,
 but then it will correspond to the weight $\pi_\beta^{'}-\zeta$). If $s^{'}_{\vartheta_i}=\sum
\limits_{\zeta_1+\zeta_2=\gamma+\delta-2\alpha}
p^{'}_{\zeta_1\zeta_2}c_{\zeta_1}c_{\zeta_2}$ (for
$\zeta_1,\zeta_1 \in \pp^{'}_u$) then the curve
$s^{'}_{\vartheta_i}=0$ considered as the curve on
$X_{\Delta^{''}}$ lie in the linear system generated by the pairs of
$(-1)$-curves $c_{\zeta_1}c_{\zeta_2}=0$ for all $\zeta_1,\zeta_1
\in \pp^{'}_u$ (the argument is analogous in the case when
$s^{'}_{\vartheta_i}$ is the cubic from the variables  $c_\zeta$).

 To do this we have to write down all linear systems corresponding to
 $s^{'}_{\vartheta_1},s^{'}_{\vartheta_2},s^{'}_{\vartheta_3}$,
 such that each linear system corresponding to  $s^{'}_{\vartheta_i}$
 contain the monomial that is a submonomial in  $c_\mu c_\nu c_\gamma
 c_\delta$, and the product of such monomials is equal to $c_\mu c_\nu c_\gamma
 c_\delta$.

All possible generators
$s^{'}_{\vartheta_1}s^{'}_{\vartheta_2}s^{'}_{\vartheta_3}$
are given on fig.5. (For the brevity on fig.5, by the sum $\sum
c_\gamma c_\delta$, we mean a sum of type $\sum
\limits_{\zeta_1+\zeta_2=\gamma+\delta-2\alpha}
p^{'}_{\zeta_1\zeta_2}c_{\zeta_1}c_{\zeta_2}$ in which we can find monomial
 $c_\gamma c_\delta$).

Now it is not difficult to calculate the coefficient by the monomial $c_\mu c_\nu
c_\gamma c_\delta$:

$$q_{\alpha(\mu\nu)(\delta\gamma)}p^{'}_{\delta^{'}\gamma^{'}}p^{'}_{\mu^{'}\nu^{'}}
\frac{x_\alpha x_{(\mu\nu)}x_{(\delta\gamma)}}{s_\alpha
s_{(\mu\nu)}s_{(\delta\gamma)}}+
q_{\alpha(\mu\gamma)(\nu\delta)}p^{'}_{\nu^{'}\delta^{'}}p^{'}_{\mu^{'}\gamma^{'}}
\frac{x_\alpha x_{(\mu\gamma)}x_{(\nu\delta)}}{s_\alpha
s_{(\mu\gamma)}s_{(\nu\delta)}}+$$

$$q_{\mu\gamma(\nu\delta)}p^{'}_{\nu^{'}\delta^{'}}n_{\mu^{'}}n_{\gamma^{'}}
\frac{x_\mu x_{\gamma}x_{(\nu\delta)}}{s_\mu
s_{\gamma}s_{(\nu\delta)}}+
q_{\mu\nu(\delta\gamma)}p^{'}_{\delta^{'}\gamma^{'}}n_{\mu^{'}}n_{\nu^{'}}
\frac{x_\mu x_{\nu}x_{(\delta\gamma)}}{s_\mu
s_{\nu}s_{(\delta\gamma)}}+$$

$$q_{\delta\nu(\mu\gamma)}p^{'}_{\mu^{'}\gamma^{'}}n_{\delta^{'}}n_{\nu^{'}}
\frac{x_\delta x_{\nu}x_{(\mu\gamma)}}{s_\delta
s_{\nu}s_{(\mu\gamma)}}+
q_{\delta\gamma(\nu\mu)}p^{'}_{\nu^{'}\mu^{'}}n_{\delta^{'}}n_{\gamma^{'}}
\frac{x_\delta x_{\gamma}x_{(\nu\mu)}}{s_\delta
s_{\gamma}s_{(\nu\mu)}}\eqno (**)
$$
 (for brevity by
$(\zeta_1\zeta_2)$ we mean the weight $\zeta_1+\zeta_2-\alpha$)

The coefficients of the fractions are equal to  $(\pm1)$. To prove this let us notice
that $e_{-\vartheta}$ for $\vartheta\in
\pp_u$ commutes pairwise  ($[e_{-\vartheta_1};e_{-\vartheta_2}]\neq 0$ only in the case of $E_8$,
but in this case they do not commute only when
$\vartheta_1+\vartheta_2=\pi_\alpha$, that is not the case since such curves
correspond to the cubic curves passing through the points $e_1\ldots e_7$).
Any two monomials in the expression   (**)
have a common variable $x_\vartheta$. Let us do one of the calculations that are equivalent to each other:

$$q_{\mu\gamma(\nu\delta)}v_{(\mu\gamma\nu\delta)}=e_{-(\nu\delta)}e_{-\mu}e_{-\gamma}v_{\pi_\alpha}=
n^{-1}_{\gamma}n^{-1}_{\mu}e_{-(\nu\delta)}[e_{-\alpha},e_{-\mu^{'}}][e_{-\alpha},e_{-\gamma^{'}}]v_{\pi_\alpha}=$$
$$
n^{-1}_{\gamma}n^{-1}_{\mu}e_{-(\nu\delta)}(e_{-\alpha}e_{-\mu^{'}}-e_{-\mu^{'}}e_{-\alpha})
(e_{-\alpha}e_{-\gamma^{'}}-e_{-\gamma^{'}}e_{-\alpha})v_{\pi_\alpha}=$$
$$
-n^{-1}_{\gamma}n^{-1}_{\mu}e_{-(\nu\delta)}e_{-\alpha}e_{-\mu^{'}}e_{-\gamma^{'}}e_{-\alpha}v_{\pi_\alpha}=
-n^{-1}_{\gamma}n^{-1}_{\mu}p^{'}_{\mu^{'}\gamma^{'}}e_{-\alpha}e_{-(\nu\delta)}e_{-(\mu\gamma)}v_{\pi_\alpha}=$$
$$=-n^{-1}_{\gamma}n^{-1}_{\mu}p^{'}_{\mu^{'}\gamma^{'}}q_{-\alpha(\nu\delta)(\mu\gamma)}v_{(\mu\gamma\nu\delta)}$$
where we used that $e_{-\gamma^{'}}v_{\pi}=0$,
$e_{-\alpha}e_{-\gamma}v_{\pi}=0$ and applied Lemma \ref{[ee]}.

\begin{center}
  \epsfig{file=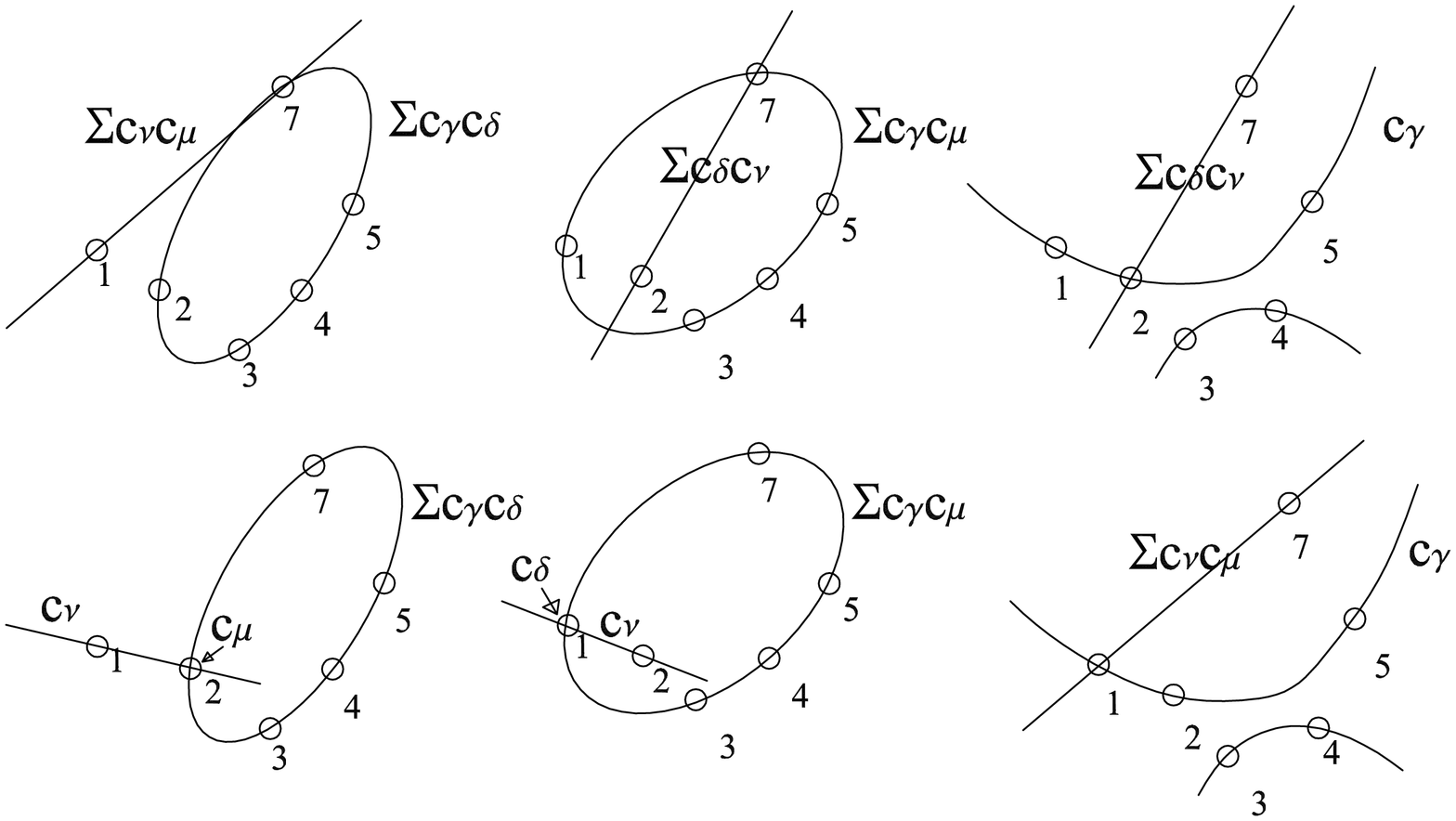,height=12cm,angle=90,clip=}

fig.5
\end{center}

From the following elementary lemma it follows that the expression   $(**)$ cannot be identically equal to zero.

\begin{lemma} Consider the points
$e_i$ that are in the general position. Let $l_{ij}$ be a line passing through
the points $e_i, e_j$ and $q_{ij}$ is a conic passing through the points
 $e_i, e_j, e_3, e_4, e_5$. $E_i$ as before denote the exceptional curves
that are the preimages of $e_i$.
Let us normalize the equations in such way that $l_{ij}(e_{\rk
\Delta})=q_{ij}(e_{\rk \Delta})=1$. Then the restriction of the expression
$$l_{12}q_{26}E_2-l_{12}q_{16}E_1+l_{16}q_{12}E_1-l_{16}q_{26}E_6+l_{26}q_{16}E_6-l_{26}q_{12}E_2$$
on the curve $l_{\Delta}=L-E_1-e_{\rk \Delta}$ is not equal to zero
(let us note that the line $l_{\Delta}$ considered as the curve on
$X_{\Delta^{'}}$ has the degree $2$ with respect to $K_{\Delta^{'}}$).
\end{lemma}
\begin{proof}
Let us notice that the curve  $l_{\Delta}$ intersect each curve
$E_1,q_{16},l_{26},q_{12}$ in a single point that we denote by
 $a,b,c$ and $d$ correspondingly, and with the curve $q_{26}$ in two points (see fig.7).

\begin{center}
  \epsfig{file=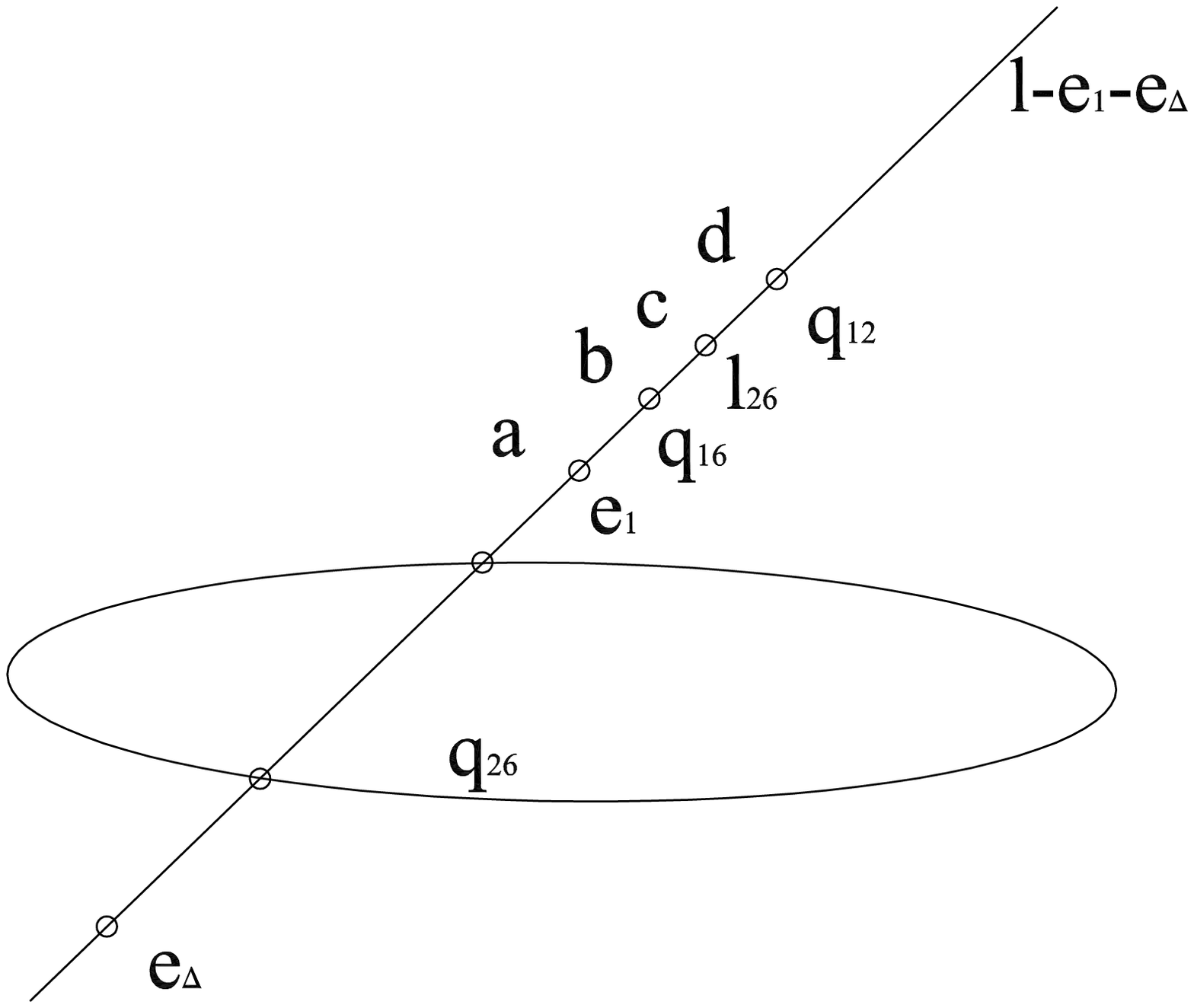,height=6cm,angle=90,clip=}

fig.6
\end{center}

 Since the lines  $l_{12}$ and $l_{16}$ do not intersect  $l_{\Delta}$ the curves $l_{16}q_{26}E_6$ and
$l_{12}q_{26}E_2$ define the same section on the line bundle  $\mathcal
O(2)$ on the curve $l_{\Delta}$, that is also equal to
$q_{26}|_{l_{\Delta}}$. Indeed all three curves intersect
${l_{\Delta}}$ in two fixed points and all the three sections are equal to  $1$ in the point
 $e_{\rk \Delta}$. From this we obtain that
 $(l_{16}q_{26}E_6-l_{12}q_{26}E_2)|_{l_{\Delta}}=0$. Without the loss of generality we assume that all points
lie in the affine chart $x=0$. Then from the normalizing condition in the point  $0$ (i.e.
$l_{ij}(e_{\rk \Delta})=q_{ij}(e_{\rk \Delta})=1$) our expression can be written as
$$-l_{12}q_{16}E_1+l_{26}q_{16}E_6-l_{26}q_{12}E_2+l_{16}q_{12}E_1=$$
$$-\frac{1}{ab}(x-a)(x-b)+\frac{1}{bc}(x-b)(x-c)-\frac{1}{cd}(x-c)(x-d)+\frac{1}{da}(x-d)(x-a)=$$
$$= \frac{(b-d)(c-a)}{abcd}x^2,
$$
thus we get that the last expression is equal to zero iff   $a=c$ or $b=d$. But the first means that the points
 $e_1,e_2,e_3$ lie on the same line  and the second states that the conics $q_{12}$ and $q_{16}$ intersect in $5$
points that implies that  $6$ belong to the same conic that contradicts the generality of position.\end{proof}

The assertion $2)$ from the Proposition \ref{P-Q} follows easily from the fact that
 $p_{\mu}(x)=0$ for $\mu \in supp(V_1)$ defines
$G^{'}/P^{'}$, and from the fact that on the quotient
$X_{\Delta^{'}}=T\backslash \!\! \backslash s^{'}s^{-1}\mathcal T^{\, '}$ these
equations define conics intersecting transversally in a single point  $e_{\rk
\Delta}$.
\end{proof}

Let us prove the following proposition

\begin{proposition}\label{blowup} Let  $\mathcal T=p_{\pi_{\alpha}}^{-1}\sigma_{0*}^{-1}(s^{'}s^{-1}\mathcal T^{\,
 '})$ be a composition of proper transform of $s^{'}s^{-1}\mathcal T^{\,
 '}$ with respect to the blow up $\sigma_0$ and preimage of the quotient  morphism
  $p_{\pi_{\alpha}}$. Then the quotient $T\backslash \!\!\backslash\mathcal T$ is the del Pezzo
  surface  $X_{\Delta}$ obtained by the blow up of del Pezzo surface
  $X_{\Delta^{'}}$ in the point $e_{\rk_\Delta}=p_{T^{'}}(s)$.
\end{proposition}
\begin{proof} Consider the quotient $\widetilde{\mathcal T}=\lambda\backslash \!\!\backslash\mathcal T$.
It is easy to see that the proper transform of  $s^{'}s^{-1}\mathcal
T^{\, '}$ with respect to the blow up $Bl(\Bbb P(V_1),G^{'}/P^{'})$ of the space
$\Bbb P(V_1)$ in the flag subvariety $G^{'}/P^{'}$ (contraction morphism we denote by $\sigma_0$).

Consider the following commutative diagram
$$
\xymatrix{ \widetilde{X} \ar[d]_{\widetilde{\sigma}} &
\ar[l]^{p_{T^{'}}}\ar[r]\ar[d]_{\sigma} \widetilde{\mathcal T}&
\ar[d]_{\sigma_0}Bl(\Bbb P(V_1),G^{'}/P^{'}) &
 \\
X_{\Delta^{'}}       & \ar[l]^{p_{T^{'}}}  s^{'}s^{-1}\mathcal
T^{\, '}
   \ar[r]    & \Bbb P(V_1)        & \\
}
$$
Where the surface $\widetilde{X}$ is defined as the quotient
$T^{'}\backslash \!\!\backslash\widetilde{\mathcal T}$.
We also notice that since $s^{'}s^{-1}\mathcal T^{\, '}\cap
G^{'}/P^{'}=T^{'}s^{'}$ we have an isomorphism:
$$\widetilde{\mathcal T}\setminus \sigma_0^{-1}(T^{'}s^{'})=s^{'}s^{-1}\mathcal T^{\, '}\setminus T^{'}s^{'}. $$
Pssing to quotients we get the isomorphism
$$\widetilde{X}\setminus \widetilde{\sigma}^{-1}(e_{\rk \Delta})=X_{\Delta^{'}} \setminus (e_{\rk \Delta}).  $$

This implies that  $\widetilde{\sigma}$ is the composition
of the blow ups with the centers outside  $X_{\Delta^{'}} \setminus
(e_{\rk \Delta})$ or in other words   $\widetilde{X}$
is obtained by the blow up of $X_{\Delta^{'}}$ in the subscheme with the support in
$e_{\rk \Delta}$ (см. \cite{har}) defined by the ideal $\mathcal
J_{e_{\rk \Delta}}$. Our aim is to show that
$\widetilde{\sigma}=X_{\Delta}$ or in other words that the ideal
$\mathcal J_{e_{\rk \Delta}}$ is  maximal.
Applying th universal property of the blow up (cf.\cite[Prop. 7.14,
Corr. 7.15]{har}) to the right hand of the diagram we get that the ideal
 defining the blow up is equal to the restriction of the ideal  $\mathcal
J_{G^{'}/P^{'}}$ on $s^{'}s^{-1}\mathcal T^{\, '}$ in other words
$\mathcal J_{0}=\mathcal J_{G^{'}/P^{'}}|_{s^{'}s^{-1}\mathcal
T^{\, '}}$.

Let us notice that we can check our proposition locally. Consider the
$T^{'}$-invariant affine chart  $U=Spec(\mathcal A)$ on
$s^{'}s^{-1}\mathcal T^{\, '}$ containing the orbit $T^{'}s^{'}$.
The morphism $p_{T^{'}}$ maps it to the affine chart
$T^{'}\backslash \!\!\backslash U=Spec(\mathcal A^{T^{'}})$ containing the point
$e_{\rk \Delta}$. Shrinking the map we assume that the divisors defined by
 $p_{\mu}(x)$ are principal and defined by the regular functions
  $f_\mu\in \mathcal A^{T^{'}}=\mathcal O(U)$. The inverse images of these
  functions are  invariant functions on $U$ with the same divisors  of zeros
  as $p_{\mu}(x)$ (Since the zeros  of
$p_{\mu}(x)$ are $T^{'}$-invariant), thus they also define
$\mathcal J_{0}|_{U}$.

 Lt us apply the universal property of the blow up to the left hand of the diagram.
 Reformulating  \cite[Prop. 7.14, Corr. 7.15]{har}, we
get that $\mathcal J_{0}$ is equal to the ideal of $\mathcal A$,
generated by the ideal  $\mathcal J_{e_{\rk \Delta}}\subset \mathcal
A^{T^{'}}$. Thus $\mathcal J_{0}^{T^{'}}=\mathcal J_{e_{\rk
\Delta}}$ (cf.\cite{kraft}). As we have seen the ideal  $\mathcal
J_{e_{\rk \Delta}}=\mathcal J_{0}^{T^{'}}$ contain the functions
$f_\mu$, but the zeros of these functions coincide with the curves $p_{\mu}(x)=0$ on
the surface $X_{\Delta^{'}}$. Any two curves from this set intersect transversally in the point
 $e_{\rk \Delta}$ (though it is sufficient to know the intersection of the curves  $L-E_1-e_{\rk
\Delta}$ и $L-E_2-e_{\rk \Delta}$). This implies that the ideal $\mathcal
J_{e_{\rk \Delta}}$ coincide with the maximal ideal of the point $e_{\rk
\Delta}$.
\end{proof}

For the case of $E_8$ let us prove an analog of Proposition \ref{blowup}

\begin{proposition}\label{blowupE_8} Let $G$ be a group of type $E_8$, $p_{\pi_{\alpha}}$
  is the quotient morphism with respect to the action of $\lambda$, and $\sigma_0$  is the morphism
 $\lambda\backslash \!\! \backslash((G/P)^{ss}_{\varepsilon}\cap D_{h_{\pi_\alpha}})\longrightarrow \Bbb P(V_1)$,
  defined by the projection $p_0$ on $\Bbb P(V_1)$.
Assume we have an embedding of the universal torsor $\mathcal T^{\,
 '}$ over del Pezzo surface of the degree $2$  in $G^{'}/P^{'}\subset \Bbb P(V_1)$, that is not contained in any
 $T^{'}$-invariant hyperplane.
Let us fix sufficiently general points $s,s^{'}\in G^{'}/P^{'}$ and denote by
 $\mathcal T=p_{\pi_{\alpha}}^{-1}\sigma_{0*}^{-1}(s^{'}s^{-1}\mathcal T^{\, '})$ the composition of the proper transform of $s^{'}s^{-1}\mathcal T^{\,
 '}$ with respect to the weighted blow up  $\sigma_0$ and its  inverse image of the quotient morphism by
  $\lambda$. Then the quotient $T\backslash \!\!\backslash\mathcal T$ is the surface $X_{\Delta}$
  obtained by blowing up the del Pezzo  surface $X_{\Delta^{'}}$ in  $e_{\rk_\Delta}=p_{T^{'}}(s)$.
\end{proposition}
\begin{proof} For the proof of the proposition let us consider the $T^{'}$-invariant affine chart
$(G/P)^{ss}_{\varepsilon}\cap D_{h_{\pi_\alpha}}\cap \mathcal U_\alpha$. According to Proposition \ref{bl},
$\lambda\backslash \!\! \backslash((G/P)^{ss}_{\varepsilon}\cap D_{h_{\pi_\alpha}}\cap \mathcal U_\alpha)\longrightarrow \Bbb P(V_1)\cap \mathcal U_\alpha$
is isomorphic to the weighted blow up  $\sigma_0$ of the variety  $\Bbb P(V_1)\cap \mathcal U_\alpha$ in $P^{'-}_uP'/P'$. The preimage $\sigma_0^{-1}(P^{'-}_uP'/P')$ is identified with $P^{'-}*_{L'}\Bbb P_{wt}(\Bbb K v_\alpha\oplus V^{'}_2)$.
Since the stabilizer of the action of   $T^{'}$ in the point $s^{'}$ is trivial the variety $\sigma_0^{-1}(T^{'}s^{'})$
is identified with $T^{'}s^{'}\times\Bbb P_{wt}(\Bbb K v_\alpha\oplus V^{'}_2)$. In particular
 $T^{'}\backslash\!\!\backslash\sigma_0^{-1}(T^{'}s^{'})\cong \Bbb P_{wt}(\Bbb K v_\alpha\oplus V^{'}_2)$.

 By the induction assumption we have the embedding $\mathcal T^{\, '}\subset  G^{'}/P^{'}$. As it was proved $s^{'}s^{-1}\mathcal T^{\, '}$
intersects  $G^{'}/P^{'}$ transversally in the orbit $T^{'}s^{'}$. The quotient $T^{'}\backslash\!\!\backslash s^{'}s^{-1}\mathcal T^{\, '}$
 a smooth del Pezzo surface of the degree $2$. This implies that in the neighbourhood  of the point
  $p_{T^{'}} (s^{'})$ the germ of this surface is the complete intersection and is defined by the system of equations
  $(f_0,\ldots, f_l)$, where
 $l=\dim V_1-\dim T^{'}-4$.
 Since $T^{'}$-variety  can be covered by $T^{'}$-invariant affine maps we can assume that
 $s^{'}\in U_{s^{'}}$ for some $T^{'}$-invariant affine chart and  $f_i\in \Bbb K[{T^{'}}\backslash\!\!\backslash U_{s^{'}}]= \Bbb K[U_{s^{'}}]^{T^{'}}$. From the above we see that the germ of the torsor $s^{'}s^{-1}\mathcal T^{\, '}$
 is aldo defined by the considered regular sequence $(f_0,\ldots, f_l)$.

Let  $\widetilde{\mathcal T}=\lambda\backslash \!\!\backslash\mathcal T$
be the proper transform $s^{'}s^{-1}\mathcal
T^{\, '}$ with respect of the weighted blow up $\sigma_0$.
Our aim is to show that the quotient  $\widetilde{\mathcal T}\cap \sigma_0^{-1}(s^{'})\cong \widetilde{\mathcal T}\cap \Bbb P_{wt}(\Bbb K v_\alpha\oplus V^{'}_2)_{s^{'}}$ by the action of torus $T^{'}$,
is isomorphic to $\Bbb P^1$.

On the considered  $T^{'}$-invariant chart $U_{s^{'}}$  we can assume that the variety
 $P^{'-}_uP^{'}/P^{'}$ is defined by the vanishing of the equations  $p_{\pi_\alpha-\gamma_i}, q_{\alpha}$.
  It is not difficult to check that these equations after restriction to the affine chart
   $v_{\pi_\alpha}\neq 0$ define equations corresponding to vanishing of the components with the weights
    $s_\alpha (\pi_\alpha-\gamma_i)$ and $-\alpha$ correspondingly for the flag variety $G/P$.
     In particular if we restrict  on the quotient  $T\backslash \!\! \backslash s^{'}s^{-1}\mathcal T^{\,
 '}$,  the equations $p_{\pi_\alpha-\gamma_i}=0$  define on the del Pezzo
 surface of degree $2$  the conics passing through the point $e_8$, and
 the equation $q_{\alpha}=0$ defines a cubic with the double point in $e_8$.

 Let us recall that the weighted blow up
   $\sigma_0:\lambda\backslash \!\! \backslash((G/P)^{ss}_{\varepsilon}\cap D_{h_{\pi_\alpha}}\longrightarrow \Bbb P(V_1)$
   is defined as the subvariety in $\Bbb K^{\dim V_1-1}\times \Bbb P_{wt}(\Bbb K v_\alpha\oplus V^{'}_2)$ by the system of equations:

\hspace{20ex}  $ \left \{ \begin
{array}{cl}
p_{\pi_\alpha-\gamma_i}c_{\gamma_j}=p_{\pi_\alpha-\gamma_j}c_{\gamma_i} \\
q_{\alpha}c^2_{\gamma_j}=p^2_{\pi_\alpha-\gamma_j}c_{\pi_\alpha}       \\
\end{array} %\ \ \ \ \ (*_{BL})
\right.$

Let us represent the equations  $f_i$ defining  $s^{'}s^{-1}\mathcal T^{\, '}$ as the sum of linear part from the variables
  $p_{\pi_\alpha-\gamma_j}$ (that we assume to be nonzero) and a remainder that consists of the monomials
  of the bigger weight with respect to  the weighted grading in consideration:
$$f_i=\sum a_{ij}p_{\pi_\alpha-\gamma_j}+f_i^{>1},
$$
where the coefficients $a_{ij}$ depend only on variables $b_\gamma$.
Moreover we assume that $a_{ij}$ are written as the polynomials from $(b_\gamma-(b_\gamma)_{s'})$
with the constant terms $\overline{a}_{ij}$ where $(b_\gamma)_{s'}$ are the coordinates of the point $s'$ in $P'_uP'/P'$.
Our aim is to invesigate which variety is the defined by the equation  $f_i=0$ in the fiber over the point $s^{'}$, where all the coordinates $p_{\pi_\alpha-\gamma_j}, q_{\alpha}$ vanish.
In the chart $c_{\gamma_1}\neq  0$ we get:
$$
f_i=p_{\pi_\alpha-\gamma_1}\sum a_{ij}\frac{c_{\gamma_j}}{c_{\gamma_1}}+p^2_{\pi_\alpha-\gamma_1}\tilde{f_i}^{>1},
$$
 where $\tilde{f_i}^{>1}$ is the polynomial from $p_{\pi_\alpha-\gamma_j}$. Dividing the expression for $f_i$ by $p_{\pi_\alpha-\gamma_1}$,
 setting $p_{\pi_\alpha-\gamma_1}=0$ and considering the homogenous equation with the variables $c_{\pi_\alpha-\gamma_j}$, we
 obtain that the intersection of the proper transform of the variety  $f_i=0$ and a fiber $\sigma_0^{-1}(s^{'})\cong \Bbb P_{wt}(\Bbb K v_\alpha\oplus V^{'}_2)$
 is defined in the fiber $\Bbb P_{wt}(\Bbb K v_\alpha\oplus V^{'}_2)$ by the equation
  $$
  \overline{f}_i=\sum \overline{a}_{ij}{c_{\gamma_j}}.
  $$
  In a similar way we get that the equation  $$f=q_{\alpha}+\sum d_{ij}p_{\pi_\alpha-\gamma_i}p_{\pi_\alpha-\gamma_j}+f_i^{>2}$$
in the fiber over the point $s^{'}$ isomorphic to $\Bbb P_{wt}(\Bbb K v_\alpha\oplus V^{'}_2)$ defines a hyperplane $f=c_{\pi_\alpha}+\sum \overline{d}_{ij}c_{\gamma_i}c_{\gamma_j}$, where $\overline{d}_{ij}$ is defined similar to $\overline{a}_{ij}$.

Let us show that the matrix that consists of the homogenous components   $\overline{f}_i$ of the degree $1$ (of the polynomials $f_i$)
has a rank equal to
$\dim \Bbb P_{wt}(\Bbb K v_\alpha\oplus V^{'}_2)-2$. Let us notice that on the surface  $T^{'}\backslash \!\!\backslash s^{'}s^{-1}\mathcal T^{\, '}$ the equation $q_{\alpha}=0$ define the cubic with a double point in $e_8$. In particular this implies that the
differential $dq_{\alpha}=0$ after the restriction on the tangent space to the torsor
in the point $s'$.
 Thus we have the linear dependance between $dq_{\alpha}$ and $df_0,\ldots df_{l}$.
 In other words we can assume that   $df_1,\ldots df_{l}$ and $dq_{\alpha}$ are linear dependent,
 and a linear part of the polynomial $df_0$ is proportional to $q_\alpha$.
Let us notice that  the restrictions of  $df_i$ to the subspace $\mathfrak t's'\subset T_{s'}(G'/P')$  are zero.
 In the space generated by the forms $db_\delta$, consider the space of forms that are zero on the subspace  $\mathfrak t's'$.
  Let us chose a basis $d\widetilde{b}_k$ in this space, where $k=\dim G'/P'-\dim T'$.
Let us write down $df_i$ in the basis $dp_{\pi_\alpha-\gamma_j},dq_\alpha,d\widetilde{b}_k$ for $i\geqslant 1$.
 We get a matrix in which in the first  $\dim \Bbb(V_1)-\dim G'/P'-2$ rows we have elements $\overline{a}_{ij}$.
 Since the matrix formed by  $df_1,\ldots df_{l}$  has the rank $\dim V_1-\dim T'-3$, than the matrix of the coefficients $\overline{a}_{ij}$  has the rank not less than  $\dim \Bbb P(V_1)-\dim T'-3-(\dim G'/P'-\dim T')=\dim V'_2-2$.
  Let us notice that on the surface  $T^{'}\backslash \!\!\backslash s^{'}s^{-1}\mathcal T^{\, '}$ the equations $p_{\pi_\alpha-\gamma_i}=0$
define the conics passing through  $e_8$ and intersecting transversally there. This implies that the differentials $dp_{\pi_\alpha-\gamma_1}$ и $dp_{\pi_\alpha-\gamma_2}$ generate the tangent space to the surface in  $e_8$,
 that gives the linear independence of  $dp_{\pi_\alpha-\gamma_1},dp_{\pi_\alpha-\gamma_2}, df_1,\ldots df_{l}$
 Thus the rank of the matrix  $\overline{a}_{ij}$ is not bigger than $\dim V^{'}_2-2$ and should be equal to this number.

 The preceding arguments show that the equations $\overline{f}_1,\ldots,\overline{f}_l$ define in the space
 $\Bbb P_{wt}(\Bbb K v_\alpha\oplus V^{'}_2)$ the weighted projective subspace
$\Bbb P(2,1,1)$. The equation $\overline{f}_0=c_{\pi_\alpha}+\sum \overline{d}_{ij}c_{\gamma_i}c_{\gamma_j}$
defines a smooth rational curve,  since it is the image of the conic   $c_0^2+\sum \overline{d}_{ij}c_{\gamma_i}c_{\gamma_j}=0$
with respect to  the quotient  by the action of the involution $\Bbb P^2\rightarrow \Bbb P(2,1,1)$ that reverses sign of the first coordinate.
 Thus we get that the fiber over the point $s'$ consists of the smooth rational curve
 (A simple calculation in the local coordinates shows that the obtained surface is smooth, see next remarks).

 That we get that
$T\backslash \!\!\backslash\mathcal T$ is the surface $X_{\Delta}$, obtained by the blow up of del Pezzo surface
 $X_{\Delta^{'}}$ in the point $e_{\rk_\Delta}=p_{T^{'}}(s)$.
\end{proof}

\begin{remark} To define the intersection of the fiber  $\sigma_0^{-1}$ and a proper transform of the hypersurface $f=0$
we can argue in the following way. Let us use the realization of the blow up from the Remark \ref{wtblrem}. Let $\overline{f}$
 be a component of $f$ that consists of the monomials with the minimal grading   $d$ in the variables $z_i$:
  $$\overline{f}=\sum_{\sum a_j i_j=d} p_{i_1\ldots i_l}(w)z_1^{i_1}\ldots z_l^{i_l}$$
Then we have $\tau^*f=\sum_{\sum a_j i_j=d} p_{i_1\ldots i_l}(w)z_1^{a_1i_1}\ldots z_l^{a_li_l}+f^{>d}$,
where the component  $f^{>d}$ has degree strictly bigger than  $d$. The intersection of $\sigma_0^{-1}$ and of the proper transform $f=0$
  for the ordinary blow up  $Bl$ is defined by the homogeneous  component of minimal degree $\overline{\tau^*f}=\sum_{\sum a_j i_j=d} p_{i_1\ldots i_l}(w)z_1^{a_1i_1}\ldots z_l^{a_li_l}$ ( In terms of $\proj$ this is the projection on  $\mathcal J^{d}/\mathcal J^{d+1}$,
 where $\mathcal J=(z_1,\ldots,z_l)$).
 To get the equation of the intersection of  $\sigma_0^{-1}$ and a proper transform of the hypersurface $f=0$
 for the weighted blow up  $Bl$, consider the quotient  $(\overline{\tau^*f}=0)$ by $\tau_{\Bbb P}$.
 It is easy to see that it is defined by the equation   $\sum_{\sum a_j i_j=d} p_{i_1\ldots i_l}(w)z_1^{i_1}\ldots z_l^{i_l}=0$.
 In more invariant terms the latter equation is the projection on the component of the degree  $d$ in the algebra $\proj_{wt} \bigoplus \mathcal J^{i}/\mathcal J^{i+1}$,
 where we consider the weighted grading where the generators $\mathcal J=(z_1,\ldots,z_l)$ carry the weights
 $(a_1,\ldots,a_l)$.
\end{remark}

\begin{remark} The preceding remark provides us an illustration of the situation described in Proposition \ref{blowupE_8}.
Consider the blow up in zero point of the cone  $z_0^2=z_1z_2$, lying in the affine space $\Bbb K^3$. This blow up give a resolution of the singularity.
 And over the single point we have a conic  $\xi_0^2=\xi_1\xi_2$, that is a  $(-2)$-curve. Now consider the involution on  $\Bbb K^3\times \Bbb P^3$  acting on the coordinates  $z_0$ and $\xi_0$ by multiplication with $-1$.   The quotient of the blow up
 $\Bbb K^3$ in zero by the action of this involution is the weighted blow up defined  in $\Bbb K^3\times \Bbb P(2,1,1)$ by the equations

 \hspace{20ex}  $ \left \{ \begin
{array}{cl}
z_0\xi_i^2=z_i^2\xi_0 \\
z_1\xi_2=z_2\xi_1       \\
\end{array}
\right.$

The quotient of the considered blown up cone is the proper transform with respect to the weighted blow up of the surface defined by the equation
$z_0=z_1z_2$.
Let us notice that after taking the quotient by this involution the conic $\xi_0^2=\xi_1\xi_2$, that is exceptional curve over zero
become a two sheeted covering of the rational curve in  $\Bbb P(2,1,1)$ defined by the equation $\xi_0=\xi_1\xi_2$.
By the projection formula from the intersection theory the latter curve is the $(-1)$-curve.
Besides it is easy to see that the image of the blown  up cone is the smooth surface that is the ordinary  blow up of affine plane.
\end{remark}

From this proposition it is easy to get the second assertion of the theorem.
Indeed let $\omega$ is the nieghbour weight to $\pi_\alpha$,
then $\omega \in V_1$. The curve  $E_{\pi_\alpha}$ is the blow up of the point
 $e_{\rk \Delta}$ on the surface $X_{\Delta^{'}}$,
the curves  $E_\omega$ are $(-1)$-curves  of the surface
$X_{\Delta^{'}}$ that do not intersect with $E_{\pi_\alpha}$,
since they do not pass throught  $e_{\rk \Delta}$.

The curves  $E_{\mu}$ and $E_{\nu}$ correspond to nonadjacent weights
are defined by the equations $p_{\mu}=0$ and $q_{\nu}=0$ where $\mu \in V_2$,
$\nu \in V_3$ are the curves of the degree $2$ and  $3$ with respect to
$-K_{X_{\Delta^{'}}}$. The intersection indexes with the curve
$E_{\pi_\alpha}$ are equal to $1$ and $2$,  since the first curve passes through the point
 $e_{\rk \Delta}$, and the second has a double point in  $e_{\rk
\Delta}$.

Now we are able to prove the main theorem  (in this step we give  the argument of \cite{skor}).

We have constructed embedding  $\mathcal T\subset (G/P)^s$. Consider the weight hyperplanes
 $H_\omega=\{\langle v\rangle \in \Bbb
P(V(\pi_\alpha))\  |\
 v_\omega=0\}$.
 By construction the divisor $E_\omega=\tau(H_\omega\cap \mathcal T)\subset X_{\Delta}$ is the
 $(-1)$-curve on $X_{\Delta}$, besides all $(-1)$-curve can be obtained in such way. It is well known
(\cite{Manin}) that $E_{\omega}$ generate $\Pic(X_\Delta)$.

For  $\widehat{\mathcal T}$ let us take the preimage in
$V(\pi_\alpha)$ of the torsor $\mathcal T$ with respect to the projectivization
$V(\pi_\alpha)\setminus \{0\}\longrightarrow \Bbb
P(V(\pi_\alpha))$.

Let $(G/P)^{s}_T$  be the set of stable points for the sheaf
 $i^*\mathcal O(1)$  and the action of $T$ (where $i:G/P\subset
\Bbb P(V(\pi_\alpha))$. Consider the subset  $(G/P)^{sf}\subset
(G/P)^{s}$ of the points with the stabilizer  $Z(G)$. As it will be showed
the set $(G/P)^{sf}$ has the codimension $>1$ in $G/P$, thus we have
$\Pic((G/P)^{sf})=\Pic (G/P)=\Xi(P)=\Bbb Z \pi_\alpha$ (cf.
\cite{popov}). Let us denote by   $\widehat{G/P}$  the preimage of
$(G/P)^{sf}$ with respect to the projectivization $V(\pi_\alpha)\setminus
\{0\}\longrightarrow \Bbb P(V(\pi_\alpha))$.

  Thus it is easy to see that
$\Pic(\widehat{G/P})=0$. Since there are no  regular
 $\widehat{T}$-semiinvariant invertible functions on $\widehat{G/P}$,
 then from the exact sequence we get that the torsor $(G/P)^{sf}$ is universal.
 The images of the hyperplanes
$\widehat{H_\omega}=\{ v \in \Bbb V(\pi_\alpha)\  |\
 v_\omega=0\}$ by the quotient morphism generate
 $\Pic (\widehat{T}\backslash \!\!\backslash(\widehat{G/P}))\cong
 \Pic_{\widehat{T}}(\widehat{G/P})$.

Let us show that the torsor $\widehat{\mathcal T}$ is also universal.
Using the exact sequence  $(CTS)$
for the torsors  $\widehat{\mathcal T}$ and $\widehat{G/P}$ we obtain a commutative diagram where
the right vertical arrow is induced by the embedding $\widehat{\mathcal
T}\hookrightarrow\widehat{G/P}$:

$$
\xymatrix{ \ar[r]  &  \Xi(\widehat{ T})\ar[r]\ar[d]_{\parallel} &
\Pic (\widehat{T}\backslash \!\!\backslash(\widehat{G/P}))\ar[r]\ar[d] &
 \\
\ar[r]       & \Xi(\widehat{T}) \ar[r]       & {\Pic}(X_\Delta)
\ar[r]&
 \\
}
$$

But we have seen that the quotients  $\widehat{T}\backslash \!\!\backslash
\widehat{H_\omega}$ of the hypersurfaces $\widehat{H_\omega}$ generate
${\Pic}(X_\Delta)$. Thus we get that the vertical map is surjective that implies that all maps in commutative diagram
are isomorphisms. That implies the universality of the torsor $\widehat{\mathcal
T}$.

In the case when the root system of type  $E_8$ the main theorem is formulated by the following way.

\begin{theorem} Let $\Delta$ be the root system of type $E_8$.
Consider the embedding of  $G/P$ in the projectivization of adjoint representation
 $\Bbb P(V(\pi_\alpha))=\Bbb P(\gg)$. For the left action of maximal torus $T$ on
  $G/P$ let us denote by
$(G/P)_{\varepsilon}^s$ the set of stable points with respect to the
$G$-linearized line bundle  $\mathcal O(2)|_{G/P}\otimes
k_{-\pi_\alpha}$. Let $\tau:(G/P)_{\varepsilon}^s\longrightarrow
T\backslash \!\!\backslash(G/P)^s$ be a quotient morphism. Consider a sufficiently general
del Pezzo $X_{
\Delta}$ of degree $1$. There exists an embedding  $\imath$ of the surface
$X_{ \Delta}$  in the quotient $T\backslash \!\!\backslash(G/P)_{\varepsilon}^s$, such that for the torsor
 $\mathcal T=\tau^{-1}(\imath (X_{ \Delta}))$
the following conditions hold:
\begin{itemize}
 \item[$\bullet$] Consider a hyperplane $H_\omega=\{\langle v\rangle \in \Bbb P(V(\pi_\alpha))\  |\
 v_\omega=0\}$. The divisor $E_\alpha=\tau(H_\alpha\cap \mathcal T)\subset X_{\Delta}$ is the
 $(-1)$-curve on $X_{\Delta}$.
 \item[$\bullet$] The  $(-1)$-curves $E_{\omega_1}$ and $E_{\omega_2}$ do not intersect iff the weights
  $\omega_1$ and $\omega_2$ are not adjacent on the weight polytop.
\end{itemize}
\end{theorem}
\begin{proof}
Assume we constructed embedding of  the torsor $\mathcal T^{'}$ over
 $X_{\Delta^{'}}$ in the flag variety $G^{'}/P^{'}\subset \Bbb
P(V_1)$. Let $e_{8}\in X_\Delta$ be a point that we blow up to obtain $X_{\Delta}$
from $X_{\Delta^{'}}$. Let us choose a point $s\in
G^{'}/P^{'}$, such that $\tau(s)=e_{8}$. By theorem \ref{P-Q} for a sufficiently general
 $s^{'}\in G^{'}/P^{'}$ we have $s^{'}s^{-1}\mathcal
T^{'}\cap G^{'}/P^{'}=Ts^{'}$. Let us notice that by Proposition
\ref{blowupE_8} the preimage of  $s^{'}s^{-1}\mathcal T^{'}$ for
 the weighted blow up $Bl_{wt}(\Bbb P( V(\pi'_\beta)), G^{'}/P^{'})\longrightarrow
\Bbb P( V(\pi'_\beta))$ is isomorphic to the $T^{'}$-torsor over $X_\Delta$,
that we denote by  $\widetilde{\mathcal T}^{'}$.

By the Theorem  \ref{E_8factor} the quotient  $\lambda\backslash \!\!\backslash
(G/P)^s_{\varepsilon}\cap D_{h_
\alpha}\cap \mathcal U_\alpha $ is isomorphic to $
Bl_{wt}(\Bbb P(V(\pi'_\beta)),
G^{'}/P^{'})$.  Let us define a required torsor taking the preimage of $\widetilde{\mathcal T}^{'}$ with
respect to the quotient by  $\lambda$, and then taking affine cone over in
$V(\pi_\alpha)$ over it.

The arguments are completely analogues to the proof of the main Theorem
\ref{main}. Thus we give only the missing steps. Let us prove the first assertion of the theorem.

\begin{center}
  \epsfig{file=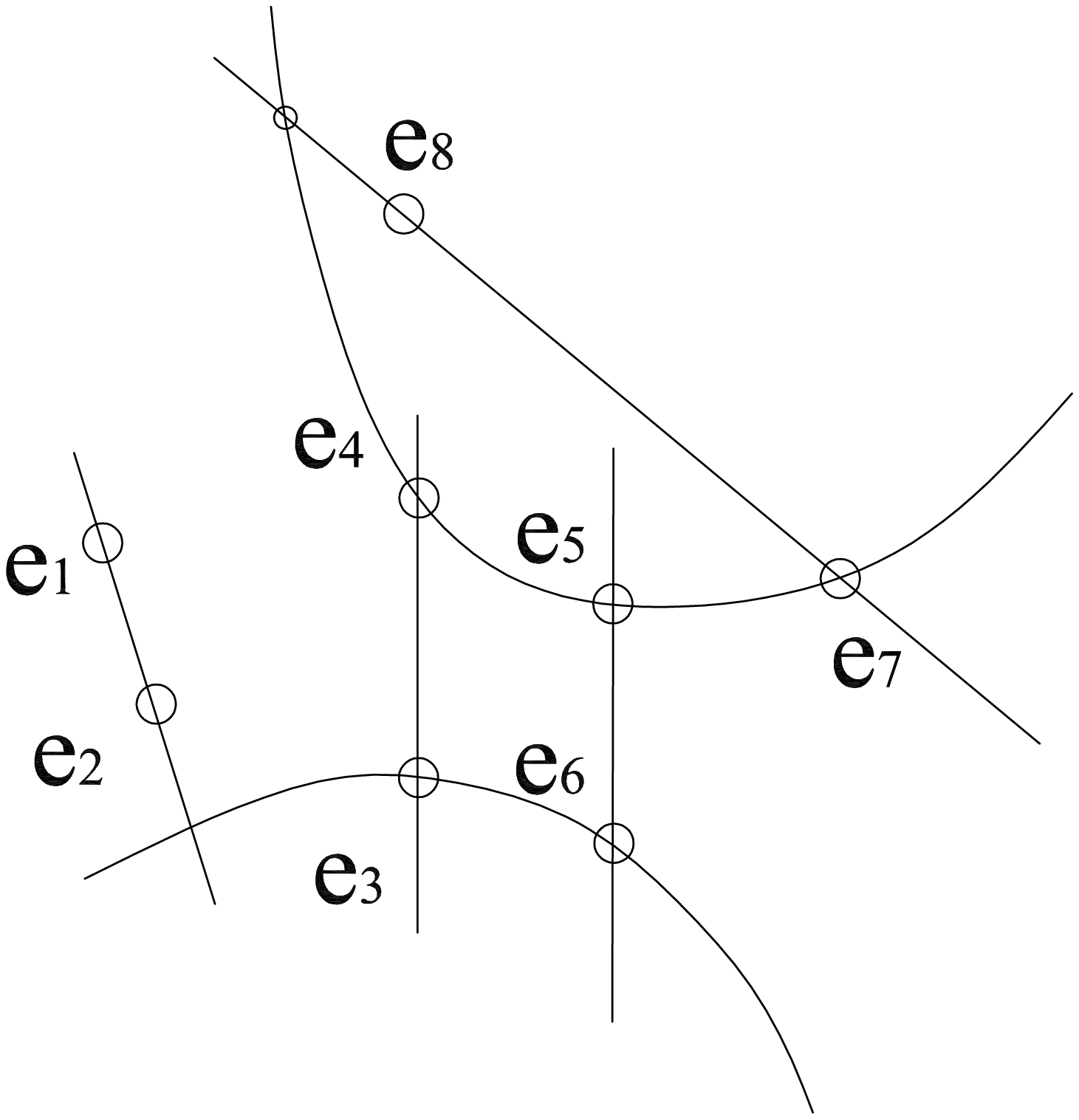,height=6cm,angle=90,clip=}

fig.7
\end{center}

Our aim is to show that for a sufficiently general point $e_8$
the homogeneous polynomial $q_4(s^{'}s^{-1}x)$ of weight $4$ is not identically zero when we restrict
it to the torsor for a sufficiently general point $s^{'}$
(for the equation of lower degrees we proved it in the Theorem
\ref{P-Q}).

Consider the exponential representation for the point $\langle
v\rangle \in P^-_uP/P$:
$$\langle
v\rangle =\langle\exp(\sum \limits_{\gamma \in
\Delta_{P_u}}x_\gamma e_{-\gamma})v_{\pi_\alpha}\rangle.$$

Let us recall that the equation  $q_4(x)$ of the degree $4$ from
$\{x_\gamma\}_{\gamma \in \Delta_{P_u}}$ is obtained by requiring that the coordinate by the
vector $v(-\pi_\alpha)$ is zero. This polynomial define on the surface $X_{\Delta^{'}}$ a curve of the degree four
 with respect to the pairing with $-K_{X_{\Delta^{'}}}$.

Let us argue as in the proof of proposition  \ref{P-Q}.
We assume that the point  $s^{'}$ belong to the open cell of
$G^{'}/P^{'}$. Thus we can represent it as the exponential map
$$s^{'}=\langle \exp(\sum_{\gamma\in \Delta_{\pp^{'}_u}} c_\gamma
e_{-\gamma})v(\pi_\alpha-\alpha)\rangle.$$ Substitute $s^{'}$ by this expression  in $q_4(s^{'}s^{-1}x)$.

And consider the coefficient  $R$ by the monomial
$c_{\delta_1}^2c^2_{\delta_2}c_{\delta_3}^2$ where
$\delta_1=L_{1}=L-E_1-E_2$, $\delta_2=L_{2}=L-E_3-E_4$,
$\delta_3=L_{3}=L-E_5-E_6$ (see fig.8). Our aim is to prove the non-triviality of this coefficient.

 The submonomials of this monomial occur in expressions of the linear systems generating the following curves:
$$C=3L-E_1-E_2-E_3-E_4-E_5-E_6-2E_7\  \text{corr. to the weight}\  \delta_1+\delta_2+\delta_3-2\alpha,$$
$$Q_1=2L-E_3-E_4-E_5-E_6-E_7,\  \text{corr. to the weight}\ \delta_2+\delta_3-\alpha,$$
$$Q_2=2L-E_1-E_2-E_5-E_6-E_7,\  \text{corr. to the weight}\  \delta_1+\delta_3-\alpha,$$
$$Q_3=2L-E_1-E_2-E_3-E_4-E_7,\  \text{corr. to the weight}\  \delta_1+\delta_2-\alpha.$$
 Using the notation of Proposition \ref{P-Q}, by $(\delta_1\delta_2\delta_3)$
  we denote the weight $\delta_1+\delta_2+\delta_3-2\alpha$, corresponding to the cubic
  $C$. In a similar way denote by   $(\delta_i\delta_j)$ the weight
   $\delta_i+\delta_j-\alpha$ corresponding to the conic $Q_k$ ($k\neq i,j$).

  The coefficient by the monomial 
$c_{\delta_1}^2c^2_{\delta_2}c_{\delta_3}^2$ we denote by $R$.
Let us notice that  $x^2_{(\delta_1\delta_2\delta_3)}$ (the last monomial defining $C^2$) is contained only in one monomial
 $\frac
{x^2_{(\delta_1\delta_2\delta_3)}x^2_\alpha}{s^2_{(\delta_1\delta_2\delta_3)}s^2_\alpha}$,
from $R$.

It is not difficult to write the full expression of the homogeneous polynomial $R$, however we won't do that.
 The first monomials from  $R$ ae of the form

 $r_{2C}\frac
{x^2_{(\delta_1\delta_2\delta_3)}x^2_\alpha}{s^2_{(\delta_1\delta_2\delta_3)}s^2_\alpha}+
r_{C,Q_1,L_1}\frac{x_{(\delta_1\delta_2\delta_3)}
x_{\delta_1}x_{(\delta_2\delta_3)}}{s_{(\delta_1\delta_2\delta_3)}s_{\delta_1}
s_{(\delta_2\delta_3)}}+r_{C,Q_2,L_2}\frac{x_{(\delta_1\delta_2\delta_3)}
x_{\delta_2}x_{(\delta_1\delta_3)}}{s_{(\delta_1\delta_2\delta_3)}s_{\delta_2}
s_{(\delta_1\delta_3)}}+$

$+r_{C,Q_3,L_3}\frac{x_{(\delta_1\delta_2\delta_3)}
x_{\delta_3}x_{(\delta_1\delta_2)}}{s_{(\delta_1\delta_2\delta_3)}s_{\delta_3}
s_{(\delta_1\delta_2)}}+\ldots,$

\noindent where $r_{\ldots}$ are the coefficients depending only from the choice
 of the basis in the $G$-module $V(\pi_\alpha)$.
Omitted monomials contain the degree of  
$x_{(\delta_1\delta_2\delta_3)}$ strictly smaller  than $2$.

Let us prove that the coefficient   $R$ by
$c_{\delta_1}c_{\delta_2}c_{\delta_3}$ is nonzero for a sufficiently general point $e_8$
 (we recall that the point $s\in \Bbb P(V_1)$ lie in the  $T'$-orbit that map into 
$e_8$ by the quotient morphism ).
Consider the restriction of  $R$  to the line $L_{78}=L-E_7-e_8$.
Let us fix this line and prove that for a general point of this line we have
$R\neq 0$. Let $l_i$ and $q_{j}$ are the intersection points of the line  $L_{78}$ with
 the line $L_{i}$ and with the conic $Q_j$ correspondingly, and $c$ --- is the intersection point of
$C$ and  $L_{78}$ different from $e_8$. As before let us notice that the values of all 
monomials from $R$ in the point $e_8$ are equal to the constants depending only from the choice basis in $V(\pi_\alpha)$.
  Let us choose the coordinate   $x$
on the line $L_{78}$ in such way that the point   $e_8$ has a coordinate  $0$.
Then the monomials after the restriction to $L_{78}$ can be written in the following way:
$$\frac
{x^2_{(\delta_1\delta_2\delta_3)}x^2_\alpha}{s^2_{(\delta_1\delta_2\delta_3)}s^2_\alpha}|_{L_{78}}
=\frac{1}{c^2e_7^4}(x-c)^2(x-e_7)^4$$
$$\frac{x_{(\delta_1\delta_2\delta_3)}
x_{\delta_1}x_{(\delta_2\delta_3)}}{s_{(\delta_1\delta_2\delta_3)}s_{\delta_1}
s_{(\delta_2\delta_3)}}|_{L_{78}}
=\frac{1}{cl_1q_1e_7^3}(x-c)(x-l_1)(x-q_1)(x-e_7)^3.$$
$$\ldots \ldots$$

We omitted other monomials since they have similar form. Let us choose the point $e_8$ on $L_{78}$ in a sufficiently small
neighbourhood of 
 $c$, i.e. $|c-e_8|<\varepsilon$. Then we have
$\min_{i,k}(|e_8-l_i|,|e_8-q_k|)>A$ and $|e_8-e_7|<B$ for some constants $A,B$. Also the coefficients 
$r_{\ldots}$ a bounded by a constant $r_{max}$.

The coefficient by $x^6$
in the expression for $R$ is bigger than $\frac{r_{2C}}{\varepsilon^2 B^4}-{C_1\frac{r_{max}}{\varepsilon
 A^5}}-
 {C_2\frac{r_{max}}{ A^6}}$ (where $C_1$  is the number of monomials in  $R$ containing $x_{(\delta_1\delta_2\delta_3)}$,
  and $C_2$ is the number of monomials in  $R$ not containing $x_{(\delta_1\delta_2\delta_3)}$),
 and the last expression is strictly  bigger than zero for sufficiently small  $\varepsilon$.
Thus $R$ as a function of $x$ is not identically zero for a sufficiently general point $e_8$. That proves the theorem.
\end{proof}

\section{Appendix}

 The Proposition \ref{intersection} is a particular case of the following proposition:

\begin{proposition} Let $H\subset G$ be some semisimple
subgroup normalized by the torus  $T$.
Let $V(\omega)$ be an irreducible 
$G$-module with the highest weight $\omega$. Consider the weight vector 
$v_{\chi} \in V(\omega)$ for which the weight $\chi$ is the vertex of the weight polytop of the representation
 $V(\omega)$,  in other words $\chi=w\omega$ for some $w\in W$.
Consider the module $V_{H}=\langle Hv_{\chi} \rangle$ generated by the vector $v_{\chi}$. Then we have an equality
$$G\langle v_{\omega}\rangle\cap \Bbb P(V_{H})=H\langle v_{\chi}\rangle.$$
 \end{proposition}
\begin{proof} Without the loss of generality we can assume that 
$\chi=\omega$.
Let us denote by $P_{\omega}$  the stabilizer in   $G$ of the line
$\langle v_{\omega} \rangle$.  Since 
$v_{\omega}$ is the highest weight vector for the group $G$, it is highest 
weight vector for $H$. Thus  $V_{H}$ is the irreducible
$H$-module. Let us notice that we can identify  $\supp_H(V_H)$
with the subset in  $\supp (V) \cap (\omega-\sum_{\gamma \in \Delta_H}\Bbb Z_+\gamma)$.
From the inclusion  $\omega \in \supp (V_H)$ for a dominant weight $\omega$,
  it follows that the vertices of the weight polytop of  $V$, that belong to
  $\supp_H(V_H)$ are the vertices of the weight polytop  $V_{H}$.

 Let $\langle v\rangle\in
G/P_{\omega}\cap \Bbb P(V_{H})$. Let us prove that the vector  $v$ 
has a nonzero component of weight
 $w\omega$ for $w\in W_H$
(we recall that $w\omega$ being a vertex of the weight polytop has multiplicity one).
 Indeed $\langle v\rangle\in
\bigcup\limits_{w\in W}B^-wP_\omega/P_\omega \cap \Bbb P(V_{H})$, in particular $\langle v\rangle\in B^-\widetilde{w}P_\omega/P_\omega
\cap \Bbb P(V_{H})$ for some $\widetilde{w} \in W$. By Lemma
\ref{BGG} we get that the component  $v_{\widetilde{w}\omega}$
of vector $v$ is not equal to zero. Since $\langle
v_{\widetilde{w}\omega}\rangle \in \Bbb P(V_{H})$ we can assume that the element 
 $\widetilde{w} \in W$ belong to  $W_H$. Applying the element
$n_{\widetilde{w}}^{-1}$ to the vector $v$ we can assume that its component  $v_{\omega}$ is not equal to zero. Then from Lemma
  \ref{BGG} we get that $\langle v\rangle$ belong to the open cell
  $P_u^{-}P/P$ in $G/P$. Thus $\langle v\rangle=\langle uv_{\omega}\rangle$
  for some $u\in P_u^{-}$. Using the presentation of $u$ as the exponential  map we get:
$$ \langle v\rangle=\langle \exp(\sum \limits_{\gamma \in \Delta_{\pp_u}}c_\gamma
e_{-\gamma})v_{\omega}\rangle.$$ Assume that there exists
$\gamma\notin \Delta_H^+$ such that $c_\gamma\neq 0$. From such roots let us choose the root  $\gamma_0$ that cannot be obtained as the positive integer
linear combination of the roots  $\mu\in \Delta_{\pp_u}$ for which $c_{\mu}\neq 0$.
 Then the component of vector  $v$ of weight $\omega-\gamma_0$,
is equal to $v_{\omega-\gamma_0}=c_{\gamma_0}e_{-\gamma_0}v_{\omega}$.
Let us show that this component is not equal to zero. The weight $\omega$ is dominant
and $\langle\omega;\gamma_0\rangle>0$ (the last equality follows from the fact that $P$ is the stabilizer of the line
spanned by the highest weight vector with the weight $\omega$ and the fact that $\gamma_0\in
\Delta_{\pp_u}$). This implies that it cannot be the lowest weight of the representation of the three dimensional 
subalgebra  generated by the triple
$\{e_{-\gamma_0},h_{\gamma_0},e_{\gamma_0}\}$ thus we get
$e_{-\gamma_0}v_{\omega}\neq 0$.

On the other hand the component
$v_{\omega-\gamma_0}$  cannot be nonzero since  $v\in
V_H$ and $\supp(V_H)\subset \omega-\sum_{\gamma\in \Delta^+_H}\Bbb
Z_{+}\gamma$. Thus we come to the contradiction with the existence of 
 $\gamma_0\notin \Delta_H^+$ such that $c_{\gamma_0}\neq 0$; this implies the inclusion $u\in
P^{-}_u\cap H$. The proof finishes similar to the proof of Proposition  \ref{intersection}.\end{proof}

\end{document}